\documentclass[a4paper,11pt,reqno]{amsart}
\usepackage[utf8]{inputenc}
\usepackage{amsmath}
\usepackage{amssymb}
\usepackage{amsthm}
\usepackage{mathtools}
\usepackage{hyperref}
\usepackage{color}
\usepackage{enumitem}
\usepackage{a4wide}
\usepackage[all]{xy}

\DeclareMathOperator{\GL}{GL}
\DeclareMathOperator{\gl}{\mathfrak{gl}}
\DeclareMathOperator{\SL}{SL}
\let\sl\relax
\DeclareMathOperator{\sl}{\mathfrak{sl}}
\DeclareMathOperator{\upO}{O}
\DeclareMathOperator{\upS}{S}
\DeclareMathOperator{\SO}{SO}
\DeclareMathOperator{\so}{\mathfrak{so}}
\DeclareMathOperator{\Sp}{Sp}
\let\sp\relax
\DeclareMathOperator{\sp}{\mathfrak{sp}}

\DeclareMathOperator{\su}{\mathfrak{su}}

\DeclareMathOperator{\spin}{\mathfrak{spin}}
\DeclareMathOperator{\upU}{U}

\newcommand{\fraka}{\mathfrak{a}}

\newcommand{\frake}{\mathfrak{e}}
\newcommand{\frakf}{\mathfrak{f}}
\newcommand{\frakg}{\mathfrak{g}}
\newcommand{\frakh}{\mathfrak{h}}
\newcommand{\frakj}{\mathfrak{j}}
\newcommand{\frakk}{\mathfrak{k}}
\newcommand{\frakl}{\mathfrak{l}}
\newcommand{\frakm}{\mathfrak{m}}
\newcommand{\frakn}{\mathfrak{n}}
\newcommand{\frakp}{\mathfrak{p}}
\newcommand{\frakq}{\mathfrak{q}}
\newcommand{\fraks}{\mathfrak{s}}
\newcommand{\frakt}{\mathfrak{t}}
\newcommand{\fraku}{\mathfrak{u}}
\newcommand{\frakz}{\mathfrak{z}}

\renewcommand{\AA}{\mathbb{A}}

\newcommand{\CC}{\mathbb{C}}

\newcommand{\FF}{\mathbb{F}}
\newcommand{\GG}{\mathbb{G}}
\newcommand{\HH}{\mathbb{H}}
\newcommand{\KK}{\mathbb{K}}
\newcommand{\MM}{\mathbb{M}}
\newcommand{\NN}{\mathbb{N}}

\newcommand{\PP}{\mathbb{P}}
\newcommand{\RR}{\mathbb{R}}
\newcommand{\ZZ}{\mathbb{Z}}

\newcommand{\calD}{\mathcal{D}}
\newcommand{\calE}{\mathcal{E}}

\newcommand{\calV}{\mathcal{V}}

\newcommand{\0}{\textbf{0}}
\renewcommand{\1}{{\rm\bf 1}}

\DeclareMathOperator{\Ind}{Ind}
\DeclareMathOperator{\tr}{tr}
\DeclareMathOperator{\ad}{ad}
\DeclareMathOperator{\Ad}{Ad}

\DeclareMathOperator{\Hom}{Hom}

\DeclareMathOperator{\id}{id}
\DeclareMathOperator{\sgn}{sgn}

\DeclareMathOperator{\diag}{diag}
\DeclareMathOperator{\blank}{\hspace{.05em}\cdot\hspace{.05em}}

\DeclareMathOperator{\supp}{supp}
\DeclareMathOperator{\linspan}{span}
\DeclareMathOperator{\rank}{rank}
\DeclareMathOperator{\Sh}{Sh}
\renewcommand\Re{\operatorname{Re}}

\renewcommand{\min}{{\textup{min}}}

\theoremstyle{plain}
\newtheorem{theorem}{Theorem}[section]
\newtheorem*{theorem*}{Theorem}
\newtheorem{proposition}[theorem]{Proposition}
\newtheorem{lemma}[theorem]{Lemma}
\newtheorem{corollary}[theorem]{Corollary}
\newtheorem{conjecture}[theorem]{Conjecture}
\newtheorem{fact}[theorem]{Fact}
\newtheorem*{fact*}{Fact}

\newtheorem{thmalph}{Theorem}
\newtheorem{coralph}[thmalph]{Corollary}

\newtheorem{factrom}{Fact}

\theoremstyle{definition}
\newtheorem{definition}[theorem]{Definition}
\newtheorem{example}[theorem]{Example}

\newtheorem{remark}[theorem]{Remark}

\numberwithin{equation}{section}

\title[Symmetry breaking operators for strongly spherical reductive pairs]{Symmetry breaking operators for strongly spherical reductive pairs}

\author{Jan Frahm}
\address{Department of Mathematics, Aarhus University, Ny Munkegade 118, 8000 Aarhus C, Denmark}
\email{frahm@math.au.dk}

\begin{document}



\maketitle

\begin{abstract}
A real reductive pair $(G,H)$ is called strongly spherical if the homogeneous space $(G\times H)/\diag(H)$ is real spherical. This geometric condition is equivalent to the representation theoretic property that $\dim\Hom_H(\pi|_H,\tau)<\infty$ for all smooth admissible representations $\pi$ of $G$ and $\tau$ of $H$. In this paper we explicitly construct for all strongly spherical pairs $(G,H)$ intertwining operators in $\Hom_H(\pi|_H,\tau)$ for $\pi$ and $\tau$ spherical principal series representations of $G$ and $H$. These so-called \textit{symmetry breaking operators} depend holomorphically on the induction parameters and we further show that they generically span the space $\Hom_H(\pi|_H,\tau)$. In the special case of multiplicity one pairs we extend our construction to vector-valued principal series representations and obtain generic formulas for the multiplicities between arbitrary principal series.\\
As an application, we prove an early version of the Gross--Prasad conjecture for complex orthogonal groups, and also provide lower bounds for the dimension of the space of Shintani functions.
\end{abstract}

\setcounter{tocdepth}{1}
\tableofcontents

\section*{Introduction}

One major question in the representation theory of real reductive groups is how an irreducible representation of a group $G$ decomposes if restricted to a subgroup $H$. In the context of infinite-dimensional representations of non-compact Lie groups this leads to the study of the multiplicities
$$ \dim\Hom_H(\pi|_H,\tau) \in \NN\cup\{\infty\}, $$
where $\pi$ and $\tau$ are irreducible representations of $G$ and $H$, usually assumed to be smooth admissible Fréchet representations of moderate growth. In general these multiplicities might be infinite, so to find a good setting for studying them one is interested in pairs $(G,H)$ of real reductive groups where $\dim\Hom_H(\pi|_H,\tau)$ is always finite. Elements of $\Hom_H(\pi|_H,\tau)$ are also referred to as \textit{symmetry breaking operators}, a term coined by Kobayashi~\cite{Kob15}.

Following \cite{KKPS17} we call a pair $(G,H)$ consisting of a real reductive group $G$ and a reductive subgroup $H$ \textit{strongly spherical} provided the homogeneous space $(G\times H)/\diag(H)$ is real spherical, i.e. a minimal parabolic subgroup $P_G\times P_H$ of $G\times H$ has an open orbit. We note that this is equivalent to the double coset space $P_H\backslash G/P_G$ being finite. The interest in strongly spherical pairs in the context of representation theory is due to the following result by Kobayashi--Oshima~\cite{KO13}:

\begin{factrom}[{see \cite[Theorem C]{KO13}}]\label{fact:FiniteMultiplicities}
If $(G,H)$ is strongly spherical then
$$ \dim\Hom_H(\pi|_H,\tau)<\infty $$
for all smooth admissible representations $\pi$ of $G$ and $\tau$ of $H$. If additionally $G$ and $H$ are defined algebraically over $\RR$, then also the converse statement holds.
\end{factrom}

Among the strongly spherical pairs, Kobayashi--Oshima~\cite{KO13} also characterized those with uniformly bounded multiplicities. The corresponding pairs of Lie algebras essentially form five families, and choosing the right Lie groups yields five families of groups $(G,H)$ whose multiplicities are uniformly bounded by one. This multiplicity one property is due to Sun--Zhu~\cite{SZ12}:

\begin{factrom}[{see \cite[Theorem B]{SZ12}}]\label{fact:MultiplicityOne}
If $(G,H)$ is one of the pairs
\begin{equation*}
\begin{gathered}
 (\GL(n+1,\CC),\GL(n,\CC)), \quad (\GL(n+1,\RR),\GL(n,\RR)), \quad (\upU(p,q+1),\upU(p,q)),\\
 (\SO(n+1,\CC),\SO(n,\CC)), \quad (\SO(p,q+1),\SO(p,q)),
\end{gathered}
\end{equation*}
then
$$ \dim\Hom_H(\pi|_H,\tau) \leq 1 $$
for all irreducible smooth admissible Fr\'{e}chet representations $\pi$ of $G$ and $\tau$ of $H$ of moderate growth.
\end{factrom}

Both Facts~\ref{fact:FiniteMultiplicities} and \ref{fact:MultiplicityOne} lead to the natural problem of determining $\dim\Hom_H(\pi|_H,\tau)$ for given irreducible representations $\pi$ and $\tau$ of $G$ and $H$, as advocated by Kobayashi~\cite{Kob15} in his ABC program as Stages B and C. For the multiplicity one pairs of general linear groups, this question is linked to the famous Rankin--Selberg integrals, and for the pairs of orthogonal and unitary groups the Gross--Prasad and Gan--Gross--Prasad conjectures make predictions about when the multiplicities are non-zero.

Since every irreducible smooth admissible Fréchet representation of moderate growth is by the Harish-Chandra Subquotient Theorem a subquotient of a principal series representation, it is reasonable to study this problem for principal series representations $\pi$ and $\tau$. This has previously been done for special cases by Kobayashi--Speh~\cite{KS15,KS18} and Clerc~\cite{Cle16a,Cle16b}. In this paper, we construct for all strongly spherical reductive pairs $(G,H)$ explicit families of symmetry breaking operators between principal series representations that depend holomorphically on the principal series parameters. Our key new ingredient is the systematic use of matrix coefficients of finite-dimensional representations of $G$ and $H$ in the construction of the corresponding distribution kernels, and it provides a conceptual way to determine these kernels explicitly. This establishes lower bounds for the multiplicities in question and provides a first step to the full classification. Together with upper bounds (obtained using Bruhat's theory of invariant distributions) we are able to compute $\dim\Hom_H(\pi|_H,\tau)$, at least generically, for all strongly spherical pairs $(G,H)$ and spherical principal series representations, and for the multiplicity one pairs also for non-spherical principal series representations. In the case $(G,H)=(\SO(n+1,\CC),\SO(n,\CC))$ this proves the local Gross--Prasad conjecture at the Archimedean place $k=\CC$ as initially stated in \cite{GP92}.

\subsection*{Statement of the main results}

Let $G$ be a real reductive group in the Harish-Chandra class and $H\subseteq G$ be a reductive subgroup such that the pair $(G,H)$ is strongly spherical. Let $\frakg$ and $\frakh$ denote the corresponding Lie algebras and assume that the pair $(\frakg,\frakh)$ is indecomposable, i.e. it cannot be written as the direct sum of two non-trivial pairs of reductive Lie algebras. A classification of all such pairs $(\frakg,\frakh)$ (under a slightly stronger indecomposability assumption) was established by Kobayashi--Matsuki~\cite{KM14} and Knop--Kr\"{o}tz--Pecher--Schlichtkrull~\cite{KKPS16,KKPS17}, and we summarize the result in Theorem~\ref{thm:ClassificationStronglySphericalPairs}. We further assume that $(\frakg,\frakh)$ is non-trivial, i.e. $\frakg\neq\frakh$. (In fact, our results do not hold in the case $\frakg=\frakh$ since intertwining operators between principal series of a real reductive group $G$ can only exist if the induction parameters are related by an element of the Weyl group.)

Let $P_G=M_GA_GN_G\subseteq G$ and $P_H=M_HA_HN_H\subseteq H$ be minimal parabolic subgroups and write $\fraka_G$ and $\fraka_H$ for the Lie algebras of $A_G$ and $A_H$. For irreducible finite-dimensional representations $\xi$ of $M_G$, $\eta$ of $M_H$ and $\lambda\in\fraka_{G,\CC}^\vee$, $\nu\in\fraka_{H,\CC}^\vee$ we consider the principal series representations (smooth normalized parabolic induction)
$$ \pi_{\xi,\lambda} = \Ind_{P_G}^G(\xi\otimes e^\lambda\otimes\1), \qquad \tau_{\eta,\nu} = \Ind_{P_H}^H(\eta\otimes e^\nu\otimes\1). $$
In case $\xi=\1$ and $\eta=\1$ are the trivial representations of $M_G$ and $M_H$, we abbreviate $\pi_\lambda=\pi_{\1,\lambda}$ and $\tau_\nu=\tau_{\1,\nu}$. Our first main result relates $\Hom_H(\pi_\lambda|_H,\tau_\nu)$ to $(P_H\backslash G/P_G)_{\rm open}$, the set of open double cosets in $P_H\backslash G/P_G$.

\begin{thmalph}[{see Corollary~\ref{cor:LowerBoundsMultiplicities} and \ref{cor:GenericMultiplicities}}]\label{intro:MainThm}
Assume that $(G,H)$ is a strongly spherical reductive pair such that $(\frakg,\frakh)$ is non-trivial and indecomposable. Then for all $(\lambda,\nu)\in\fraka_{G,\CC}^\vee\times\fraka_{H,\CC}^\vee$ we have the lower multiplicity bound
$$ \dim\Hom_H(\pi_\lambda|_H,\tau_\nu) \geq \#(P_H\backslash G/P_G)_{\rm open}, $$
and for generic $(\lambda,\nu)\in\fraka_{G,\CC}^\vee\times\fraka_{H,\CC}^\vee$ (see \eqref{eq:GenericCondition} for the precise condition) we have
$$ \dim\Hom_H(\pi_\lambda|_H,\tau_\nu) = \#(P_H\backslash G/P_G)_{\rm open}. $$
\end{thmalph}

Let us briefly explain the method of proof. First of all, the generic upper multiplicity bound $\leq\#(P_H\backslash G/P_G)_{\rm open}$ is established by identifying intertwining operators with their distribution kernels which are certain $(P_H\times P_G)$-invariant distributions on $G$ (see Section~\ref{sec:InvariantDistrbutionsSub}). To bound the dimension of the space of invariant distributions we use Bruhat's theory. Here, the non-open $(P_H\times P_G)$-orbits in $G$ are particularly important, and we prove a new structural result about them (see Theorem~\ref{thm:AHProjectionStabilizers}) which is the necessary technical ingredient in the proof of generic upper bounds (see Theorem~\ref{thm:MultiplicityBounds}). The lower multiplicity bounds are established by explicitly constructing a non-trivial holomorphic family $A_{\lambda,\nu}\in\Hom_H(\pi_\lambda|_H,\tau_\nu)$ of symmetry breaking operators for every open double coset in $P_H\backslash G/P_G$. The construction of this family of operators is in terms of their distribution kernels which turn out to be products of complex powers of matrix coefficients belonging to finite-dimensional spherical representations of $G$. The technical ingredients in this part are the existence of \textit{enough} such matrix coefficients (see Theorem~\ref{thm:EnoughWeights}) and the meromorphic/holomorphic extension of the complex powers (see Theorem~\ref{thm:MeromorphicContinuation}). Then the lower bounds are established by regularizing the families $A_{\lambda,\nu}$ in the holomorphic parameters $(\lambda,\nu)$.

For the multiplicity one pairs in Fact~\ref{fact:MultiplicityOne} we also consider non-spherical principal series representations and show the following result:

\begin{thmalph}[{see Theorem~\ref{thm:MultiplicitiesMultOnePairs}}]\label{intro:ThmMultOne}
Assume that $(G,H)$ is a multiplicity one pair. Then for all $(\xi,\eta)\in\widehat{M}_G\times\widehat{M}_H$ and $(\lambda,\nu)\in\fraka_{G,\CC}^\vee\times\fraka_{H,\CC}^\vee$ we have the lower multiplicity bound
$$ \dim\Hom_H(\pi_{\xi,\lambda}|_H,\tau_{\eta,\nu}) \geq 1 \qquad \mbox{whenever }\Hom_M(\xi|_M,\eta|_M)\neq\{0\}, $$
and for generic $(\lambda,\nu)\in\fraka_{G,\CC}^\vee\times\fraka_{H,\CC}^\vee$ we have
$$ \dim\Hom_H(\pi_{\xi,\lambda}|_H,\tau_{\eta,\nu}) = \begin{cases}1&\mbox{for $\Hom_M(\xi|_M,\eta|_M)\neq\{0\}$,}\\0&\mbox{for $\Hom_M(\xi|_M,\eta|_M)=\{0\}$.}\end{cases} $$
Here $M\subseteq M_G\cap M_H$ denotes the stabilizer of the unique open $P_H$-orbit in $G/P_G$.
\end{thmalph}

The passage from spherical principal series $\pi_\lambda$ as treated in Theorem~\ref{intro:MainThm} to general principal series $\pi_{\xi,\lambda}$ uses a variant of the Jantzen--Zuckerman \textit{translation principle}. Here both $\pi_\lambda$ and $\tau_\nu$ are tensored with certain finite-dimensional representations of $G$ and $H$, and we show the existence of such representations case-by-case (see Theorem~\ref{prop:ExGRepsForMultOnePairs}). We illustrate the construction of symmetry breaking operators for the multiplicity one pair $(G,H)=(\GL(n+1,\RR),\GL(n,\RR))$ by providing explicit formulas for the distribution kernels (see Section~\ref{sec:ExGLn}).

We remark that Theorems~\ref{intro:MainThm} and \ref{intro:ThmMultOne} leave open the question of determining the multiplicities $\dim\Hom_H(\pi_{\xi,\lambda}|_H,\tau_{\eta,\nu})$ for all parameters. In general, this question is much more involved and has so far only been solved in special cases (see Clerc~\cite{Cle16a,Cle16b} and Kobayashi--Speh~\cite{KS15}). In these cases, all operators in $\Hom_H(\pi_{\xi,\lambda}|_H,\tau_{\eta,\nu})$ arise from a holomorphic family of operators, so that our explicit construction of meromorphic families provides a first major step for the full classification for general strongly spherical reductive pairs. We hope to return to this point later.

\subsection*{Applications}

Let us present two interesting applications of the main results. Combining Theorem~\ref{intro:ThmMultOne} with Fact~\ref{fact:MultiplicityOne} we immediately obtain the following corollary:

\begin{coralph}[{see Corollary~\ref{cor:GPConjecture}}]
Assume that $(G,H)$ is a multiplicity one pair, where we additionally assume $p=q$ or $p=q+1$ in the case of indefinite orthogonal or unitary groups. Then, if both $\pi_{\xi,\lambda}$ and $\tau_{\eta,\nu}$ are irreducible we have
$$ \dim\Hom_H(\pi_{\xi,\lambda}|_H,\tau_{\eta,\nu}) = 1. $$
\end{coralph}

For $(G,H)=(\SO(n+1,\CC),\SO(n,\CC))$ this proves the local Gross--Prasad conjecture~\cite[Conjecture 11.5]{GP92} at the Archimedean place $k=\CC$ (see Conjecture~\ref{conj:GrossPrasadComplex}). We remark that the Gross--Prasad conjecture has been extended by several people to the case of reducible principal series, for which we cannot make precise statements in the generality discussed in this paper. We expect that our results also provide some information towards the local Gross--Prasad conjecture at the Archimedean place $k=\RR$ where one considers the pairs $(G,H)=(\SO(p,q+1),\SO(p,q))$. However, for real groups it will be necessary to also consider principal series representations induced from more general cuspidal parabolic subgroups. We hope to return to this topic in a future work.

Another application concerns the study of Shintani functions for the pair $(G,H)$ which was recently taken up by Kobayashi~\cite{Kob14}. For $(\lambda,\nu)\in\fraka_{G,\CC}^\vee\times\fraka_{H,\CC}^\vee$ we write $\Sh(\lambda,\nu)$ for the space of Shintani functions for $(G,H)$ of type $(\lambda,\nu)$ and $\Sh_{\rm mod}(\lambda,\nu)$ for its subspace of Shintani functions of moderate growth (see Section~\ref{sec:ShintaniFunctions} for the precise definitions). Kobayashi~\cite{Kob14} showed that $\dim\Sh(\lambda,\nu)<\infty$ for all $(\lambda,\nu)\in\fraka_{G,\CC}^\vee\times\fraka_{H,\CC}^\vee$ if and only if $(G,H)$ is strongly spherical. Combining results from \cite{Kob14} with Theorem~\ref{intro:MainThm} shows the following corollary:

\begin{coralph}[{see Theorem~\ref{thm:LowerBoundsShintaniFunctions}}]
Assume that $(G,H)$ is a strongly spherical reductive pair such that $(\frakg,\frakh)$ is non-trivial and indecomposable. Then for all $(\lambda,\nu)\in\fraka_{G,\CC}^\vee\times\fraka_{H,\CC}^\vee$ we have
$$ \dim\Sh(\lambda,\nu) \geq \dim\Sh_{\rm mod}(\lambda,\nu) \geq \#(P_H\backslash G/P_G)_{\rm open}, $$
and for generic $(\lambda,\nu)\in\fraka_{G,\CC}^\vee\times\fraka_{H,\CC}^\vee$ we have
$$ \dim\Sh_{\rm mod}(\lambda,\nu) = \#(P_H\backslash G/P_G)_{\rm open}. $$
\end{coralph}

\subsection*{Relation to other work}

Let us first mention previous works in which holomorphic families of intertwining operators $A_{\lambda,\nu}\in\Hom_H(\pi_\lambda|_H,\tau_\nu)$ appear. The construction of operators $A_{\lambda,\nu}$ for the pairs $(\frakg,\frakh)=(\so(1,n)+\so(1,n),\diag\so(1,n))$ can be found in the work of Oksak~\cite{Oks73} for $n=3$ (the case of $\so(1,3)\simeq\sl(2,\CC)$), Bernstein--Reznikov~\cite{BR04} for $n=2$ (the case of $\so(1,2)\simeq\sl(2,\RR)$) and Clerc--Kobayashi--{\O}rsted--Pevzner~\cite{CKOP11} for arbitrary $n\geq2$. Later Clerc~\cite{Cle16a,Cle16b} gave a complete description of the space $\Hom_H(\pi_\lambda|_H,\tau_\nu)$ for all parameters $(\lambda,\nu)$. For $(\frakg,\frakh)=(\so(1,n+1),\so(1,n))$ Kobayashi--Speh~\cite{KS15} obtained a full classification of all symmetry breaking operators in terms of the holomorphic family $A_{\lambda,\nu}$. In fact, some of the analytic arguments we use can also be found in \cite{KS15,KS18} (see e.g. \cite[Lemma 11.10]{KS15} which deduces lower multiplicity bounds from meromorphic families of intertwining operators). For some symmetric pairs of low rank, in particular for all symmetric pairs $(\frakg,\frakh)$ with $\frakg$ of rank one, the work of M\"{o}llers--{\O}rsted--Oshima~\cite{MOO16} yields holomorphic families of symmetry breaking operators. We also note that for $(\frakg,\frakh)=(\gl(n+1,\RR),\gl(n,\RR))$ the kernel functions given in Section~\ref{sec:ExGLn} can be found in the work of Murase--Sugano~\cite{MS96} in the context of $p$-adic groups, and in a slightly different form also in the recent work of Neretin~\cite{Ner15} in the context of finite-dimensional representations. The conceptual construction via finite-dimensional matrix-coefficients that we present in this paper seems to be new, and generalizes all previous constructions to the setting of strongly spherical pairs.

We remark that some of our statements can also be proven differently using the recent work of Gourevitch--Sahi--Sayag~\cite{GSS16} on the extension of invariant distributions. In fact, their work uses a similar idea, namely the extension of an invariant distribution to a meromorphic family of distributions. However, their meromorphic families only depend on one complex parameter whereas our constructed families depend on $(\lambda,\nu)$ and hence contain more information. Moreover, we provide a method to explicitly determine the meromorphic families in terms of matrix coefficients, while in \cite{GSS16} the construction of the invariant distributions is more indirect.

Although our holomorphic families of symmetry breaking operators generically span the space of all intertwining operators, it is much more difficult to determine $\Hom_H(\pi_{\xi,\lambda}|_H,\tau_{\eta,\nu})$ for singular parameters $(\lambda,\nu)\in\fraka_{G,\CC}^\vee\times\fraka_{H,\CC}^\vee$. Theorems~\ref{intro:MainThm} and \ref{intro:ThmMultOne} provide lower bounds for the multiplicities, but it turns out that for singular parameters the multiplicities can be larger. A systematic study of multiplicities and symmetry breaking operators for $(\frakg,\frakh)=(\so(1,n+1),\so(1,n))$ was initiated by Kobayashi, and we refer the reader to the relevant articles by Kobayashi--Speh~\cite{KS15,KS18} and Kobayashi--Kubo--Pevzner~\cite{KKP16} (see also the work of Fischmann--Juhl--Somberg~\cite{FJS16}). We expect that our holomorphic families play a major role in the full classification of symmetry breaking operators as is the case for $(\frakg,\frakh)=(\so(1,n+1),\so(1,n))$ (see \cite{KS15}), and therefore view our general construction as a first step into this direction for the class of strongly spherical reductive pairs.

\subsection*{Outlook}

The principal series representations considered in this paper are all induced from a minimal parabolic subgroup. However, by the Langlands classification every smooth admissible Fréchet representation of moderate growth is the unique irreducible quotient of a generalized principal series representation, induced from an arbitrary cuspidal parabolic subgroup. At least for the multiplicity one pairs a generalization of our construction to cuspidal parabolic subgroups is desirable.

Our holomorphic families of symmetry breaking operators also allow an interpretation as invariant distribution vectors. More precisely, one can view the distribution kernels as $\diag(H)$-invariant distribution vectors on principal series representations of $G\times H$. As such, they are expected to contribute to the most continuous part of the Plancherel formula for the real spherical homogeneous spaces $(G\times H)/\diag(H)$. It would be interesting to investigate this topic further, especially in connection with the recent advances in the harmonic analysis on real spherical spaces (see e.g. \cite{KS16b} and references therein).

Another possible application of symmetry breaking operators is the explicit construction of branching laws for unitary representations. This was successfully carried out for the pair $(G,H)=(\upO(1,n+1),\upO(1,m+1)\times\upO(n-m))$ in the case of unitary principal series and complementary series representations by M\"{o}llers--Oshima~\cite{MO15}. For $(G,H)=(\GL(n+1,\CC),\GL(n,\CC))$ a similar suggestion was made by Neretin~\cite[Section 5]{Ner15}.

Let us also mention connections to boundary value problems~\cite{MOZ16} and automorphic forms~\cite{BR04,MO14} that were established for $(G,H)=(\upO(1,n+1),\upO(1,n))$ and might be of interest also for more general strongly spherical pairs.

\subsection*{Structure of the paper}

In Section~\ref{sec:StructureTheory} we recall some structure theory and the classification of strongly spherical real reductive pairs. The main new result here is a characterization of the open double cosets in $P_H\backslash G/P_G$ (see Theorem~\ref{thm:AHProjectionStabilizers}). The construction of $(P_H\times P_G)$-equivariant matrix coefficients on $G$ is the content of Section~\ref{sec:FiniteDimBranching}, and Theorem~\ref{thm:EnoughWeights} ensures the existence of enough such matrix coefficients. Section~\ref{sec:ConstructionOfSBOs} deals with the explicit construction of symmetry breaking operators between spherical principal series representations of strongly spherical pairs (see Theorem~\ref{thm:MeromorphicContinuation}). This construction uses the results of Sections~\ref{sec:StructureTheory} and \ref{sec:FiniteDimBranching} in a crucial way, and implies the claimed lower bounds for multiplicities. The upper bounds are established in Section~\ref{sec:InvariantDistributions} using Bruhat's theory of invariant distributions (see Theorem~\ref{thm:MultiplicityBounds}). The application of this technique depends heavily on the results of Section~\ref{sec:StructureTheory}. In Section~\ref{sec:ShintaniFunctions} the previous results are applied to obtain bounds for the space of Shintani functions (see Theorem~\ref{thm:LowerBoundsShintaniFunctions}), following Kobayashi's recent approach via symmetry breaking operators. The topic of Section~\ref{sec:MultiplicityOnePairs} is the construction of symmetry breaking operators between general principal series from symmetry breaking operators between spherical principal series for multiplicity one pairs. To prove the main statement Theorem~\ref{thm:MultiplicitiesMultOnePairs} we employ a variant of the translation principle which we apply case-by-case to all multiplicity one pairs. Finally, Section~\ref{sec:ExGLn} illustrates symmetry breaking operators between principal series for the pair $(G,H)=(\GL(n+1,\RR),\GL(n,\RR))$ by explicit formulas.

\subsection*{Acknowledgements} We thank Yoshiki Oshima for helpful and inspiring conversations on the topic of this paper. We further thank an anonymous referee for pointing out a classification-free proof of Theorem~\ref{thm:EnoughWeights} using the local structure theorem for real spherical varieties.

\subsection*{Notation} $\NN=\{0,1,2,\ldots,\}$, $V^\vee=\Hom_\CC(V,\CC)$.

\section{The structure of strongly spherical reductive pairs}\label{sec:StructureTheory}

We discuss strongly spherical reductive pairs $(G,H)$ and their structure theory following \cite{KKPS17,KM14}. First, using results from \cite{KKPS17}, we reduce the study of general strongly spherical reductive pairs to that of strongly spherical symmetric pairs. For the latter we recall some structure theory as developed in \cite[Section 3]{KM14}. This is used to derive some new results about the double coset space $P_H\backslash G/P_G$ for $P_G\subseteq G$ and $P_H\subseteq H$ minimal parabolic subgroups (see Theorem~\ref{thm:AHProjectionStabilizers}). These results are used both in Section~\ref{sec:ConstructionOfSBOs} for the construction of symmetry breaking operators and in Section~\ref{sec:InvariantDistributions} for their uniqueness.

We remark that for complex groups the double coset space $P_H\backslash G/P_G$ was studied in \cite{HNOO13} with similar techniques.

\subsection{Strongly spherical reductive pairs}

Consider a real reductive pair $(\frakg,\frakh)$ of Lie algebras, i.e. $\frakg$ is a reductive Lie algebra and $\frakh$ a reductive subalgebra of $\frakg$. A pair of Lie groups $(G,H)$ with $H$ a closed subgroup of $G$ is called real reductive if the underlying pair $(\frakg,\frakh)$ of Lie algebras is real reductive. In this paper we will additionally assume that $G$ is of \textit{Harish-Chandra class} (see e.g. \cite[Chapter VII.2]{Kna02} for the precise definition).

\begin{definition}\label{def:StronglySpherical}
A real reductive pair $(\frakg,\frakh)$ of Lie algebras is called \textit{strongly spherical} if there exist minimal parabolic subalgebras $\frakp_G\subseteq\frakg$ and $\frakp_H\subseteq\frakh$ such that
$$ \frakg = \frakp_G + \frakp_H. $$
A real reductive pair $(G,H)$ of Lie groups is called \textit{strongly spherical} if the corresponding pair $(\frakg,\frakh)$ of Lie algebras is strongly spherical.
\end{definition}

This property for a real reductive pair $(G,H)$ was introduced by Kobayashi--Oshima~\cite{KO13} as property (PP). In this paper we use the notion \textit{strongly spherical}, following \cite{KKPS17}. In view of Fact~\ref{fact:FiniteMultiplicities}, strongly spherical reductive pairs are sometimes also referred to as \textit{finite-multiplicity pairs} (see e.g. \cite{KM14,KO13}). The following characterization of strongly spherical pairs follows e.g. from \cite[Lemma 5.3~(1)]{KO13} and \cite[Theorem 1.1]{KS16a}:

\begin{proposition}\label{prop:CharacterizationStronglySpherical}
Let $(G,H)$ be a real reductive pair of Lie groups and let $P_G\subseteq G$ and $P_H\subseteq H$ be minimal parabolic subalgebras. Then the following statements are equivalent:
\begin{enumerate}
\item $(G,H)$ is strongly spherical.
\item The homogeneous space $(G\times H)/\diag(H)$ is real spherical.
\item There exists an open double coset in $P_H\backslash G/P_G$.
\item $\#(P_H\backslash G/P_G)<\infty$.
\end{enumerate}
\end{proposition}

A reductive pair $(\frakg,\frakh)$ is called \textit{non-trivial} if $\frakg\neq\frakh$. We call $(\frakg,\frakh)$ \textit{indecomposable} if there does not exist a non-trivial decomposition $\frakg=\frakg_1\oplus\frakg_2$ with ideals $\frakg_i\subseteq\frakg$ such that $\frakh=(\frakh\cap\frakg_1)\oplus(\frakh\cap\frakg_2)$. Every reductive pair $(\frakg,\frakh)$ of Lie algebras can be written as a direct sum of indecomposable pairs, and we therefore assume that $(\frakg,\frakh)$ is indecomposable.

The main result of this section is a statement about non-open double cosets in $P_H\backslash G/P_G$. Note that for $P_HgP_G\in P_H\backslash G/P_G$ the stabilizer of the $P_H$-orbit through $gP_G\in G/P_G$ is given by $P_H\cap gP_Gg^{-1}$. Let $P_G=M_GA_GN_G$ and $P_H=M_HA_HN_H$ be Langlands decompositions of $P_G$ and $P_H$ and write $\frakm_G$, $\fraka_G$, $\frakn_G$ and $\frakm_H$, $\fraka_H$, $\frakn_H$ for the Lie algebras of $M_G$, $A_G$, $N_G$ and $M_H$, $A_H$, $N_H$.

\begin{theorem}\label{thm:AHProjectionStabilizers}
Assume that $(G,H)$ is a strongly spherical reductive pair such that $(\frakg,\frakh)$ is non-trivial and indecomposable. Then for a double coset $P_HgP_G\in P_H\backslash G/P_G$ the following are equivalent:
\begin{enumerate}
\item $P_HgP_G$ is open in $G$.
\item The projection of $\frakp_H\cap\Ad(g)\frakp_G$ to $\fraka_H$ along $\frakm_H\oplus\frakn_H$ is trivial.
\end{enumerate}
\end{theorem}

We first show that Theorem~\ref{thm:AHProjectionStabilizers} can be reduced to the case of semisimple $\frakg$. For this write
$$ \frakg = \frakg_{\rm n} \oplus \frakg_{\rm el}, \qquad \frakh = \frakh_{\rm n} \oplus \frakh_{\rm el}, $$
where $\frakg_{\rm n}$ resp. $\frakh_{\rm n}$ is the direct sum of all simple non-compact ideals and $\frakg_{\rm el}$ resp. $\frakh_{\rm el}$ the sum of all simple compact and all abelian ideals. Denote by $p:\frakg\to\frakg_{\rm n}$ the projection map onto the $\frakg_{\rm n}$ along $\frakg_{\rm el}$.

\begin{lemma}\label{lem:ReductionToSemisimple}
	Let $(\frakg,\frakh)$ be an indecomposable reductive pair.
	\begin{enumerate}
		\item $\ker p|_{\frakh}$ does not contain any non-compact abelian ideals of $\frakh$.
		\item If $(\frakg,\frakh)$ is strongly spherical then $(\frakg_{\rm n},p(\frakh))$ is strongly spherical.
	\end{enumerate}
\end{lemma}

\begin{proof}
	\begin{enumerate}
		\item Let $\fraka\subseteq\ker p|_{\frakh}$ be a non-compact abelian ideal, then $\fraka\subseteq\frakg_{\rm el}$ and since $\fraka$ is non-compact it has to be an ideal in $\frakg$. Write $\frakg=\fraka\oplus\frakg'$ as sum of ideals, then $\frakh=\fraka\oplus(\frakg'\cap\frakh)$ so that the indecomposability of $(\frakg,\frakh)$ forces $\fraka=0$ or $(\frakg,\frakh)=(\fraka,\fraka)$. Since $(\frakg,\frakh)$ is assumed to be non-trivial, the latter case cannot occur.
		\item We have $\frakp_G+\frakp_H=\frakg$ for some choice of minimal parabolic subalgebras $\frakp_G\subseteq\frakg$ and $\frakp_H\subseteq\frakh$. Since $\frakg_{\rm el}\subseteq\frakp_G$ we have $\frakp_G=(\frakp_G\cap\frakg_{\rm n})\oplus\frakg_{\rm el}$ and $\frakp_G\cap\frakg_{\rm n}\subseteq\frakg_{\rm n}$ is a minimal parabolic subalgebra. This implies $p(\frakp_H)+(\frakp_G\cap\frakg_{\rm n})=\frakg_{\rm n}$. Clearly $p(\frakp_H)$ is a minimal parabolic subalgebra of $p(\frakh)$ so the claim follows.\qedhere
	\end{enumerate}
\end{proof}

By the previous lemma, $p|_{\fraka_H}:\fraka_H\to\frakg_{\rm n}$ is injective, and using $\Ad(g)\frakg_{\rm n}=\frakg_{\rm n}$ it is easy to see that Theorem~\ref{thm:AHProjectionStabilizers} holds for $(\frakg,\frakh)$ if and only if it holds for $(\frakg_{\rm n},p(\frakh))$. Hence, it suffices to show Theorem~\ref{thm:AHProjectionStabilizers} in the case where $\frakg$ is semisimple.

In this section we prove the implication (2)$\Rightarrow$(1) by showing that for every non-open double coset $P_HgP_G$ there exists $Z=Z_M+Z_A+Z_N\in\frakp_H\cap\Ad(g)\frakp_G$ with $Z_A\neq0$. The remaining implication (1)$\Rightarrow$(2) is proved in Section~\ref{sec:MatrixCoefficients} (see Corollary~\ref{cor:OpenDoubleCosetImpliesTrivialAHStabilizer}). Note that Theorem~\ref{thm:AHProjectionStabilizers} is not used in Section~\ref{sec:FiniteDimBranching}.

\subsection{Classification of strongly spherical reductive pairs}\label{sec:ClassificationStronglySphericalPairs}

To state an efficient classification of strongly spherical pairs, we need an additional assumption to avoid exotic situations as in the following example:

\begin{example}\label{ex:ExoticStronglySphericalPairs}
	As observed in \cite{KKPS17}, there exist indecomposable strongly spherical pairs $(\frakg,\frakh)$ such that $\frakg$ has arbitrarily many non-compact simple factors. For instance we have the indecomposable strongly spherical reductive pair
	$$ (\frakg,\frakh) = (\sp(p,q+1)+\cdots+\sp(p,q+1),\sp(p,q)+\cdots+\sp(p,q)+\sp(1)), $$
	where each of the $k$ factors $\sp(p,q)$ of $\frakh$ is embedded into the corresponding factor $\sp(p,q+1)$ of $\frakg$ in the standard way, and $\sp(1)$ is embedded diagonally into $\frakg$ as the centralizer of $\sp(p,q)$ in each factor $\sp(p,q+1)$.
\end{example}

 Following \cite{KKPS17} we call a pair $(\frakg,\frakh)$ \textit{strictly indecomposable} if both $(\frakg,\frakh)$ and $(\frakg_{\rm n},\frakh_{\rm n})$ are indecomposable. Note that $\frakh_{\rm n}\subseteq\frakg_{\rm n}$ is automatic from the definition of $\frakg_{\rm n}$ and $\frakh_{\rm n}$. This definition excludes exotic reductive pairs as in Example~\ref{ex:ExoticStronglySphericalPairs} but still allows certain central extensions of $\frakg$ and $\frakh$ as for the multiplicity one pairs $(\frakg,\frakh)=(\gl(n+1,\FF),\gl(n,\FF))$, $\FF=\RR,\CC$, and $(\fraku(p,q+1),\fraku(p,q))$.
 
 By Lemma~\ref{lem:ReductionToSemisimple}, the study of strictly indecomposable reductive pairs $(\frakg,\frakh)$ can be reduced to the case where $\frakg$ is semisimple. In this case, a classification was obtained by Kobayashi--Matsuki~\cite{KM14} for symmetric pairs and by Knop--Kr\"{o}tz--Pecher--Schlichtkrull~\cite{KKPS16,KKPS17} for arbitrary reductive pairs:

\begin{theorem}[see \cite{KKPS16,KKPS17,KM14}]\label{thm:ClassificationStronglySphericalPairs}
Let $(\frakg,\frakh)$ be a strictly indecomposable reductive pair with $\frakg$ semisimple. Then $(\frakg,\frakh)$ is strongly spherical if and only if it is isomorphic to one of the following pairs:
\begin{enumerate}
\item[{\rm A)}] Trivial case: $\frakg=\frakh$.
\item[{\rm C)}] Compact case: $\frakg$ is the Lie algebra of a compact simple Lie group.
\item[{\rm D)}] Compact subgroup case: $\frakh=\frakk$ is the Lie algebra of a maximal compact subgroup $K$ of a non-compact simple Lie group $G$ with Lie algebra $\frakg$, or $(\frakg,\frakh)$ is one of the following subpairs of $(\frakg,\frakk)$:
\begin{itemize}
\item $(\so(1,2n),\su(n)+\frakf)$ ($n\geq2$) with $\frakf\subseteq\fraku(1)$.
\item $(\so(1,4n),\sp(n)+\frakf)$ ($n\geq1$) with $\frakf\subseteq\sp(1)$.
\item $(\so(1,16),\spin(9))$.
\item $(\so(p,7),\so(p)+\frakg_2)$ ($p=1,2$).
\item $(\so(p,8),\so(p)+\spin(7))$ ($p=1,2,3$).
\item $(\su(1,2n),\sp(n)+\frakf)$ ($n\geq1$) with $\frakf\subseteq\fraku(1)$.
\item $(\sp(1,n),\sp(n)+\frakf)$ ($n\geq1$) with $\frakf\subseteq\sp(1)$.
\item $(\su(p,q),\su(p)+\su(q))$ ($p,q\geq1$, $p\neq q$).
\item $(\so^*(2n),\su(n))$ ($n\geq3$ odd).
\item $(\frake_{6(-14)},\so(10))$.
\end{itemize}
\item[{\rm E)}] Split rank one case ($\rank_\RR\frakg=1$):
\begin{enumerate}
\item[{\rm E1)}] $(\so(1,p+q),\so(1,p)+\so(q))$ ($p,q\geq1$) or one of the following subpairs:
\begin{itemize}
\item $(\so(1,p+2q),\so(1,p)+\su(q)+\frakf)$ ($p\geq1,q\geq2$) with $\frakf\subseteq\fraku(1)$.
\item $(\so(1,p+4q),\so(1,p)+\sp(q)+\frakf)$ ($p\geq1,q\geq2$) with $\frakf\subseteq\sp(1)$.
\item $(\so(1,p+7),\so(1,p)+\frakg_2)$ ($p\geq0$).
\item $(\so(1,p+8),\so(1,p)+\spin(7))$ ($p\geq0$).
\item $(\so(1,p+16),\so(1,p)+\spin(9))$ ($p\geq0$).
\end{itemize}
\item[{\rm E2)}] $(\su(1,p+q),\su(1,p)+\su(q)+\fraku(1))$ ($p,q\geq1$) or one of the following subpairs:
\begin{itemize}
\item $(\su(1,p+q),\su(1,p)+\su(q))$ ($p,q\geq1$, $p+q\geq3$).
\item $(\su(1,p+2q),\su(1,p)+\sp(q)+\frakf)$ ($p,q\geq1$) with $\frakf\subseteq\fraku(1)$.
\end{itemize}
\item[{\rm E3)}] $(\sp(1,p+q),\sp(1,p)+\sp(q))$ ($p,q\geq1$) or the following subpair:
\begin{itemize}
\item $(\sp(1,p+1),\sp(1,p))$ ($p\geq1$).
\end{itemize}
\item[{\rm E4)}] $(\frakf_{4(-20)},\so(8,1))$.
\end{enumerate}
\item[{\rm F)}] Strong Gelfand pairs and their real forms:
\begin{enumerate}
\item[{\rm F1)}] $(\sl(n+1,\CC),\gl(n,\CC))$ ($n\geq2$).
\item[{\rm F2)}] $(\so(n+1,\CC),\so(n,\CC))$ ($n\geq2$).
\item[{\rm F3)}] $(\sl(n+1,\RR),\gl(n,\RR))$ ($n\geq1$).
\item[{\rm F4)}] $(\su(p,q+1),\su(p,q)+\frakf)$ ($p,q+1\geq1$) with $\frakf\subseteq\fraku(1)$ and $\frakf=\fraku(1)$ for $p=q,q+1$.
\item[{\rm F5)}] $(\so(p,q+1),\so(p,q))$ ($p+q\geq2$).
\end{enumerate}
\item[{\rm G)}] Group case: $(\frakg,\frakh)=(\frakg'+\frakg',\diag\frakg')$
\begin{enumerate}
\item[{\rm G1)}] $\frakg'$ is the Lie algebra of a compact simple Lie group.
\item[{\rm G2)}] $\frakg'=\so(1,n)$ ($n\geq2$).
\end{enumerate}
\item[{\rm H)}] Other cases:
\begin{enumerate}
\item[{\rm H1)}] $(\so(2,2n),\su(1,n)+\frakf)$ ($n\geq1$) with $\frakf\subseteq\fraku(1)$.
\item[{\rm H2)}] $(\su^*(2n+2),\su^*(2n)+\RR+\frakf)$ ($n\geq2$) with $\frakf\subseteq\su(2)$.
\item[{\rm H3)}] $(\so^*(2n+2),\so^*(2n)+\frakf)$ ($n\geq1$) with $\frakf\subseteq\so(2)$.
\item[{\rm H4)}] $(\sp(p,q+1),\sp(p,q)+\frakf)$ ($p,q+1\geq1$) with $\frakf\subseteq\sp(1)$.
\item[{\rm H5)}] $(\frake_{6(-26)},\so(9,1)+\RR)$.
\end{enumerate}
\end{enumerate}
\end{theorem}

The list and its enumeration is a copy of the list in \cite[Theorem 1.3]{KM14} to which we added the non-symmetric cases obtained in \cite[Table 9]{KKPS17} and the cases with $\frakh$ compact which are listed in \cite{KKPS16}. This is the reason why B) is missing since it was listed in \cite{KM14} as the abelian case $(\frakg,\frakh)=(\RR,0)$, which we exclude by assuming that $\frakg$ is semisimple.

It is immediate from the classification that every non-trivial strictly indecomposable strongly spherical reductive pair $(\frakg,\frakh)$ with $\frakg$ semisimple is contained inside a non-trivial symmetric pair $(\frakg,\frakg^\sigma)$ (see also \cite[Lemma 1.4]{KKPS16}). This statement is still true if one replaces \emph{strictly indecomposable} by \emph{indecomposable}:

\begin{corollary}\label{cor:ReductiveContainedInSymmetric}
Let $(\frakg,\frakh)$ be a non-trivial indecomposable strongly spherical reductive pair with $\frakg$ semisimple. Then there exists a non-trivial involution $\sigma$ of $\frakg$ such that $\frakh\subseteq\frakg^\sigma$, $\frakh_{\rm n}=(\frakg^\sigma)_{\rm n}$, and $\frakh_{\rm el}$ and $(\frakg^\sigma)_{\rm el}$ only differ in compact factors.
\end{corollary}

\begin{proof}
	Write $(\frakg,\frakh_{\rm n})=(\frakg_1,\frakh_1)\oplus\cdots(\frakg_p,\frakh_p)$ with each $(\frakg_j,\frakh_j)$ indecomposable. Let $p_j:\frakg\to\frakg_j$ denote the projection onto $\frakg_j$ in the decomposition $\frakg=\frakg_1\oplus\cdots\oplus\frakg_p$. Then each pair $(\frakg_j,\frakh_j\oplus p_j(\frakh_{\rm el})))$ is strongly spherical, strictly indecomposable and non-trivial, so by the classification in Theorem~\ref{thm:ClassificationStronglySphericalPairs} there exists an involution $\sigma_j$ on $\frakg_j$ such that $\frakh_j\subseteq\frakg_j^{\sigma_j}\subseteq\frakg_j$, $(\frakh_j)_{\rm n}=(\frakg_j^{\sigma_j})_{\rm n}$, and $p_j(\frakh_{\rm el})$ and $(\frakg_j^{\sigma_j})_{\rm el}$ only differ in compact factors. Define $\sigma$ on $\frakg=\frakg_1\oplus\cdots\oplus\frakg_p$ by
	$$ \sigma(X_1+\cdots+X_p) = \sigma_1(X_1)+\cdots+\sigma_p(X_p), \qquad (X_j\in\frakg_j), $$
	then clearly $\frakh\subseteq\frakg^\sigma\subseteq\frakg$ and  $\frakh_{\rm n}=(\frakh_1)_{\rm n}\oplus\cdots\oplus(\frakh_p)_{\rm n}=(\frakg_1^{\sigma_1})_{\rm n}\oplus\cdots\oplus(\frakg_p^{\sigma_p})_{\rm n}=(\frakg^\sigma)_{\rm n}$. It remains to show that $(\frakg^\sigma)_{\rm el}=(\frakg_1^{\sigma_1})_{\rm el}\oplus\cdots\oplus(\frakg_p^{\sigma_p})_{\rm el}$ and $\frakh_{\rm el}$ only differ in compact factors. Since for every $j$ the subalgebras $p_j(\frakh_{\rm el})$ and $(\frakg_j^{\sigma_j})_{\rm el}$ only differ in compact factors, it suffices to show that a non-compact abelian ideal $\fraka\subseteq\frakh_{\rm el}$ is already contained in one of the subalgebras $\frakg_j^{\sigma_j}$. Let $\frakh_{\rm el}=\fraka\oplus\frakh_{\rm el}'$ with $\fraka$ a non-compact abelian ideal and assume that $p_j(\fraka)\neq\{0\}$ for some $j$. Since $p_j(\fraka)$ is an ideal in $\frakh_j\oplus p_j(\frakh_{\rm el})$ and the pair $(\frakg_j,\frakh_j\oplus p_j(\frakh_{\rm el}))$ is in the list in Theorem~\ref{thm:ClassificationStronglySphericalPairs}, it follows from a case-by-case inspection that $p_j(\fraka)$ is one-dimensional and that $(\frakg_j,\frakh_j\oplus p_j(\frakh_{\rm el}'))$ is \emph{not} strongly spherical. Assume now that $p_k(\fraka)\neq\{0\}$ for some $k\neq j$; then by the same argument $p_k(\fraka)$ is one-dimensional and $(\frakg_k,\frakh_k\oplus p_k(\frakh_{\rm el}'))$ is \emph{not} strongly spherical. This contradicts the fact that $(\frakg,\frakh)$ is strongly spherical by a simple dimension count. Hence, $\fraka\subseteq\frakg_j$, and another look at the classification shows that $\fraka\subseteq\frakg_j^{\sigma_j}$. This finishes the proof.
\end{proof}

It is clear that in this case $(\frakg,\frakg^\sigma)$ is a strongly spherical symmetric pair. This observation will be used to reduce several statements to the case of symmetric pairs.

\subsection{Structure of strongly spherical reductive pairs}\label{sec:StructureStronglySphericalReductivePairs}

We adapt the structure theory developed in \cite{KM14} for strongly spherical symmetric pairs to the case of strongly spherical reductive pairs. For the rest of this section let $(G,H)$ be a strongly spherical reductive pair with $\frakg$ semisimple such that $(\frakg,\frakh)$ is non-trivial and indecomposable. By Corollary~\ref{cor:ReductiveContainedInSymmetric} there exists an involution $\sigma$ of $G$ such that $\frakh\subseteq\frakg^\sigma$, $\frakh_{\rm n}=(\frakg^\sigma)_{\rm n}$, and $\frakh_{\rm el}$ and $(\frakg^\sigma)_{\rm el}$ differ only in compact factors. We first choose minimal parabolic subgroups $P_G\subseteq G$ and $P_H\subseteq H$ in a compatible way.

There exists a Cartan involution $\theta$ of $G$ which commutes with $\sigma$ and leaves $H$ invariant, and hence
$$ K = G^\theta \subseteq G \qquad \mbox{and} \qquad H\cap K = H^\theta \subseteq H $$
are maximal compact subgroups of $G$ and $H$. Fix a maximal abelian subspace $\fraka_H\subseteq\frakh^{-\theta}=\frakg^{\sigma,-\theta}$ and extend it to a maximal abelian subspace $\fraka_G$ in $\frakg^{-\theta}$. Then $\fraka_G$ is $\sigma$-stable and $\fraka_G=\fraka_H\oplus\fraka_G^{-\sigma}$ with $\fraka_H=\fraka_G^\sigma$. We put
$$ A_G = \exp(\fraka_G), \qquad A_H = \exp(\fraka_H). $$

For $\alpha\in\fraka_G^\vee$ and $\beta\in\fraka_H^\vee$ we write
\begin{align*}
 \frakg(\fraka_G;\alpha) &= \{X\in\frakg:[H,X]=\alpha(H)X\,\forall\,H\in\fraka_G\},\\
 \frakg(\fraka_H;\beta) &= \{X\in\frakg:[H,X]=\beta(H)X\,\forall\,H\in\fraka_H\}.
\end{align*}
for the corresponding weight spaces. Let $\Sigma(\frakg,\fraka_G)$ and $\Sigma(\frakg,\fraka_H)$ denote the respective non-zero weights with non-trivial weight spaces; then both sets form root systems. Denote by $\overline{\alpha}=\alpha|_{\fraka_H}$ the restriction of a root $\alpha\in\Sigma(\frakg,\fraka_G)$ to $\fraka_H$, then $\overline{\alpha}\in\Sigma(\frakg,\fraka_H)\cup\{0\}$. We choose compatible positive systems $\Sigma^+(\frakg,\fraka_G)$ and $\Sigma^+(\frakg,\fraka_H)$ in the sense that
$$ \overline{\alpha}\in\Sigma^+(\frakg,\fraka_H)\cup\{0\} \quad \forall\,\alpha\in\Sigma^+(\frakg,\fraka_G). $$
As usual, for $\alpha\in\Sigma(\frakg,\fraka_G)$ we write $\alpha>0$ if $\alpha\in\Sigma^+(\frakg,\fraka_G)$ and $\alpha<0$ if $-\alpha\in\Sigma^+(\frakg,\fraka_G)$.

Further, define the nilpotent subalgebras
$$ \frakn_G = \bigoplus_{\alpha\in\Sigma^+(\frakg,\fraka_G)} \frakg(\fraka_G;\alpha), \qquad \frakn = \bigoplus_{\substack{\alpha\in\Sigma^+(\frakg,\fraka_G)\\\overline{\alpha}\neq0}} \frakg(\fraka_G;\alpha) = \bigoplus_{\beta\in\Sigma^+(\frakg,\fraka_H)} \frakg(\fraka_H;\beta). $$
Then $\frakn$ is $\sigma$-stable and therefore we have a direct sum decomposition
$$ \frakn = \frakn^\sigma \oplus \frakn^{-\sigma}. $$
Put $\frakn_H=\frakn^\sigma$ and
$$ N_G = \exp(\frakn_G), \qquad N = \exp(\frakn), \qquad N_H = \exp(\frakn_H). $$

Finally, we define
$$ M_G = Z_K(\fraka_G), \qquad L = Z_G(\fraka_H), \qquad M_H = Z_{H\cap K}(\fraka_H). $$
Then $P_G=M_GA_GN_G$ is a minimal parabolic subgroup of $G$, $Q=LN$ is another parabolic subgroup of $G$, and $P_H=M_HA_HN_H$ is a minimal parabolic subgroup of $H$ such that
$$ P_G\subseteq Q\supseteq P_H. $$

\subsection{The double coset space $P_H\backslash G/P_G$}

To study the $P_H$-orbits in $G/P_G$ we use the Bruhat decomposition of $G$ with respect to the parabolic subgroups $P_G$ and $Q$. Let $W=W(\fraka_G)=N_K(A_G)/A_G$ denote the Weyl group of $\Sigma(\frakg,\fraka_G)$ and pick a representative $\tilde{w}\in N_K(A_G)$ for every $w\in W$. Then, since $G$ is of Harish-Chandra class, we have the Bruhat decomposition (see e.g. \cite[Proposition 1.2.1.10]{War72})
$$ G = \bigcup_{w\in W_Q\backslash W} Q\tilde{w}P_G, $$
where
$$ W_Q = Z_W(A_H) = \{w\in W:\tilde{w}\in L\cap K\}. $$
Since $P_H\subseteq Q$, every $P_H$-orbit $P_H\cdot gP_G$ is contained in a Bruhat cell $Q\cdot\tilde{w}P_G\subseteq G/P$. As homogeneous spaces we have
$$ Q\cdot\tilde{w}P_G \simeq Q/(Q\cap\tilde{w}P_G\tilde{w}^{-1}) $$
whence we are led to study the $P_H$-orbits in $Q/(Q\cap\tilde{w}P_G\tilde{w}^{-1})$ for $w\in W$. The following result is shown in \cite[Lemma 3.5~(2)]{KM14} for strongly spherical symmetric pairs, and we extend it to the context of strongly spherical reductive pairs:

\begin{lemma}\label{lem:MGplusMH}
The extension of $\fraka_H$ to $\fraka_G$ can be chosen such that
$$ \frakl\cap\frakk = Z_\frakk(\fraka_H) = \frakm_H+\frakm_G \qquad \mbox{and} \qquad L\cap K = Z_K(\fraka_H) = M_HM_G. $$
\end{lemma}

\begin{proof}
Let $P_{G^\sigma}=Z_{K^\sigma}(\fraka_H)A_HN_H$, a minimal parabolic subgroup of $G^\sigma$. Since $(G,H)$ is strongly spherical, there exists $g\in G$ such that $P_HgP_G$ is open. Now $P_H\subseteq P_{G^\sigma}$ and hence $P_HgP_G$ is contained in an open double coset in $P_{G^\sigma}\backslash G/P_G$. By \cite[Lemma 3.7]{KM14} such open double cosets have representatives in $\exp(\frakn^{-\sigma})\widetilde{w}_0$ so that $P_HgP_G\subseteq P_{G^\sigma}\exp(X)\widetilde{w}_0P_G$ for some $X\in\frakn^{-\sigma}$. Hence, $g\in p\exp(X)\widetilde{w}_0P_G$ with $p\in P_{G^\sigma}$ so that
$$ P_HgP_G = P_Hp\exp(X)\widetilde{w}_0P_G. $$
Now $P_Hp=P_Hm$ with $m\in Z_{K^\sigma}(\fraka_H)$ and $(m^{-1}P_Hm)\exp(X)\widetilde{w}_0P_G\subseteq G$ is open. This implies
$$ \Ad(\exp(-X))\Ad(m^{-1})\frakp_H+\overline{\frakp}_G=\frakg, $$
where $\overline{\frakp}_G=\Ad(\widetilde{w}_0)\frakp_G$ denotes the opposite parabolic subalgebra. Note that for $Z=Z_M+Z_A+Z_N\in\Ad(m)^{-1}\frakm_H+\fraka_H+\frakn_H=\Ad(m^{-1})\frakp_H$ we have
$$ \Ad(\exp(-X))Z = \underbrace{Z_M+Z_A}_{\in\Ad(m^{-1})\frakm_H+\fraka_H\subseteq\frakl} + \underbrace{(1-\Ad(\exp(-X)))(Z_M+Z_A)+\Ad(\exp(-X))Z_N}_{\in\frakn} $$
and further
$$ \overline{\frakp}_G = \underbrace{(\frakl\cap\overline{\frakp}_G)}_{\subseteq\frakl} \oplus \underbrace{\overline{\frakn}_G}_{\subseteq\overline{\frakn}}. $$
From the decomposition $\frakg=\overline{\frakn}\oplus\frakl\oplus\frakn$ it now follows that $\Ad(m^{-1})\frakm_H+(\frakl\cap\overline{\frakp}_G)=\frakl$. Intersecting with $\frakk$ we obtain $\Ad(m^{-1})\frakm_H+\frakm_G=\frakl\cap\frakk$, or equivalently $\frakm_H+\Ad(m)\frakm_G=\frakl\cap\frakk$. Replacing $\fraka_G$ by $\Ad(m)\fraka_G$ changes $\frakm_G$ to $\Ad(m)\frakm_G$ and hence we may assume $\frakm_H+\frakm_G=\frakl\cap\frakk$. The identity $L\cap K=M_HM_G$ then follows since $M_GA_G\subseteq L$ is the Levi factor of a minimal parabolic subgroup of $L$ and hence $M_G$ meets every connected component of $L$.
\end{proof}

For the rest of this section we choose the extension $\fraka_G$ of $\fraka_H$ as in Lemma~\ref{lem:MGplusMH}. Then we obtain the following generalization of \cite[Lemma 3.7]{KM14} to not necessarily symmetric $(G,H)$:

\begin{lemma}\label{lem:RepresentativesDoubleCosets}
For every $w\in W$ we have
$$ Q = M_HN_H\exp(\frakn^{-\sigma})(L\cap\tilde{w}P_G\tilde{w}^{-1}). $$
In particular, every $P_H$-orbit in $Q/(Q\cap\tilde{w}P_G\tilde{w}^{-1})$ has a representative of the form $\exp(X)$ for some $X\in\frakn^{-\sigma}$, or equivalently every $P_H$-orbit in $G/P_G$ has a representative of the form $\exp(X)\tilde{w}$ for some $X\in\frakn^{-\sigma}$, $w\in W$.
\end{lemma}

\begin{proof}
First note that $\tilde{w}P_G\tilde{w}^{-1}$ is a minimal parabolic subgroup of $G$. Since $L$ is a reductive subgroup of $G$, the intersection $L\cap\tilde{w}P_G\tilde{w}^{-1}$ is parabolic in $L$. By the Iwasawa decomposition for $L$ we find that $L=(L\cap K)(L\cap\tilde{w}P_G\tilde{w}^{-1})$. Now, $L\cap K=M_HM_G$ by Lemma~\ref{lem:MGplusMH}, whence
$$ L = M_HM_G(L\cap\tilde{w}P_G\tilde{w}^{-1}) = M_H(L\cap\tilde{w}P_G\tilde{w}^{-1}). $$
Inserting this into the Langlands decomposition for $Q$ we find
$$ Q = NL = NM_H(L\cap\tilde{w}P_G\tilde{w}^{-1}) = M_HN(L\cap\tilde{w}P_G\tilde{w}^{-1}). $$
Finally, $N=N_H\exp(\frakn^{-\sigma})$ by \cite[Lemma 3.6]{KM14} and the proof is complete.
\end{proof}

\subsection{$P_H$-orbits in the open Bruhat cell}

Let $w_0\in W$ denote the longest Weyl group element. Then the Bruhat cell $Q\tilde{w}_0P_G$ is open and dense in $G$. Hence, $Q\tilde{w}_0P_G$ is the unique open Bruhat cell.

Now let us consider the double cosets $P_HgP_G$ which are contained in the open Bruhat cell $Q\tilde{w}_0P_H$. These can be identified with the orbits of the adjoint action of $(M_G\cap M_H)A_H$ on $\frakn^{-\sigma}$:

\begin{lemma}\label{lem:PHOrbitsInOpenBruhatCell}
The natural inclusion
$$ \frakn^{-\sigma}\stackrel{\exp}{\to} N \hookrightarrow Q $$
induces a bijection
\begin{equation}
 \frakn^{-\sigma}/(M_G\cap M_H)A_H \simeq P_H\backslash Q/(Q\cap\tilde{w}_0P_G\tilde{w}_0^{-1}).\label{eq:LinearizationOpenBruhatCell}
\end{equation}
\end{lemma}

\begin{proof}
The proof of \cite[Lemma 3.7]{KM14} for the case of symmetric pairs translates literally to our situation.
\end{proof}

Since $\fraka_H$ preserves $\frakn^{\pm\sigma}$ we can write
\begin{equation}
 \frakn^{\pm\sigma} = \bigoplus_{\beta\in\Delta(\frakn^{\pm\sigma})} \frakg^{\pm\sigma}(\fraka_H;\beta),\label{eq:DecompositionN-sigma}
\end{equation}
where $\frakg^{\pm\sigma}(\fraka_H;\beta)=\frakg(\fraka_H;\beta)\cap\frakg^{\pm\sigma}$ and
$$ \Delta(\frakn^{\pm\sigma}) = \{\beta\in\Sigma^+(\frakg,\fraka_H):\frakg(\fraka_H;\beta)\cap\frakn^{\pm\sigma}\neq\{0\}\}. $$
Clearly $M_H$ preserves this decomposition. We can therefore endow $\frakn^{-\sigma}$ with an $M_H$-invariant inner product such that the decomposition~\eqref{eq:DecompositionN-sigma} is orthogonal. For each $\beta\in\Delta(\frakn^{-\sigma})$, denote by $S_\beta\subseteq\frakg^{-\sigma}(\fraka_H;\beta)$ the unit sphere with respect to this inner product. We then have the following version of \cite[Proposition 3.11]{KM14}:

\begin{lemma}\label{lem:RealSphImpliesTransitiveActionOnSpheres}
If the pair $(G,H)$ is strongly spherical, then
\begin{enumerate}
\item\label{lem:RealSphImpliesTransitiveActionOnSpheres1} the orbits of $(M_G\cap M_H)$ on $\prod_{\beta\in\Delta(\frakn^{-\sigma})}S_\beta$ are unions of connected components, in particular open and compact;
\item\label{lem:RealSphImpliesTransitiveActionOnSpheres2} $\Delta(\frakn^{-\sigma})$ is a basis of $\fraka_H^\vee$.
\end{enumerate}
\end{lemma}

\begin{proof}
The proof of \cite[Proposition 3.11]{KM14} implies \eqref{lem:RealSphImpliesTransitiveActionOnSpheres1} and linear independence of $\Delta(\frakn^{-\sigma})$ in \eqref{lem:RealSphImpliesTransitiveActionOnSpheres2}. Assume that $\Delta(\frakn^{-\sigma})$ does not span $\fraka_H^\vee$; then there exists $0\neq H\in\fraka_H$ such that $\beta(H)=0$ for all $\beta\in\Delta(\frakn^{-\sigma})$, or equivalently $\ad(H)|_{\frakn^{-\sigma}}=0$. But this implies $H\in\frakp_H\cap\Ad(e^X\tilde{w}_0)\frakp_G=\frakp_H\cap\Ad(e^X)\overline{\frakp}_G$ for all $X\in\frakn^{-\sigma}$. Since by Lemma~\ref{lem:RepresentativesDoubleCosets} every open double coset in $P_HgP_G\in P_H\backslash G/P_G$ has a representative of the form $g=e^X\tilde{w}_0$, this shows that the projection of $\frakp_H\cap\Ad(g)\frakp_G$ to $\fraka_H$ is non-trivial, which contradicts Theorem~\ref{thm:AHProjectionStabilizers}.
\end{proof}

\subsection{Proof of Theorem~\ref{thm:AHProjectionStabilizers}}\label{sec:ProofOfAHProjectionStabilizersTheorem}

Let $P_H\cdot gP_G$ be a non-open $P_H$-orbit. By Lemma~\ref{lem:RepresentativesDoubleCosets} there exists $p\in P_H$ such that $pgP_G=\exp(X)\tilde{w}P_G$ for some $X\in\frakn^{-\sigma}$ and $w\in W$. Then the stabilizers of $gP_G$ and $\exp(X)\tilde{w}P_G$ are conjugate in $P_H$ via $p$. Since the projection of $\Ad(p)\fraka_H\subseteq\frakp_H$ to $\fraka_H$ is equal to $\fraka_H$, we may without loss of generality assume that $g=\exp(X)\tilde{w}$.

Write
\begin{align*}
 S(X,\tilde{w}) &= P_H\cap(\exp(X)\tilde{w}P_G\tilde{w}^{-1}\exp(-X)) = \{p\in P_H:p\exp(X)\in\exp(X)(\tilde{w}P_G\tilde{w}^{-1})\}
\intertext{for the stabilizer of $\exp(X)\tilde{w}P_G$ in $P_H$ and}
 \fraks(X,\tilde{w}) &= \{Z\in\frakp_H:\exp(\RR Z)\exp(X)\subseteq\exp(X)(\tilde{w}P_G\tilde{w}^{-1})\} = \{Z\in\frakp_H:e^{-\ad(X)}Z\in w\cdot\frakp_G\}
\end{align*}
for its Lie algebra. Note that the Lie algebra $w\cdot\frakp_G$ of $\tilde{w}P_G\tilde{w}^{-1}$ is given by
$$ w\cdot\frakp_G = \frakm_G\oplus\fraka_G\oplus\bigoplus_{\substack{\alpha\in\Sigma(\frakg,\fraka_G)\\w^{-1}\alpha>0}}\frakg(\fraka_G;\alpha). $$
In view of the decomposition \eqref{eq:DecompositionN-sigma} we can write
$$ X = \sum_{\beta\in\Delta(\frakn^{-\sigma})}X_\beta $$
with $X_\beta\in\frakg^{-\sigma}(\fraka_H;\beta)$.

\subsubsection{Reduction to generic $X$}\label{sec:ReductionToGenericX}

Assume first that $X_\beta=0$ for some $\beta\in\Delta(\frakn^{-\sigma})$. It follows from Lemma~\ref{lem:RealSphImpliesTransitiveActionOnSpheres}~(2) that the intersection
$$ \bigcap_{\beta'\in\Delta(\frakn^{-\sigma})\setminus\{\beta\}} \ker\beta' \subseteq \fraka_H $$
is non-trivial, so there exists $Z_A\in\fraka_H\setminus\{0\}$ with $\beta'(Z_A)=0$ for all $\beta'\neq\beta$. Then $[Z_A,X]=0$ since $X_\beta=0$ and $[Z_A,X_{\beta'}]=0$ for $\beta'\neq\beta$, and hence $e^{tZ_A}e^X=e^Xe^{tZ_A}$ for all $t\in\RR$. Since $Z_A\in\fraka_H\subseteq w\cdot\frakp_G$ this implies $Z_A\in\fraks(X,\tilde{w})$.

We may therefore assume $X_\beta\neq0$ for all $\beta\in\Delta(\frakn^{-\sigma})$ for the rest of the proof. By Lemma~\ref{lem:RealSphImpliesTransitiveActionOnSpheres} this guarantees that the $(M_G\cap M_H)A_H$-orbit of $X$ in $\frakn^{-\sigma}$ is open.

\subsubsection{Double cosets in the open Bruhat cell}

Now suppose $Q\tilde{w}P_G=Q\tilde{w}_0P_G$, i.e. the double coset $P_H\exp(X)\tilde{w}P_G$ is contained in the open Bruhat cell $Q\tilde{w}_0P_G$. Without loss of generality we may assume $w=w_0$. If the double coset $P_H\exp(X)\tilde{w}_0P_G$ is not open, then by Lemma~\ref{lem:PHOrbitsInOpenBruhatCell} the $(M_G\cap M_H)A_H$-orbit of $X$ in $\frakn^{-\sigma}$ is not open. By Lemma~\ref{lem:RealSphImpliesTransitiveActionOnSpheres} this implies that one of the $X_\beta$ must vanish, the case we already treated in Section~\ref{sec:ReductionToGenericX}.

\subsubsection{Double cosets in the non-open Bruhat cells}

Now assume that $Q\tilde{w}P_G\neq Q\tilde{w}_0P_G$.

\begin{lemma}\label{lem:WeylGroupLemma}
Let $w\in W$. If $Q\tilde{w}P_G\neq Q\tilde{w}_0P_G$, then there exists $\alpha\in\Sigma^+(\frakg,\fraka_G)\cap w\Sigma^+(\frakg,\fraka_G)$ with $\overline{\alpha}\neq0$.
\end{lemma}

\begin{proof}
Assume $w^{-1}\alpha<0$ for all $\alpha>0$ with $\overline{\alpha}\neq0$. The double coset $Q\tilde{w}P_G$ is the orbit of $\tilde{w}$ under the action of $Q\times P_G$ on $G$ given by $(q,p)\cdot g=qgp^{-1}$. Then the stabilizer of $\tilde{w}$ in $Q\times P_G$ is given by
$$ \{(q,\tilde{w}^{-1}q\tilde{w}):q\in Q\cap\tilde{w}P_G\tilde{w}^{-1}\} $$
and hence its Lie algebra is isomorphic to $\frakq\cap w\cdot\frakp_G$. Since $\fraka_G\subseteq\frakq\cap w\cdot\frakp_G$ we can decompose this Lie algebra into root spaces with respect to $\fraka_G$:
$$ \frakq\cap w\cdot\frakp_G = \frakm_G\oplus\fraka_G\oplus\bigoplus_{\substack{\overline{\alpha}=0\\w^{-1}\alpha>0}}\frakg(\fraka_G;\alpha)\oplus\bigoplus_{\substack{\overline{\alpha}\neq0\\\alpha>0,\,w^{-1}\alpha>0}}\frakg(\fraka_G;\alpha) = \frakm_G\oplus\fraka_G\oplus\bigoplus_{\substack{\overline{\alpha}=0\\w^{-1}\alpha>0}}\frakg(\fraka_G;\alpha). $$
Now, if $\overline{\alpha}=0$ then either $w^{-1}\alpha>0$ or $w^{-1}(-\alpha)=-w^{-1}\alpha>0$, and therefore
$$ \dim(\frakq\cap w\cdot\frakp_G) = \dim\frakm_G+\dim\fraka_G+\sum_{\substack{\alpha>0\\\overline{\alpha}=0}}\dim\frakg(\fraka_G;\alpha) = \dim\frakp_G - \dim\frakn. $$
Hence,
$$ \dim(Q\tilde{w}P_G) = \dim\frakq+\dim\frakp_G-\dim(\frakq\cap w\cdot\frakp_G) = \dim\frakq + \dim\frakn = \dim\frakg $$
so that $Q\tilde{w}P_G$ must be open, whence equal to the unique open Bruhat cell $Q\tilde{w}_0P_G$. This contradicts the assumption $Q\tilde{w}P_G\neq Q\tilde{w}_0P_G$ and the proof is complete.
\end{proof}

Choose any root $\alpha>0$ with $\overline{\alpha}\neq0$ and $w^{-1}\alpha>0$, then $\frakg(\fraka_G;\alpha)\subseteq\frakq\cap w\cdot\frakp_G$. Let us fix $Y\in\frakg(\fraka_G;\alpha)$ for now; later we will specify an appropriate choice of $Y$. Then $e^{tY}\in N$ and also $e^Xe^{tY}\in N$ for all $t\in\RR$. By \cite[Lemma 3.6]{KM14} we have $N=N^\sigma\exp(\frakn^{-\sigma})$ so that we can uniquely write
$$ e^Xe^{tY}=n_te^{X_t} $$
with $n_t\in N^\sigma$ and $X_t\in\frakn^{-\sigma}$ depending differentiably on $t\in\RR$. Recall that we may assume $X_\beta\neq0$ for all $\beta\in\Delta(\frakn^{-\sigma})$, so that $X$ is contained in an open $(M_G\cap M_H)A_H$-orbit in $\frakn^{-\sigma}$. Then there exists an interval $(-\varepsilon,\varepsilon)$ such that $X_t$ belongs to the open orbit $\Ad((M_G\cap M_H)A_H)X$, so there exists $m_ta_t\in(M_G\cap M_H)A_H$ such that $\Ad(m_ta_t)X=X_t$. Clearly, $m_ta_t$ can be chosen to depend differentiably on $t$ with $m_0=a_0=\1$. Summarizing, we have
\begin{equation}
 e^Xe^{tY} = n_te^{\Ad(m_ta_t)X}, \qquad t\in(-\varepsilon,\varepsilon).\label{eq:DefiningEquationForNtAtMt}
\end{equation}
Denoting $p_t=n_ta_tm_t\in N_HA_H(M_G\cap M_H)\subseteq P_G\cap P_H$ we have
$$ p_te^X = e^Xe^{tY}m_ta_t \in e^X(\tilde{w}P_G\tilde{w}^{-1}), $$
whence $p_t\in S(X,\tilde{w})$. Now put
$$ Z = \left.\frac{d}{dt}\right|_{t=0}p_t \in \fraks(X,\tilde{w}) $$
and write $Z=Z_M+Z_A+Z_N\in(\frakm_G\cap\frakm_H)\oplus\fraka_H\oplus\frakn_H$. It remains to show that $Y\in\frakg(\fraka_G;\alpha)$ can be chosen such that $Z_A\neq0$. For this we use the following identity which follows by taking the left logarithmic derivative of \eqref{eq:DefiningEquationForNtAtMt}:
\begin{equation}
 Y = e^{-\ad(X)}Z-(Z_M+Z_A) = (e^{-\ad(X)}-\1)(Z_M+Z_A) + e^{-\ad(X)}Z_N.\label{eq:DefiningEquationForZ}
\end{equation}

Now, write
$$ Z_N = \sum_{\beta\in\Delta(\frakn^\sigma)} Z_{N,\beta} $$
with $Z_{N,\beta}\in\frakg^\sigma(\fraka_H;\beta)$.

\begin{lemma}\label{lem:IdentityForYandZ}
If $Y\in\frakg(\fraka_G;\alpha)$ for $\alpha>0$ with $\beta=\overline{\alpha}\neq0$ and $w^{-1}\alpha>0$, and $Z=Z_M+Z_A+Z_N\in(\frakm_G\cap\frakm_H)\oplus\fraka_H\oplus\frakn_H$ satisfies \eqref{eq:DefiningEquationForZ}, then
$$ Y = \ad(Z_M+Z_A)X_\beta + Z_{N,\beta}. $$
\end{lemma}

\begin{proof}
We can enumerate the positive $\fraka_H$-roots as $\Sigma^+(\frakg,\fraka_H)=\{\beta_1,\ldots,\beta_p\}$ so that $\beta_j\not<\beta_i$ whenever $i\leq j$. Form the nilpotent subalgebras
$$ \frakn_i = \bigoplus_{k=i}^p \frakg(\fraka_H;\beta_k) \subseteq \frakn, $$
then $\frakn_1=\frakn$ and $[\frakn,\frakn_i]\subseteq\frakn_{i+1}$. Since $\alpha>0$ with $\overline{\alpha}\neq0$ there exists $1\leq i\leq p$ with $\overline{\alpha}=\beta_i$. We first prove by induction that
$$ \ad(Z_M+Z_A)X_{\beta_j} = 0 = Z_{N,\beta_j} \qquad \forall\,1\leq j<i. $$
For $j=1<p$, the $\beta_1$-component of $Y$ is trivial, and therefore, taking the $\beta_1$-component of \eqref{eq:DefiningEquationForZ} yields
$$ 0 = \ad(Z_M+Z_A)X_{\beta_1} + Z_{N,\beta_1}. $$
Since $\ad(Z_M+Z_A)X_{\beta_1}\in\frakn^{-\sigma}$ and $Z_{N,\beta_1}\in\frakn^\sigma$, this implies $\ad(Z_M+Z_A)X_{\beta_1}=0=Z_{N,\beta_1}$. For the induction step assume that $\ad(Z_M+Z_A)X_{\beta_k}=0=Z_{N,\beta_k}$ for $1\leq k\leq j-1$. If $j<i$ then the $\beta_j$-component of $Y$ is trivial, and we can again take the $\beta_j$-component of \eqref{eq:DefiningEquationForZ} and find
$$ 0 = \ad(Z_M+Z_A)X_{\beta_j} + Z_{N,\beta_j}. $$
The same argument as above shows $\ad(Z_M+Z_A)X_{\beta_j}=0=Z_{N,\beta_j}$. Finally, taking the $\beta_i$-component in \eqref{eq:DefiningEquationForZ} gives the desired identity.
\end{proof}

To choose $Y\in\frakg(\fraka_G;\alpha)$ such that $Z_A\neq0$, we need to relate $\frakg(\fraka_G;\alpha)$ and $\frakg(\fraka_H;\beta)$.

\begin{lemma}\label{lem:RelationAGandAHRootSpaces}
Let $\alpha\in\Sigma^+(\frakg,\fraka_G)$ with $\overline{\alpha}\neq0$. Then precisely one of the following three statements holds:
\begin{enumerate}
\item\label{lem:RelationAGandAHRootSpaces3} $\sigma\alpha\neq\alpha$ and $\frakg^{-\sigma}(\fraka_H;\overline{\alpha})=\{Y-\sigma Y:Y\in\frakg(\fraka_G;\alpha)\}\neq\{0\}$,
\item\label{lem:RelationAGandAHRootSpaces1} $\sigma\alpha=\alpha$ and $\frakg^{-\sigma}(\fraka_H;\overline{\alpha})=\frakg(\fraka_G;\alpha)\cap\frakn^{-\sigma}\neq\{0\}$,
\item\label{lem:RelationAGandAHRootSpaces2} $\sigma\alpha=\alpha$ and $\frakg(\fraka_G;\alpha)\subseteq\frakn_H$.
\end{enumerate}
\end{lemma}

\begin{proof}
Assume first that $\sigma\alpha\neq\alpha$; then $\sigma(\frakg(\fraka_G;\alpha))=\frakg(\fraka_G;\sigma\alpha)$ and the map
$$ \frakg(\fraka_G;\alpha) \to \frakg^{-\sigma}(\fraka_H;\overline{\alpha}), \quad Y\mapsto Y-\sigma(Y) $$
is non-trivial and $(M_G\cap M_H)$-equivariant. Since $\frakg^{-\sigma}(\fraka_H;\overline{\alpha})$ is irreducible under the action of $M_G\cap M_H$ by Lemma~\ref{lem:RealSphImpliesTransitiveActionOnSpheres}~(1), the map must be a linear isomorphism, so \eqref{lem:RelationAGandAHRootSpaces3} holds.\\
Now assume, $\sigma\alpha=\alpha$, then $\frakg(\fraka_G;\alpha)$ is $\sigma$-stable and we have a decomposition into eigenspaces
$$ \frakg(\fraka_G;\alpha) = (\frakg(\fraka_G;\alpha)\cap\frakn_H)\oplus(\frakg(\fraka_G;\alpha)\cap\frakn^{-\sigma}). $$
If $\frakg(\fraka_G;\alpha)\cap\frakn^{-\sigma}\neq\{0\}$, then this is a non-trivial $(M_G\cap M_H)$-invariant subspace of $\frakg^{-\sigma}(\fraka_H;\overline{\alpha})$ and the latter is irreducible under the action of $M_G\cap M_H$. This implies \eqref{lem:RelationAGandAHRootSpaces1}. The remaining possibility is $\frakg(\fraka_G;\alpha)\cap\frakn^{-\sigma}=\{0\}$, which clearly implies \eqref{lem:RelationAGandAHRootSpaces2}.
\end{proof}

Before we can finish the proof we need to study case \eqref{lem:RelationAGandAHRootSpaces2} in Lemma~\ref{lem:RelationAGandAHRootSpaces} in more detail. For this, the following two results will be used:

\begin{lemma}[{\cite[Lemma 2.9]{OS82}}]\label{lem:OSLemma}
Let $\alpha_1,\alpha_2\in\Sigma(\frakg,\fraka_G)$ and assume that $(\alpha_1,\alpha_2)<0$; then for any $X_1\in\frakg(\fraka_G;\alpha_1)$ and $X_2\in\frakg(\fraka_G;\alpha_2)$, $X_1,X_2\neq0$, we have $[X_1,X_2]\neq0$.
\end{lemma}

\begin{lemma}\label{lem:RootSequence}
Let $w\in W$, then for any $\alpha\in\Sigma^+(\frakg,\fraka_G)\cap w\Sigma^+(\frakg,\fraka_G)$ there exists a sequence $\alpha=\alpha_1,\ldots,\alpha_r\in\Sigma^+(\frakg,\fraka_G)\cap w\Sigma^+(\frakg,\fraka_G)$ such that $\langle\alpha_i,\alpha_{i+1}\rangle\neq0$ for $i=1,\ldots,r-1$ and $\alpha_r$ is simple in $\Sigma^+(\frakg,\fraka_G)$.
\end{lemma}

This result and its proof were communicated to us by Yoshiki Oshima.

\begin{proof}
Let $ww_0=s_{\beta_1}\cdots s_{\beta_n}$ be a reduced expression for $ww_0\in W$, i.e. $\beta_1,\ldots,\beta_n$ are simple roots. It is known that
$$ \Sigma^+(\frakg,\fraka_G)\cap w\Sigma^+(\frakg,\fraka_G) = \Sigma^+(\frakg,\fraka_G)\cap\big(-ww_0\Sigma^+(\frakg,\fraka_G)\big) = \{s_{\beta_1}\cdots s_{\beta_{k-1}}\beta_k:1\leq k\leq n\}, $$
and we can write $\alpha_1=\alpha=s_{\beta_1}\cdots s_{\beta_{k-1}}\beta_k$. If $\alpha_1$ is not a simple root, there exists $1\leq i<k$ such that $s_{\beta_{i+1}}\cdots s_{\beta_{k-1}}\beta_k$ is simple, but $s_{\beta_i}\cdots s_{\beta_{k-1}}\beta_k$ is not. In particular, $s_{\beta_{i+1}}\cdots s_{\beta_{k-1}}\beta_k\neq s_{\beta_i}\cdots s_{\beta_{k-1}}\beta_k$ and hence $\langle\beta_i,s_{\beta_{i+1}}\cdots s_{\beta_{k-1}}\beta_k\rangle\neq0$, which implies $\langle s_{\beta_1}\cdots s_{\beta_{i-1}}\beta_i,s_{\beta_1}\cdots s_{\beta_{k-1}}\beta_k\rangle\neq0$. Put $\alpha_2:=s_{\beta_1}\cdots s_{\beta_{i-1}}\beta_i$, then $\langle\alpha_1,\alpha_2\rangle\neq0$. Repeating this argument yields the desired sequence.
\end{proof}

The next lemma provides more information about case \eqref{lem:RelationAGandAHRootSpaces2} in Lemma~\ref{lem:RelationAGandAHRootSpaces}:

\begin{lemma}\label{lem:RemainingCase}
Let $w\in W$ with $Q\tilde{w}P_G\neq Q\tilde{w}_0P_G$. If for all $\alpha\in\Sigma^+(\frakg,\fraka_G)\cap w\Sigma^+(\frakg,\fraka_G)$ with $\overline{\alpha}\neq0$ we have $\frakg(\fraka_G;\alpha)\subseteq\frakn_H$, then for one of those $\alpha$ we must have $[\frakg(\fraka_G;\alpha),\frakn^{-\sigma}]\neq\{0\}$.
\end{lemma}

\begin{proof}
Let $w\in W$ with $Q\tilde{w}P_G\neq Q\tilde{w}_0P_G$ and assume that $\frakg(\fraka_G;\alpha)\subseteq\frakn_H$ for all $\alpha\in\Sigma^+(\frakg,\fraka_G)\cap w\Sigma^+(\frakg,\fraka_G)$ with $\overline{\alpha}\neq0$. We explicitly construct a root $\alpha\in\Sigma^+(\frakg,\fraka_G)\cap w\Sigma^+(\frakg,\fraka_G)$ with $\overline{\alpha}\neq0$ such that $[\frakg(\fraka_G;\alpha),\frakn^{-\sigma}]\neq\{0\}$.\\
\textbf{Step 1.} If $Q\tilde{w}P_G\neq Q\tilde{w}_0P_G$, then there exists $\alpha\in\Sigma^+(\frakg,\fraka_G)\cap w\Sigma^+(\frakg,\fraka_G)$ with $\overline{\alpha}\neq0$ by Lemma~\ref{lem:WeylGroupLemma}. We claim that there exists a simple root $\alpha$ with this property. In fact, by Lemma~\ref{lem:RootSequence} there exists a sequence $\alpha=\alpha_1,\ldots,\alpha_r\in\Sigma^+(\frakg,\fraka_G)\cap w\Sigma^+(\frakg,\fraka_G)$ with $\langle\alpha_i,\alpha_{i+1}\rangle\neq0$ for $i=1,\ldots,r-1$ and $\alpha_r$ simple. By our assumption, every $\beta\in\Sigma^+(\frakg,\fraka_G)\cap w\Sigma^+(\frakg,\fraka_G)$ satisfies either $\overline{\beta}=0$ or $\frakg(\fraka_G;\beta)\subseteq\frakn_H$. By Lemma~\ref{lem:RelationAGandAHRootSpaces} this implies $\sigma\beta=-\beta$ or $\sigma\beta=\beta$. These two types of roots are obviously orthogonal to each other. Hence, $\sigma\alpha=\alpha$ implies $\sigma\alpha_i=\alpha_i$ for all roots $\alpha_i$ in the sequence. In particular, $\frakg(\fraka_G;\alpha_r)\subseteq\frakn_H$ by Lemma~\ref{lem:RelationAGandAHRootSpaces}, and we can replace $\alpha$ by the simple root $\alpha_r$.\\
\textbf{Step 2.} Next, we claim that there exists a simple root $\alpha'\in\Sigma^+(\frakg,\fraka_G)$ with $\overline{\alpha'}\neq0$ and $\frakg(\fraka_G;\alpha')\not\subseteq\frakn_H$. Assume that such a simple root does not exist. Then for every simple root $\alpha'\in\Sigma^+(\frakg,\fraka_G)$ we have either $\overline{\alpha'}=0$ or $\frakg(\fraka_G;\alpha')\subseteq\frakn_H$. This implies that either $\sigma\alpha'=-\alpha'$ or $\sigma\alpha'=\alpha'$, so that the set of simple roots is the disjoint union of the two mutually orthogonal subsets $\{\alpha'\in\Sigma^+(\frakg,\fraka_G)\mbox{ simple, }\sigma\alpha'=\pm\alpha'\}$. First note that this cannot occur in the case $(\frakg,\frakh)=(\frakg'+\frakg',\diag\frakg')$, so that we may assume $\frakg$ to be simple. Then the Dynkin diagram of $\Sigma(\frakg,\fraka_G)$ is connected and we must have $\sigma\alpha'=\alpha'$ for all simple roots, whence $\frakg(\fraka_G;\alpha')\subseteq\frakn_H$ for all simple roots. This implies $\frakg(\fraka_G;\alpha')\subseteq\frakn_H$ for all positive roots and therefore $\frakn_G=\frakn_H$ and also $\overline{\frakn}_G=\overline{\frakn}_H$. But $\frakn_G$ and $\overline{\frakn}_G$ generate $\frakg$, whence $\frakg=\frakh$ which contradicts our assumption that $(\frakg,\frakh)$ is non-trivial.\\
\textbf{Step 3.} By Step 1 we find a simple root $\alpha\in\Sigma^+(\frakg,\fraka_G)\cap w\Sigma^+(\frakg,\fraka_G)$ with $\overline{\alpha}\neq0$ and hence $\frakg(\fraka_G;\alpha)\subseteq\frakn_H$. We claim that $[\frakg(\fraka_G;\alpha),\frakn^{-\sigma}]\neq\{0\}$. To see this, we use Step 2 to find another simple root $\alpha'\in\Sigma^+(\frakg,\fraka_G)$ with $\overline{\alpha'}\neq0$ and $\frakg(\fraka_G;\alpha')\not\subseteq\frakn_H$. Connecting $\alpha$ and $\alpha'$ in the Dynkin diagram for $\Sigma(\frakg,\fraka_G)$, we obtain a sequence of simple roots $\alpha=\alpha_1,\alpha_2,\ldots,\alpha_{p-1},\alpha_p=\alpha'$ such that $(\alpha_i,\alpha_j)\neq0$ if and only if $|i-j|\leq1$. By possibly replacing $\alpha_p$ by one of the $\alpha_i$ we may assume that for all $1\leq i\leq p-1$ we either have $\overline{\alpha_i}=0$ or $\frakg(\fraka_G;\alpha_i)\subseteq\frakn_H$. We now construct a root $\alpha''=n_2\alpha_2+\cdots+n_p\alpha_p$, $n_i\geq1$, such that $\overline{\alpha''}\neq0$ and $\frakg(\fraka_G;\alpha'')\not\subseteq\frakn_H$, then $[\frakg(\fraka_G;\alpha),\frakg(\fraka_G;\alpha'')]\neq\{0\}$ by Lemma~\ref{lem:OSLemma} since $(\alpha,\alpha'')=n_2(\alpha_1,\alpha_2)<0$. By Lemma~\ref{lem:RelationAGandAHRootSpaces} this implies $[\frakg(\fraka_G;\alpha),\frakn^{-\sigma}]\neq\{0\}$.\\
We inductively construct a root $\beta_k=n_k\alpha_k+\cdots+n_p\alpha_p$ ($2\leq k\leq p$) with $\overline{\beta_k}\neq0$ and $\frakg(\fraka_G;\beta_k)\not\subseteq\frakn_H$. Note that for $\beta_k$ of the above form we always have $\overline{\beta_k}=n_k\overline{\alpha_k}+\cdots+n_p\overline{\alpha_p}\neq0$ since $\overline{\alpha_i}$ is either $=0$ or a positive root, and by assumption $\overline{\alpha_p}\neq0$. For $k=p$ we can choose the simple root $\beta_p=\alpha_p$. Now assume $\beta_{k+1}=n_{k+1}\alpha_{k+1}+\cdots+n_p\alpha_p$ has been constructed with $\frakg(\fraka_G;\beta_{k+1})\not\subseteq\frakn_H$. Then by Lemma~\ref{lem:RelationAGandAHRootSpaces} there are four possibilities for $\beta_{k+1}$ and $\alpha_k$:
\begin{enumerate}
\item $\sigma\beta_{k+1}=\beta_{k+1}$ and $\overline{\alpha_k}\neq0$. Then $\frakg(\fraka_G;\beta_{k+1})\cap\frakn^{-\sigma}\neq\{0\}$ and $\frakg(\fraka_G;\alpha_k)\subseteq\frakn_H$, hence by Lemma~\ref{lem:OSLemma}
$$ \{0\} \neq [\frakg(\fraka_G;\alpha_k),\frakg(\fraka_G;\beta_{k+1})\cap\frakn^{-\sigma}] \subseteq \frakg(\fraka_G;\alpha_k+\beta_{k+1})\cap\frakn^{-\sigma}, $$
so that $\beta_k=\alpha_k+\beta_{k+1}$ satisfies $\frakg(\fraka_G;\beta_k)\not\subseteq\frakn_H$.
\item $\sigma\beta_{k+1}=\beta_{k+1}$ and $\overline{\alpha_k}=0$. Then $\frakg(\fraka_G;\beta_{k+1})\cap\frakn^{-\sigma}\neq\{0\}$ and $\sigma\alpha_k=-\alpha_k$. Hence, we have for any $X\in\frakg(\fraka_G;\beta_{k+1})\cap\frakn^{-\sigma}$ and $Y\in\frakg(\fraka_G;\alpha_k)$, $X,Y\neq0$, that $0\neq[X,Y]\in\frakg(\fraka_G;\alpha_k+\beta_{k+1})$ by Lemma~\ref{lem:OSLemma} and $\sigma[X,Y]=-[X,\sigma Y]\in\frakg(\fraka_G;-\alpha_k+\beta_{k+1})$. Therefore, $\sigma[X,Y]\neq[X,Y]$ and hence $\frakg(\fraka_G;\alpha_k+\beta_{k+1})\not\subseteq\frakn_H$, so that we can choose $\beta_k=\alpha_k+\beta_{k+1}$.
\item $\sigma\beta_{k+1}\neq\beta_{k+1}$ and $\overline{\alpha_k}\neq0$. Then $\{X-\sigma X:X\in\frakg(\fraka_G;\beta_{k+1})\}\subseteq\frakn^{-\sigma}$ and $\frakg(\fraka_G;\alpha_k)\subseteq\frakn_H$. Hence, we have for any $X\in\frakg(\fraka_G;\beta_{k+1})$ and $Y\in\frakg(\fraka_G;\alpha_k)$, $X,Y\neq0$, that $0\neq[X,Y]\in\frakg(\fraka_G;\alpha_k+\beta_{k+1})$ by Lemma~\ref{lem:OSLemma} and $\sigma[X,Y]=[\sigma X,Y]\in\frakg(\fraka_G;\alpha_k+\sigma\beta_{k+1})$. Therefore, $\sigma[X,Y]\neq[X,Y]$ and hence $\frakg(\fraka_G;\alpha_k+\beta_{k+1})\not\subseteq\frakn_H$, so that we can choose $\beta_k=\alpha_k+\beta_{k+1}$.
\item $\sigma\beta_{k+1}\neq\beta_{k+1}$ and $\overline{\alpha_k}=0$. Then $\{X-\sigma X:X\in\frakg(\fraka_G;\beta_{k+1})\}\subseteq\frakn^{-\sigma}$ and $\sigma\alpha_k=-\alpha_k$. Let $X\in\frakg(\fraka_G;\beta_{k+1})$ and $Y\in\frakg(\fraka_G;\alpha_k)$, $X,Y\neq0$; then by Lemma~\ref{lem:OSLemma} we have $[X,Y]\neq0$.
\begin{enumerate}
\item If $\sigma[X,Y]\neq[X,Y]$ then we can choose $\beta_k=\alpha_k+\beta_{k+1}$ as in (1), (2) and (3).
\item If $\sigma[X,Y]=[X,Y]$ then $\sigma(\alpha_k+\beta_{k+1})=\alpha_k+\beta_{k+1}$ so that $\sigma\beta_{k+1}=2\alpha_k+\beta_{k+1}$. In this case we choose $\beta_k=\sigma\beta_{k+1}=2\alpha_k+n_{k+1}\alpha_{k+1}+\cdots+n_p\alpha_p$ which is clearly a positive root. Further, $\frakg(\fraka_G;\beta_k)\not\subseteq\frakn_H$ since $\sigma\beta_k=\beta_{k+1}\neq\beta_k$.
\end{enumerate}
\end{enumerate}
Inductively, for $k=2$ this produces the desired root $\alpha''=\beta_2$.
\end{proof}

We can finally finish the proof of Theorem~\ref{thm:AHProjectionStabilizers} by choosing $Y\in\frakg(\fraka_G;\alpha)$ according to the three cases in Lemma~\ref{lem:RelationAGandAHRootSpaces}. Write $\beta=\overline{\alpha}$.
\begin{enumerate}
\item If $\sigma\alpha\neq\alpha$ we can write $X_\beta=Y-\sigma Y$ for some $Y\in\frakg(\fraka_G;\alpha)$. Using this $Y$ in the above construction, we have by Lemma~\ref{lem:IdentityForYandZ}
$$ Y = \ad(Z_M+Z_A)X_\beta + Z_{N,\beta}. $$
Both sides are contained in $\frakn=\frakn^\sigma\oplus\frakn^{-\sigma}$, and taking $\frakn^{-\sigma}$-components gives $2X_\beta=\ad(Z_M+Z_A)X_\beta=\ad(Z_M)X_\beta+\beta(Z_A)X_\beta$. Note that we assume $X_\beta\neq0$ by the reduction in Section~\ref{sec:ReductionToGenericX}. Then $\ad(Z_M)X_\beta=(2-\beta(Z_A))X_\beta$ so that $Z_M$ acts by a scalar on $X_\beta$. Since $Z_M\in\frakm_G\cap\frakm_H$ and $M_G\cap M_H$ is compact, this scalar has to be imaginary, so that $\beta(Z_A)=2$. In particular, $Z_A\neq0$.
\item If $\sigma\alpha=\alpha$ and $\frakg^{-\sigma}(\fraka_G;\alpha)=\frakg^{-\sigma}(\fraka_H;\beta)$ we take $Y=X_\beta$. Then Lemma~\ref{lem:IdentityForYandZ} implies
$$ X_\beta = Y = \ad(Z_M+Z_A)X_\beta = \ad(Z_M)X_\beta + \beta(Z_A)X_\beta. $$
The same argument as in (1) shows $\beta(Z_A)=1$ and in particular $Z_A\neq0$.
\item If (1) and (2) do not hold for any such $\alpha$, we have $\sigma\alpha=\alpha$ and $\frakg(\fraka_G;\alpha)\subseteq\frakn_H$ for all $\alpha>0$ with $w^{-1}\alpha>0$, $\overline{\alpha}\neq0$. Then Lemma~\ref{lem:RemainingCase} implies that there exists an $\alpha$ with $[\frakg(\fraka_G;\alpha),\frakn^{-\sigma}]\neq\{0\}$. Let $0\neq Y\in\frakg(\fraka_G;\alpha)$ and $0\neq B_i\in\frakg^{-\sigma}(\fraka_H;\beta_i)$ ($i=1,2$) with $\beta_1,\beta_2\in\Delta(\frakn^{-\sigma})$, such that $[Y,B_1]=B_2$. By Lemma~\ref{lem:RealSphImpliesTransitiveActionOnSpheres} there exists $ma\in(M_G\cap M_H)A_H$ such that $\Ad(ma)B_i=\pm X_{\beta_i}$ ($i=1,2$). Hence, $[\Ad(ma)Y,X_{\beta_1}]=\pm X_{\beta_2}$. Replacing $Y$ by $\Ad(ma)Y\in\frakg(\fraka_G;\alpha)$ we may therefore assume that $[Y,X_{\beta_1}]=\pm X_{\beta_2}$.\\
Note that
$$ e^Xe^{tY} = e^{tY}e^{\Ad(e^{-tY})X}\in N_H\exp(\frakn^{-\sigma}), $$
and $\Ad(e^{-tY})X=e^{-t\ad(Y)}X\in\frakn^{-\sigma}$ so that $n_t=e^{tY}$ and $X_t=\Ad(m_ta_t)X=e^{-t\ad(Y)}X$. Hence, $Z_N=Y$ and $[Z_M+Z_A,X]=-[Y,X]$. Taking the $\beta_2$-component gives
$$ [Z_M,X_{\beta_2}]+\beta_2(Z_A)X_{\beta_2} = -[Y,X_{\beta_1}] = \mp X_{\beta_2}. $$
By the same argument as in (1) and (2) we find $\beta_2(Z_A)=\mp1$ and hence $Z_A\neq0$.
\end{enumerate}
This finishes the proof of Theorem~\ref{thm:AHProjectionStabilizers}.\qed

\section{Spherical matrix coefficients}\label{sec:FiniteDimBranching}

We study matrix coefficients of finite-dimensional representations of $G$ which are equivariant under the action of $P_H\times P_G$ by left and right multiplication. Such matrix coefficients correspond to finite-dimensional spherical representations of $G$ whose restriction to $H$ contains a spherical representation, and we show that there exist \textit{enough} such representations (see Theorem~\ref{thm:EnoughWeights}). In Section~\ref{sec:ConstructionOfSBOs} these matrix coefficients are used to explicitly construct symmetry breaking operators.

\subsection{Reduction to complex connected groups}

Since both $G$ and $H$ might be disconnected, their finite-dimensional representations are not easily described in terms of highest weights. To overcome this difficulty we first reduce the construction of matrix coefficients to the case of complex connected groups. Since $G$ is of Harish-Chandra class, there exist
\begin{itemize}
\item a complex connected linear reductive group $\GG_\CC$ with Lie algebra $\frakg_\CC$,
\item an antiholomorphic involution $\tau:\GG_\CC\to\GG_\CC$ such that the derived involution $\tau:\frakg_\CC\to\frakg_\CC$ is the conjugation with respect to the real form $\frakg$,
\item a homomorphism $\mu:G\to\GG$ from $G$ to the real form $\GG=\GG_\CC^\tau$ of $\GG_\CC$ with finite kernel and cokernel.
\end{itemize}
Note that the Lie algebra of $\GG$ is equal to $\frakg$. We denote by $\HH_\CC$ the complex connected subgroup of $\GG_\CC$ with Lie algebra $\frakh_\CC$ and by $\HH=\mu(H)_0=\mu(H_0)\subseteq\HH_\CC$ the connected subgroup of $\HH_\CC$ with Lie algebra $\frakh$. Then the finite-dimensional holomorphic representations of $\GG_\CC$ and $\HH_\CC$ are classified in terms of their highest weights, and via the homomorphism $\mu$ they give rise to finite-dimensional representations of $G$ and $H_0$.

The image $\KK=\mu(K)$ of the maximal compact subgroup of $G$ under $\mu$ is a maximal compact subgroup of $\GG$, and the intersection $\HH\cap\KK$ is maximal compact in $\HH$. Further, let
$$ \MM_G = \mu(M_G), \quad \AA_G = \mu(A_G) \quad \mbox{and} \quad \NN_G = \mu(N_G), $$
then $\PP_G=\mu(P_G)=\MM_G\AA_G\NN_G$ is a minimal parabolic subgroup of $\GG$.

\subsection{The Cartan--Helgason Theorem}

We now recall the classification of irreducible finite-dimensional spherical representations of $\GG$ in terms of their highest weights, the so-called Cartan--Helgason Theorem. Recall that a representation of $\GG$ is called spherical if it contains a non-zero $\KK$-invariant vector.

We choose a maximal abelian subalgebra $\frakt_G$ in $\frakm_G$; then $\frakj_G=\frakt_G\oplus\fraka_G$ is a Cartan subalgebra of $\frakg$ and $\frakj_{G,\CC}$ is a Cartan subalgebra of $\frakg_\CC$. Roots in $\Sigma(\frakg_\CC,\frakj_{G,\CC})$ are real on $\fraka_G$ and imaginary on $\frakt_G$. Fix a positive system $\Sigma^+(\frakg_\CC,\frakj_{G,\CC})$ such that the non-zero restrictions to $\fraka_G$ are contained in $\Sigma^+(\frakg,\fraka_G)$. With respect to this data, the irreducible finite-dimensional representations of $\GG$ are classified by their highest weights in $\frakj_{G,\CC}^\vee$.

Recall the following theorem (see e.g. \cite[Theorem 8.49]{Kna02}):

\begin{theorem}[Cartan--Helgason]\label{thm:CartanHelgason}
For an irreducible finite-dimensional representation $\varphi$ of $\GG$ the following statements are equivalent:
\begin{enumerate}
\item $\varphi$ has a non-zero $\KK$-fixed vector.
\item $\MM_G$ acts by the $1$-dimensional trivial representation in the highest restricted weight space of $\varphi$.
\item The highest weight of $\varphi$ vanishes on $\frakt_G$, and its restriction to $\fraka_G$ is contained in the set
$$ \Lambda^+(\frakg,\fraka_G) = \{\lambda\in\fraka_G^\vee:\langle\lambda,\alpha\rangle/|\alpha|^2\in\NN\,\forall\,\alpha\in\Sigma^+(\frakg,\fraka_G)\}. $$
\end{enumerate}
\end{theorem}

Let $\Lambda^+_\GG(\frakg,\fraka_G)\subseteq\Lambda^+(\frakg,\fraka_G)$ denote the subset of all $\lambda$ for which there exists an irreducible finite-dimensional $\GG$-representation $(\varphi_\lambda,V_\lambda)$ of highest weight $\lambda$. Then $\linspan_\RR\Lambda^+_\GG(\frakg,\fraka_G)=\linspan_\RR\Lambda^+(\frakg,\fraka_G)=\fraka_G^\vee$, and if $\GG$ is simply connected even $\Lambda^+_\GG(\frakg,\fraka_G)=\Lambda^+(\frakg,\fraka_G)$. Theorem~\ref{thm:CartanHelgason} immediately gives the action of $\PP_G$ on the highest weight space of $V_\lambda$:

\begin{corollary}\label{cor:CartanHelgason}
For every $\lambda\in\Lambda^+_\GG(\frakg,\fraka_G)$ the minimal parabolic subgroup $\PP_G=\MM_G\AA_G\NN_G$ acts on the highest weight space of $V_\lambda$ by the character $\1\otimes e^\lambda\otimes\1$.
\end{corollary}

Here $\1$ denotes the trivial representation of $\MM_G$ and $\NN_G$, respectively, and $e^\lambda$ is the character of $\AA_G=\exp(\fraka_G)$ given by $e^\lambda(e^X)=e^{\lambda(X)}$, $X\in\fraka_G$.

Now, for an irreducible finite-dimensional representation $(\varphi,V)$ of $\GG$ we denote by $(\varphi^\vee,V^\vee)$ the contragredient representation on the dual space $V^\vee=\Hom_\CC(V,\CC)$ given by
$$ \langle\varphi^\vee(g)\phi,v\rangle = \langle\phi,\varphi(g^{-1})v\rangle, \qquad g\in\GG,v\in V,\phi\in V^\vee. $$
The following statement is standard:

\begin{lemma}\label{lem:DualCartanHelgason}
\begin{enumerate}
\item $(\varphi,V)$ has a non-zero $\KK$-fixed vector if and only if $(\varphi^\vee,V^\vee)$ has a non-zero $\KK$-fixed vector.
\item For $\lambda\in\Lambda^+_\GG(\frakg,\fraka_G)$ let $\lambda^\vee\in\Lambda^+_\GG(\frakg,\fraka_G)$ be defined by $(\varphi_\lambda^\vee,V_\lambda^\vee)\simeq(\varphi_{\lambda^\vee},V_{\lambda^\vee})$, then the map
$$ \Lambda^+_\GG(\frakg,\fraka_G)\to\Lambda^+_\GG(\frakg,\fraka_G), \quad \lambda\mapsto\lambda^\vee $$
is the restriction to $\Lambda^+_\GG(\frakg,\fraka_G)$ of a linear map $\fraka_G^\vee\to\fraka_G^\vee$.
\end{enumerate}
\end{lemma}

\subsection{Matrix coefficients}\label{sec:MatrixCoefficients}

As for $\GG$, we denote by $\Lambda^+_\HH(\frakh,\fraka_H)\subseteq\fraka_H^\vee$ the set of all highest weights $\nu$ of irreducible finite-dimensional representations $(\psi_\nu,W_\nu)$ of $\HH$ which have a one-dimensional $\PP_H$-invariant subspace isomorphic to $\1\otimes e^\nu\otimes\1$, where $\PP_H=\MM_H\AA_H\NN_H$ is the corresponding minimal parabolic subgroup of $\HH$.

\begin{lemma}\label{lem:MatrixCoefficients}
Let $\lambda\in\Lambda_\GG^+(\frakg,\fraka_G)$ and $\nu\in\Lambda_\HH^+(\frakh,\fraka_H)$ and pick non-zero highest weight vectors $v_0\in V_\lambda$ and $\phi_0\in W_\nu^\vee$. Then for every $0\neq\eta\in\Hom_\HH(\varphi_\lambda|_\HH,\psi_\nu)$ the function
$$ f:\GG\to\CC, \quad f(g) = \langle\phi_0,\eta(\varphi_\lambda(g)v_0)\rangle $$
is non-zero, real-analytic and satisfies
$$ f(m'a'n'gman) = a^\lambda(a')^{-\nu^\vee}f(g) $$
for $g\in\GG$, $man\in\PP_G$ and $m'a'n'\in\PP_H$.
\end{lemma}

\begin{proof}
It is clear that $f$ is real analytic as a matrix coefficient of a finite-dimensional representation. Further, $f$ is non-zero since $(\varphi_\lambda,V_\lambda)$ is irreducible. By Corollary~\ref{cor:CartanHelgason} and Lemma~\ref{lem:DualCartanHelgason} we have $\varphi_\lambda(man)v_0=a^\lambda v_0$ and $\psi_\nu^\vee(m'a'n')\phi_0=(a')^{\nu^\vee}\phi_0$ and the claim follows.
\end{proof}

Note that since $(\varphi_\lambda,V_\lambda)$ and $(\psi_\nu,W_\nu)$ extend to holomorphic representations of the complex connected groups $\GG_\CC$ and $\HH_\CC$, we have $\Hom_\HH(\varphi_\lambda|_\HH,\psi_\nu)=\Hom_\frakh(\varphi_\lambda|_\frakh,\psi_\nu)$. Abusing notation, we also write $(\varphi_\lambda,V_\lambda)$ and $(\psi_\nu,W_\nu)$ for the Lie algebra representations of $\frakg$ and $\frakh$ for arbitrary $(\lambda,\nu)\in\Lambda^+(\frakg,\fraka_H)$ and $\nu\in\Lambda^+(\frakh,\fraka_H)$.

To obtain matrix coefficients with the same properties as in Lemma~\ref{lem:MatrixCoefficients}, but for the pair $(G,H)$ instead of $(\GG,\HH)$, we use the homomorphism $\mu:G\to\GG$.

\begin{proposition}\label{prop:MatrixCoefficients}
Assume that $(G,H)$ is a strongly spherical reductive pair. Then for each pair $(\lambda,\nu)\in\Lambda^+(\frakg,\fraka_G)\times\Lambda^+(\frakh,\fraka_H)$ with $\Hom_\frakh(\varphi_\lambda|_\frakh,\psi_\nu)\neq\{0\}$ there exists $k\geq1$ and a non-zero real-analytic function $F:G\to\RR$, $F\geq0$, satisfying
\begin{equation}
 F(m'a'n'gman) = a^{k\lambda}(a')^{-k\nu^\vee}F(g)\label{eq:EquivarianceMatrixCoefficients}
\end{equation}
for $g\in G$, $man\in P_G$ and $m'a'n'\in P_H$.
\end{proposition}

\begin{proof}
First note that there exists $k\geq1$ such that $k\lambda\in\Lambda^+_\GG(\frakg,\fraka_G)$ and $k\nu\in\Lambda^+_\HH(\frakh,\fraka_H)$. Then also $\Hom_\frakh(\varphi_{k\lambda}|_\frakh,\psi_{k\nu})\neq\{0\}$ and by Lemma~\ref{lem:MatrixCoefficients} there exists a non-zero real-analytic function $f:\GG\to\CC$ with
$$ f(m'a'n'gman) = a^{k\lambda}(a')^{-k\nu^\vee}f(g) $$
for $g\in\GG$, $man\in\PP_G$ and $m'a'n'\in\PP_H$. Replacing $f$ by $|f|^2$ we may further assume that $f:\GG\to\RR$ and $f\geq0$. We consider the non-zero real-analytic function $f\circ\mu:G\to\RR$ which satisfies \eqref{eq:EquivarianceMatrixCoefficients} at least for $m'\in M_{H,0}\subseteq\mu^{-1}(\MM_H)\subseteq M_H$. Since the component group $M_H/M_{H,0}$ of $M_H$ is finite, we can form the finite sum
$$ F(g) = \sum_{mM_{H,0}\in M_H/M_{H,0}} f(\mu(mg)), $$
and this clearly defines a real-analytic function $F:G\to\RR$, $F\geq0$, with the equivariance property \eqref{eq:EquivarianceMatrixCoefficients}. Finally, $F$ is non-zero since $f\circ\mu\geq0$ and $f\circ\mu\neq0$.
\end{proof}

The main result of this section asserts that for all strongly spherical reductive pairs $(G,H)$ there exist \textit{enough} pairs $(\lambda,\nu)\in\fraka_G^\vee\times\fraka_H^\vee$ with $\Hom_\frakh(\varphi_\lambda|_\frakh,\psi_\nu)\neq\{0\}$:

\begin{theorem}\label{thm:EnoughWeights}
Assume that $(G,H)$ is a strongly spherical reductive pair such that $(\frakg,\frakh)$ is non-trivial and indecomposable. Then the set of pairs $(\lambda,\nu)\in\Lambda^+(\frakg,\fraka_G)\times\Lambda^+(\frakh,\fraka_H)$ such that $\Hom_\frakh(\varphi_\lambda|_\frakh,\psi_\nu)\neq\{0\}$ spans $\fraka_G^\vee\times\fraka_H^\vee$.
\end{theorem}

\begin{remark}\label{rem:LocalStructureTheorem}
One can use the local structure theorem for real spherical varieties by Knop--Kr\"{o}tz--Schlichtkrull~\cite{KKS15} to give a classification-free proof of Theorem~\ref{thm:EnoughWeights}. In fact, let $Z=(G\times H)/\diag(H)$; then using \cite[Theorem 2.8]{KKS15} it is easy to see that for $X\in\frakn^{-\sigma}$ contained in an open $(M_G\cap M_H)A_H$-orbit, the minimal parabolic subgroup $P=e^{-X}\overline{P}_Ge^X\times P_H$ is $Z$-adapted in the sense of \cite[Definition 2.7]{KKS15}. Further, the Levi subgroup $L=e^{-X}M_GA_Ge^X\times M_HA_H$ of $P$ satisfies $L\cap\diag(H)=\diag(M)$, where $M=(M_G\cap M_H)^X$ is the stabilizer of $X$. This implies $\fraka_Z=\Ad(e^{-X})\fraka_G\times\fraka_H$ in the notation of \cite[Section 2.3]{KKS15} and therefore $\rank(Z)=\dim\fraka_Z=\dim\fraka_G+\dim\fraka_H$. By \cite[Remark 3.5]{KKS15} the statement of Theorem~\ref{thm:EnoughWeights} follows.\\
However, since the explicit form of the integral kernels of symmetry breaking operators plays an important role in the classification of symmetry breaking operators (see e.g. \cite{Cle16a,Cle16b,KS15}), we prove Theorem~\ref{thm:EnoughWeights} using the classification of strongly spherical reductive pairs. From this one can explicitly determine the matrix coefficients which serve as building blocks for the integral kernels of symmetry breaking operators.
\end{remark}

Before we come to the proof of this result, let us see how it can be used to show the implication (1)$\Rightarrow$(2) in Theorem~\ref{thm:AHProjectionStabilizers}:

\begin{corollary}\label{cor:OpenDoubleCosetImpliesTrivialAHStabilizer}
If the double coset $P_HgP_G$ is open, then the projection of $\frakp_H\cap\Ad(g)\frakp_G$ to $\fraka_H$ is trivial.
\end{corollary}

\begin{proof}
Let $P_HgP_G$ be an open double coset and $Z=Z_M+Z_A+Z_N\in\frakp_H\cap\Ad(g)\frakp_G$. Put $X=\Ad(g)^{-1}Z=X_M+X_A+X_N\in\frakp_G$. For any $(\lambda,\nu)\in\Lambda^+(\frakg,\fraka_G)\times\Lambda^+(\frakh,\fraka_H)$ with $\Hom_\frakh(\varphi_\lambda|_\frakh,\psi_\nu)\neq\{0\}$ let $F$ be as in Proposition~\ref{prop:MatrixCoefficients}. Since $F$ is non-zero and real-analytic, it is non-zero on the open set $P_HgP_G$. In particular, $F(g)\neq0$ and for all $t\in\RR$ we have
$$ e^{-t\nu^\vee(Z_A)}F(g) = F(e^{tZ}g) = F(ge^{tX}) = e^{t\lambda(X_A)}F(g), $$
so that $\lambda(X_A)+\nu^\vee(Z_A)=0$. Since the pairs $(\lambda,\nu)\in\Lambda^+(\frakg,\fraka_G)\times\Lambda^+(\frakh,\fraka_H)$ satisfying $\Hom_\frakh(\varphi_\lambda|_\frakh,\psi_\nu)\neq\{0\}$ span $\fraka_G^\vee\times\fraka_H^\vee$, this implies $X_A=0$ and $Z_A=0$ and the proof is complete.
\end{proof}

Note that the statement in Theorem~\ref{thm:EnoughWeights} only depends on the pair of Lie algebras $(\frakg,\frakh)$. If we define
$$ \Lambda(\frakg,\frakh) = \{(\lambda,\nu)\in\Lambda^+(\frakg,\fraka_G)\times\Lambda^+(\frakh,\fraka_H):\Hom_\frakh(\varphi_\lambda|_\frakh,\psi_\nu)\neq\{0\}\}, $$
then we have to show that $\Lambda(\frakg,\frakh)$ spans $\fraka_G^\vee\times\fraka_H^\vee$. Note that $\Lambda(\frakg,\frakh)$ is a subsemigroup of $\Lambda^+(\frakg,\fraka_G)\times\Lambda^+(\frakh,\fraka_H)$. We first reduce Theorem~\ref{thm:EnoughWeights} to the case of $\frakg$ semisimple and $(\frakg,\frakh)$ symmetric and then use the classification in Theorem~\ref{thm:ClassificationStronglySphericalPairs} to show the statement case-by-case.

\begin{lemma}
Assume that Theorem~\ref{thm:EnoughWeights} holds for $\frakg$ semisimple, then it holds for $\frakg$ reductive.
\end{lemma}

\begin{proof}
As in Section~\ref{sec:ClassificationStronglySphericalPairs} we write $\frakg=\frakg_{\rm n}\oplus\frakg_{\rm el}$ and $p:\frakg\to\frakg_{\rm n}$ for the canonical projection. Then $\fraka_G=\fraka_{G,\rm n}\oplus\fraka_{G,\rm el}$ with $\fraka_{G,\rm n}=\fraka_G\cap\frakg_{\rm n}$ and $\fraka_{G,\rm el}=\fraka_G\cap\frakg_{\rm el}$, and by Lemma~\ref{lem:ReductionToSemisimple}~(1) the restriction $p|_{\fraka_H}:\fraka_H\to\fraka_{G,\rm n}$ is injective. Further, by Lemma~\ref{lem:ReductionToSemisimple}~(2) the pair $(\frakg_{\rm n},p(\frakh))$ is also strongly spherical, so by assumption $\Lambda(\frakg_{\rm n},p(\frakh))$ spans $\fraka_{G,\rm n}^\vee\times p(\fraka_H)^\vee$. Now for every pair $(\lambda_0,\nu_0)\in\Lambda(\frakg_{\rm n},p(\frakh))$ the pair $(\lambda,\nu)\in\fraka_G^\vee\times\fraka_H^\vee$ with $\lambda|_{\fraka_{G,\rm n}}=\lambda_0$, $\lambda|_{\fraka_{G,\rm el}}=0$ and $\nu=\nu_0\circ p|_{\fraka_H}$ is contained in $\Lambda(\frakg,\frakh)$. Hence, the span of $\Lambda(\frakg,\frakh)$ contains at least $\fraka_{G,\rm n}^\vee\times\fraka_H^\vee$. Further, for every $\lambda_1\in\fraka_{G,\rm el}^\vee$ the representation $\varphi_\lambda$ is one-dimensional and hence its restriction is an $\frakh\cap\frakk$-spherical representation of $\frakh$ with highest weight $\nu_1\in\fraka_H^\vee$, so that $(\lambda_1,\nu_1)\in\Lambda(\frakg,\frakh)$. This shows that $\Lambda(\frakg,\frakh)$ indeed spans $\fraka_G^\vee\times\fraka_H^\vee$.
\end{proof}

\begin{lemma}
Assume that Theorem~\ref{thm:EnoughWeights} holds for $(\frakg,\frakh)$ symmetric, then it holds for $(\frakg,\frakh)$ reductive.
\end{lemma}

\begin{proof}
By the previous lemma we may assume that $\frakg$ is semisimple. Then, thanks to Corollary~\ref{cor:ReductiveContainedInSymmetric}, there exists a non-trivial involution $\sigma$ of $\frakg$ such that $\frakh\subseteq\frakg^\sigma$, and $\frakh$ and $\frakg^\sigma$ differ only in compact factors. Hence, $\Lambda^+(\frakh,\fraka_H)=\Lambda^+(\frakg^\sigma,\fraka_H)$ and $\Lambda(\frakg,\frakg^\sigma)\subseteq\Lambda(\frakg,\frakh)$.
\end{proof}

\subsection{Finite-dimensional branching}

We now prove Theorem~\ref{thm:EnoughWeights} case-by-case for all symmetric pairs in the classification in Section~\ref{sec:ClassificationStronglySphericalPairs}. For this we first fix some notation.

Let $\frakj_H\subseteq\frakh$ be a Cartan subalgebra of $\frakh$ and extend it to a Cartan subalgebra $\frakj_H\subseteq\frakj_G\subseteq\frakg$ of $\frakg$. Note that we no longer assume that $\fraka_G\subseteq\frakj_G$ and $\fraka_H\subseteq\frakj_H$. Choose a system of positive roots $\Sigma^+(\frakg_\CC,\frakj_{G,\CC})$ for $\frakg$ such that 
$$ \Sigma^+(\frakh_\CC,\frakj_{H,\CC}) = \{\alpha|_{\frakj_{H,\CC}}:\alpha\in\Sigma^+(\frakg_\CC,\frakj_{G,\CC})\}\cap\Sigma(\frakh_\CC,\frakj_{H,\CC}) $$
is a system of positive roots for $\frakh$. Denote by $\varpi_1,\ldots,\varpi_s$ the fundamental weights for $\frakg$ with respect to $\Sigma^+(\frakg_\CC,\frakj_{G,\CC})$ and by $\zeta_1,\ldots,\zeta_t$ the fundamental weights for $\frakh$ with respect to $\Sigma^+(\frakh_\CC,\frakj_{H,\CC})$. Then any dominant integral weight of $\frakg$ with respect to $\Sigma^+(\frakg_\CC,\frakj_{G,\CC})$ is of the form $\varpi=k_1\varpi_1+\cdots+k_s\varpi_s$, $k_1,\ldots,k_s\in\NN$, and we write $F^\frakg(\varpi)$ for the corresponding finite-dimensional representation of $\frakg$. Analogous notation is used for $\frakh$ and ideals of $\frakh$. To simplify some statements we further put $\zeta_0:=0$ so that $F^\frakh(\zeta_0)$ is the trivial representation of $\frakh$.

We make use of the Satake diagrams for $\frakg$ and $\frakh$ (see e.g. \cite[Chapter X, Appendix F]{Hel78} for details). From the Satake diagram the highest weights belonging to spherical representations can be read off. In fact, for every simple root $\alpha_i$ whose vertex in the Satake diagram is white and not linked to any other vertex by an arrow, the representations $F^\frakg(2k\varpi_i)$ ($k\in\NN$) are spherical. If the vertices of two simple roots $\alpha_i$ and $\alpha_j$ are white and linked by an arrow, then $2k(\varpi_i+\varpi_j)$ ($k\in\NN$) are highest weights of spherical representations. Moreover, if $F^\frakg(\varpi)$ and $F^\frakg(\varpi')$ are spherical, then $F^\frakg(\varpi+\varpi')$ is spherical. In many cases we compute the explicit branching for $F^\frakg(\varpi_i)$ resp. $F^\frakg(\varpi_i+\varpi_j)$ and then use the semigroup property of $\Lambda(\frakg,\frakh)$ to conclude that the spherical representation $F^\frakg(2\varpi_i)$ resp. $F^\frakg(2(\varpi_i+\varpi_j))$ contains a certain spherical $\frakh$-representation.

The following reduction from complex Lie algebras to split real forms allows to minimize the number of different cases:

\begin{lemma}\label{lem:SplitImpliesComplex}
Let $(\frakg,\frakh)$ be a reductive pair with $\frakg$ and $\frakh$ split. If the statement in Theorem~\ref{thm:EnoughWeights} holds for $(\frakg,\frakh)$, then it also holds for $(\frakg_\CC,\frakh_\CC)$ viewed as real Lie algebras.
\end{lemma}

\begin{proof}
Since $\frakg$ and $\frakh$ are split, we can choose $\frakj_G=\fraka_G$ and $\frakj_H=\fraka_H$. Let $\fraku=\frakk+i\frakp\subseteq\frakg_\CC$; then $\fraku$ is maximally compact in $\frakg_\CC$ with complement $\fraku^\perp=i\frakk+\frakp$. Further, $\fraka_{G,\CC}$ is a (real) Cartan subalgebra of $\frakg_\CC$ which splits into $\fraka_{G,\CC}=i\fraka_G+\fraka_G$ with $i\fraka_G\subseteq\fraku$ and $\fraka_G\subseteq\fraku^\perp$. By the Cartan--Helgason Theorem, the highest weights of $\fraku$-spherical representations of $\frakg_\CC$ vanish on $i\fraka_G$ and their restrictions to $\fraka_G$ are contained in $\Lambda^+(\frakg_\CC,\fraka_G)=\Lambda^+(\frakg,\fraka_G)$. If $(\varphi_\lambda,V_\lambda)$ denotes a $\frakk$-spherical representation of $\frakg$ with highest weight $\lambda\in\Lambda^+(\frakg,\fraka_G)$, then a $\fraku$-spherical representation of $\frakg_\CC$ with highest weight $2\lambda$ is given by $V_\lambda\otimes V_\lambda$ where $\frakg_\CC$ acts by
$$ X(v\otimes v') = (Xv)\otimes v'+v\otimes(\overline{X}v'), \qquad X\in\frakg_\CC,v,v'\in V_\lambda. $$
Now let $(\lambda,\nu)\in\Lambda^+(\frakg,\fraka_G)\times\Lambda^+(\frakh,\fraka_H)$. Then for any $A\in\Hom_\frakh(V_\lambda,W_\nu)$ we clearly have $A\otimes A\in\Hom_{\frakh_\CC}(V_\lambda\otimes V_\lambda,W_\nu\otimes W_\nu)$. Hence, $(\lambda,\nu)\in\Lambda(\frakg,\frakh)$ implies $(2\lambda,2\nu)\in\Lambda(\frakg_\CC,\frakh_\CC)$ and the claim follows.
\end{proof}

Finally, we prove Theorem~\ref{thm:EnoughWeights} case-by-case for all strongly spherical symmetric pairs in the classification of Theorem~\ref{thm:ClassificationStronglySphericalPairs}:

\subsubsection*{{\rm A)} Trivial case}

This case $\frakg=\frakh$ is by assumption excluded.

\subsubsection*{{\rm C)} Compact case}

Let $\frakg$ be the Lie algebra of a compact simple Lie group; then also $\frakh$ is the Lie algebra of a compact group and $\fraka_G=\fraka_H=\{0\}$ so that $\Lambda(\frakg,\frakh)=\fraka_G^\vee\times\fraka_H^\vee=\{0\}\times\{0\}$ holds trivially.

\subsubsection*{{\rm D)} Compact subgroup case}

Let $\frakh=\frakk$ be the Lie algebra of a maximal compact subgroup $K$ of a non-compact simple Lie group $G$ with Lie algebra $\frakg$. Then $\fraka_H=\{0\}$ and the only spherical representation of $\frakh$ is the trivial representation $W_0=\CC$. By definition, for every $\lambda\in\Lambda^+(\frakg,\fraka_G)$ the representation $V_\lambda$ contains a $\frakk$-fixed vector, hence also $V_\lambda^\vee$ contains a $\frakk$-fixed vector by Lemma~\ref{lem:DualCartanHelgason}. But $\frakk$-fixed vectors in $V_\lambda^\vee=\Hom_\CC(V_\lambda,\CC)$ are simply $\frakk$-equivariant homomorphisms from $V_\lambda$ to the trivial representation $W_0$ of $\frakk$. Thus, $\Lambda(\frakg,\frakh)=\Lambda^+(\frakg,\fraka_G)\times\{0\}$ spans $\fraka_G^\vee\times\{0\}=\fraka_G^\vee\times\fraka_H^\vee$.

\subsubsection*{{\rm E1)} $(\frakg,\frakh)=(\so(1,p+q),\so(1,p)+\so(q))$}

The Satake diagrams of $\frakg$ and $\so(1,p)\subseteq\frakh$ are
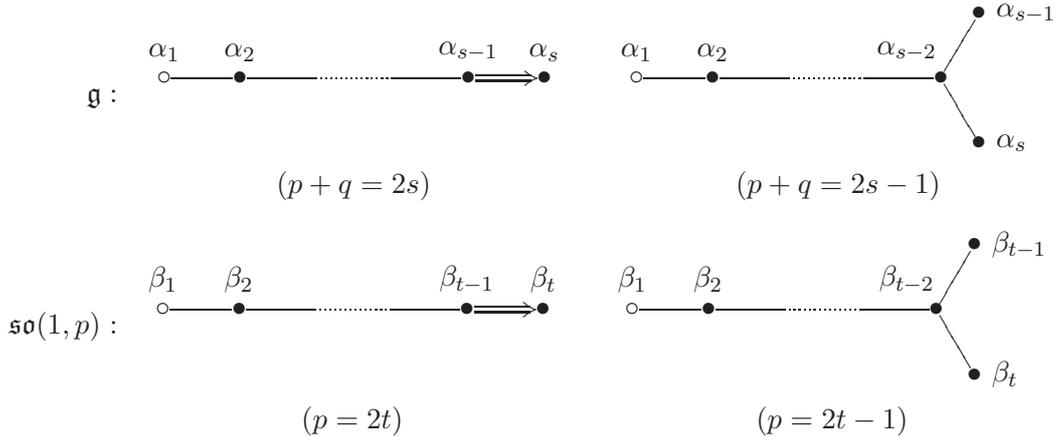
\begin{align*}
\frakg:&\begin{tabular}{cc}\begin{xy}
\ar@{-} (0,0) *++!D{\alpha_1} *{\circ}="A";
  (10,0) *++!D{\alpha_2}  *{\bullet}="B"
\ar@{-} "B"; (20,0)
\ar@{.} (20,0); (30,0) 
\ar@{-} (30,0); (40,0) *++!D{\alpha_{s-1}}  *{\bullet}="C"
\ar@{=>} "C"; (50,0) *++!D{\alpha_s}  *{\bullet}="D"
\end{xy}&\begin{xy}
\ar@{-} (0,0) *++!D{\alpha_1} *{\circ}="A"; 
 (10,0)  *++!D{\alpha_2} *{\bullet}="B"
\ar@{-} "B";  (20,0)
\ar@{.} (20,0); (30,0) 
\ar@{-} (30,0); (40,0) *++!D{\alpha_{s-2}\ \ \ \ \ \ \ } *{\bullet}="F"
\ar@{-} "F"; (45,8.6)  *++!L{\alpha_{s-1}} *{\bullet}
\ar@{-} "F"; (45,-8.6)  *++!L{\alpha_s} *{\bullet}
\end{xy}\\$(p+q=2s)$&$(p+q=2s-1)$\end{tabular}\\
\so(1,p):&\begin{tabular}{cc}\begin{xy}
\ar@{-} (0,0) *++!D{\beta_1} *{\circ}="A";
  (10,0) *++!D{\beta_2}  *{\bullet}="B"
\ar@{-} "B"; (20,0)
\ar@{.} (20,0); (30,0) 
\ar@{-} (30,0); (40,0) *++!D{\beta_{t-1}}  *{\bullet}="C"
\ar@{=>} "C"; (50,0) *++!D{\beta_t}  *{\bullet}="D"
\end{xy}&\begin{xy}
\ar@{-} (0,0) *++!D{\beta_1} *{\circ}="A"; 
 (10,0)  *++!D{\beta_2} *{\bullet}="B"
\ar@{-} "B";  (20,0)
\ar@{.} (20,0); (30,0) 
\ar@{-} (30,0); (40,0) *++!D{\beta_{t-2}\ \ \ \ \ \ } *{\bullet}="F"
\ar@{-} "F"; (45,8.6)  *++!L{\beta_{t-1}} *{\bullet}
\ar@{-} "F"; (45,-8.6)  *++!L{\beta_t} *{\bullet}
\end{xy}\\$(p=2t)$&$(p=2t-1)$\end{tabular}
\end{align*}

We realize the root system of $\frakg$ as $\{\pm e_i\pm e_j:1\leq i<j\leq s\}$ and additionally $\{\pm e_i:1\leq i\leq s\}$ if $p+q$ is even. Choose the simple roots $\alpha_i=e_i-e_{i+1}$ ($1\leq i\leq s-1$) and $\alpha_s=e_s$ for $p+q$ even and $\alpha_s=e_{s-1}+e_s$ for $p+q$ odd. We distinguish two cases:
\begin{itemize}
\item Assume first that $p+q$ is even or $p$ is odd; then $\frakj_G=\frakj_H$. If we choose the simple roots for $\so(1,p)\subseteq\frakh$ as $\beta_i=\alpha_i$ ($1\leq i\leq t-1$) and $\beta_t=e_t$ for $p$ even and $\beta_t=e_{t-1}+e_t$ for $p$ odd, then $\beta_i=\alpha_i$ ($t+1\leq i\leq s-1$) and $\beta_s=e_{s-1}+e_s$ for $q$ even and $\beta_s=e_s$ for $q$ odd are the simple roots for $\so(q)\subseteq\frakh$.
\item Assume now that $p+q$ is odd and $p$ is even; then we can choose $\frakj_H\subseteq\frakj_G$ such that $e_s|_{\frakj_{H,\CC}}=0$ and the simple roots for $\so(1,p)\subseteq\frakh$ are $\beta_i=\alpha_i|_{\frakj_{H,\CC}}$ ($1\leq i\leq t-1$) plus $\beta_t=e_t|_{\frakj_{H,\CC}}$, and the simple roots for $\so(q)\subseteq\frakh$ are $\beta_i=\alpha_i|_{\frakj_{H,\CC}}$ ($t+1\leq i\leq s-1$).
\end{itemize}

Consider the fundamental weight $\varpi_1=e_1$. Clearly $e_1|_{\frakj_{H,\CC}}=\zeta_1$ is also a highest weight for $\so(1,p)\subseteq\frakh$ and hence $(2\varpi_1,2\zeta_1)\in\Lambda(\frakg,\frakh)$. Further, from the Satake diagram for the real form $\so(p+1,q)$ of $\frakg_\CC\simeq\so(p+q+1,\CC)$, it follows that $F^\frakg(2\varpi_1)$ is $\frakh$-spherical and hence also $(2\varpi_1,0)\in\Lambda(\frakg,\frakh)$. Clearly $(2\varpi_1,2\zeta_1)$ and $(2\varpi_1,0)$ span $\fraka_G^\vee\times\fraka_H^\vee$, which is $2$-dimensional.

\subsubsection*{{\rm E2)} $(\frakg,\frakh)=(\su(1,p+q),\fraks(\fraku(1,p)+\fraku(q)))$}

The Satake diagrams of $\frakg$ and $\frakh$ are
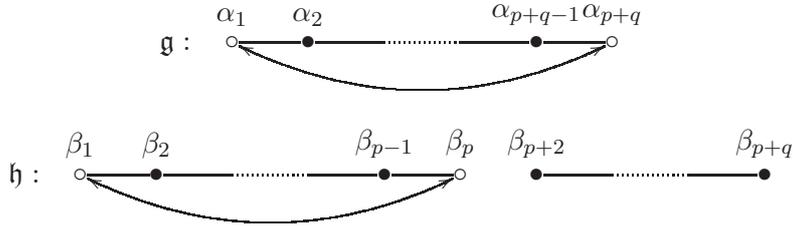
\[\frakg:\begin{xy}
\ar@{-} (0,0) *++!D{\alpha_1} *{\circ}="A";
  (10,0) *++!D{\alpha_2}  *{\bullet}="B"
\ar@{-} "B"; (20,0)
\ar@{.} (20,0); (30,0) 
\ar@{-} (30,0); (40,0) *++!D{\alpha_{p+q-1}}  *{\bullet}="C"
\ar@{-} "C"; (50,0) *++!D{\alpha_{p+q}}  *{\circ}="D"
\ar@/_1.5pc/@{<->} "A"; "D"
\end{xy}\]
\[\frakh:\begin{xy}
\ar@{-} (0,0) *++!D{\beta_1} *{\circ}="A";
  (10,0) *++!D{\beta_2}  *{\bullet}="B"
\ar@{-} "B"; (20,0)
\ar@{.} (20,0); (30,0) 
\ar@{-} (30,0); (40,0) *++!D{\beta_{p-1}}  *{\bullet}="C"
\ar@{-} "C"; (50,0) *++!D{\beta_p}  *{\circ}="D"
\ar@/_1.5pc/@{<->} "A"; "D"
\ar@{-} (60,0) *++!D{\beta_{p+2}} *{\bullet}="E"; (70,0)
\ar@{.} (70,0); (80,0) 
\ar@{-} (80,0); (90,0) *++!D{\beta_{p+q}}  *{\bullet}="F"
\end{xy}\]

We realize the root system of $\frakg$ as $\{\pm(e_i-e_j):1\leq i<j\leq p+q+1\}$ in the vector space $\{x\in\RR^{p+q+1}:x_1+\cdots+x_{p+q+1}=0\}$ and choose the simple roots $\alpha_i=e_i-e_{i+1}$ ($1\leq i\leq p+q$). If we choose the simple roots $\beta_i=\alpha_i$ ($1\leq i\leq p$) for $\su(1,p)\subseteq\frakh$, then $\beta_i=\alpha_i$ ($p+2\leq i\leq p+q$) are simple roots for $\su(q)\subseteq\frakh$ and the fundamental weight $\varpi_{p+1}$ describes a character of $\fraku(1)\subseteq\frakh$.

Consider the dominant integral weight $\varpi_1+\varpi_{p+q}=e_1-e_{p+q+1}$. In the Weyl group orbit of $e_1-e_{p+q+1}$ the weight $e_1-e_{p+1}=\zeta_1+\zeta_p$ is a highest weight for $\frakh$ and hence $(2(\varpi_1+\varpi_{p+q}),2(\zeta_1+\zeta_p))\in\Lambda(\frakg,\frakh)$. Further, from the Satake diagram for the real form $\su(p+1,q)$ of $\frakg_\CC\simeq\sl(p+q+1,\CC)$ it follows that $F^\frakg(2(\varpi_1+\varpi_{p+q}))$ is $\frakh$-spherical and hence also $(2(\varpi_1+\varpi_{p+q}),0)\in\Lambda(\frakg,\frakh)$. Clearly $(2(\varpi_1+\varpi_{p+q}),2(\zeta_1+\zeta_p))$ and $(2(\varpi_1+\varpi_{p+q}),0)$ span $\fraka_G^\vee\times\fraka_H^\vee$, which is $2$-dimensional.

\subsubsection*{{\rm E3)} $(\frakg,\frakh)=(\sp(1,p+q),\sp(1,p)+\sp(q)))$}

The Satake diagrams of $\frakg$ and $\frakh$ are
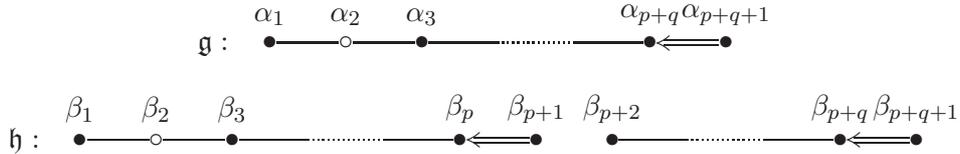
\[\frakg:\begin{xy}
\ar@{-} (0,0) *++!D{\alpha_1} *{\bullet}="A";
  (10,0) *++!D{\alpha_2}  *{\circ}="B"
\ar@{-} "B"; (20,0) *++!D{\alpha_3} *{\bullet}="C"
\ar@{-} "C"; (30,0)
\ar@{.} (30,0); (40,0)
\ar@{-} (40,0); (50,0) *++!D{\alpha_{p+q}}  *{\bullet}="D"
\ar@{<=} "D"; (60,0) *++!D{\alpha_{p+q+1}}  *{\bullet}="E"
\end{xy}\]
\[\frakh:\begin{xy}
\ar@{-} (0,0) *++!D{\beta_1} *{\bullet}="A";
  (10,0) *++!D{\beta_2}  *{\circ}="B"
\ar@{-} "B"; (20,0) *++!D{\beta_3} *{\bullet}="C"
\ar@{-} "C"; (30,0)
\ar@{.} (30,0); (40,0)
\ar@{-} (40,0); (50,0) *++!D{\beta_p}  *{\bullet}="D"
\ar@{<=} "D"; (60,0) *++!D{\beta_{p+1}}  *{\bullet}="E"
\ar@{-} (70,0) *++!D{\beta_{p+2}} *{\bullet}="F"; (80,0)
\ar@{.} "F"; (90,0)
\ar@{-} (90,0); (100,0) *++!D{\beta_{p+q}}  *{\bullet}="G"
\ar@{<=} "G"; (110,0) *++!D{\beta_{p+q+1}}  *{\bullet}="H"
\end{xy}\]

We realize the root system of $\frakg$ as $\{\pm e_i\pm e_j:1\leq i<j\leq p+q+1\}\cup\{\pm2e_i:1\leq i\leq p+q+1\}$ and choose the simple roots $\alpha_i=e_i-e_{i+1}$ ($1\leq i\leq p+q$) and $\alpha_{p+q+1}=2e_{p+q+1}$. If we choose the simple roots $\beta_i=\alpha_i$ ($1\leq i\leq p$) and $\beta_{p+1}=2e_{p+1}$ for $\sp(1,p)\subseteq\frakh$, then $\beta_i=\alpha_i$ ($p+2\leq i\leq p+q+1$) are simple roots for $\sp(q)\subseteq\frakh$.

Consider the fundamental weight $\varpi_2=e_1+e_2$. Clearly $e_1+e_2=\zeta_2$ is also a highest weight for $\sp(1,p)\subseteq\frakh$ and hence $(2\varpi_2,2\zeta_2)\in\Lambda(\frakg,\frakh)$. Further, from the Satake diagram for the real form $\sp(p+1,q)$ of $\frakg_\CC\simeq\sp(p+q+1,\CC)$, it follows that $F^\frakg(2\varpi_2)$ is $\frakh$-spherical and hence also $(2\varpi_2,0)\in\Lambda(\frakg,\frakh)$. Clearly $(2\varpi_2,2\zeta_2)$ and $(2\varpi_2,0)$ span $\fraka_G^\vee\times\fraka_H^\vee$, which is $2$-dimensional.

\subsubsection*{{\rm E4)} $(\frakg,\frakh)=(\frakf_{4(-20)},\so(8,1))$}

The Satake diagrams of $\frakg$ and $\frakh$ are
\[\frakg:\begin{xy}
\ar@{-} (0,0) *++!D{\alpha_1} *{\bullet}="A";
  (10,0) *++!D{\alpha_2}  *{\bullet}="B"
\ar@{=>} "B"; (20,0) *++!D{\alpha_3} *{\bullet}="C"
\ar@{-} (20,0); (30,0) *++!D{\alpha_4}  *{\circ}="D"
\end{xy}\]
\[\frakh:\begin{xy}
\ar@{-} (0,0) *++!D{\beta_1} *{\circ}="A";
  (10,0) *++!D{\beta_2}  *{\bullet}="B"
\ar@{-} (10,0); (20,0) *++!D{\beta_3}  *{\bullet}="C"
\ar@{=>} "C"; (30,0) *++!D{\beta_4}  *{\bullet}="D"
\end{xy}\]

By \cite{McKP81} we have
$$ F^\frakg(\varpi_4)|_\frakh \simeq F^\frakh(\zeta_1)\oplus F^\frakh(\zeta_4)\oplus F^\frakh(0) $$
and hence $(2\varpi_4,2\zeta_1),(2\varpi_4,0)\in\Lambda(\frakg,\frakh)$ span $\fraka_G^\vee\times\fraka_H^\vee$, which is $2$-dimensional.

\subsubsection*{{\rm F1)} $(\frakg,\frakh)=(\sl(n+1,\CC),\gl(n,\CC))$}

This follows from F3) with Lemma~\ref{lem:SplitImpliesComplex}.

\subsubsection*{{\rm F2)} $(\frakg,\frakh)=(\so(n+1,\CC),\so(n,\CC))$}

This follows from F5) with Lemma~\ref{lem:SplitImpliesComplex}.

\subsubsection*{{\rm F3)} $(\frakg,\frakh)=(\sl(n+1,\RR),\gl(n,\RR))$}

The Satake diagrams of $\frakg$ and $\frakh$ are
$$\frakg:\begin{xy}
\ar@{-} (0,0) *++!D{\alpha_1} *{\circ}="A";
  (10,0) *++!D{\alpha_2}  *{\circ}="B"
\ar@{-} "B"; (20,0)
\ar@{.} (20,0); (30,0) 
\ar@{-} (30,0); (40,0) *++!D{\alpha_{n-1}}  *{\circ}="C"
\ar@{-} "C"; (50,0) *++!D{\alpha_n}  *{\circ}="D"
\end{xy} $$
$$\frakh:\begin{xy}
\ar@{-} (0,0) *++!D{\beta_1} *{\circ}="A";
  (10,0) *++!D{\beta_2}  *{\circ}="B"
\ar@{-} "B"; (20,0)
\ar@{.} (20,0); (30,0) 
\ar@{-} (30,0); (40,0) *++!D{\beta_{n-1}}  *{\circ}="C"
\end{xy} $$

We choose $\beta_i=\alpha_i$ ($1\leq i\leq n-1$) as simple roots for $\sl(n,\RR)\subseteq\frakh$, then $\zeta_n:=\varpi_n$ defines a non-trivial character of $\RR\subseteq\frakh$. By Appendix~\ref{app:BranchingSLn} the pairs $(2\varpi_i,2\zeta_i+2\tfrac{i}{n}\zeta_n)$ ($1\leq i\leq n-1$), $(2\varpi_i,2\zeta_{i-1}-2\tfrac{n-i+1}{n}\zeta_n)$ ($2\leq i\leq n$) and $(2\varpi_1,-2\zeta_n),(2\varpi_n,2\zeta_n)$ are contained in $\Lambda(\frakg,\frakh)$ and they clearly span $\fraka_G^\vee\times\fraka_H^\vee$, which is $2n$-dimensional.

\subsubsection*{{\rm F4)} $(\frakg,\frakh)=(\su(p,q+1),\fraku(p,q))$}

Set $s=\min(p,q+1)$ and $t=\min(p,q)$, then the Satake diagrams of $\frakg$ and $\frakh$ are
\[\frakg:\begin{cases}\begin{xy}
\ar@{-} (0,0) *++!D{\alpha_1} *{\circ}="A";
\ar@{-} "A"; (10,0)
\ar@{.} (10,0); (20,0)
\ar@{-} (20,0); (30,0) *++!D{\alpha_s}  *{\circ}="B"
\ar@{-} "B"; (40,0) *++!D{\alpha_{s+1}}  *{\bullet}="C"
\ar@{-} "C"; (50,0)
\ar@{.} (50,0); (60,0)
\ar@{-} (60,0); (70,0) *++!DR{\alpha_{p+q-s}\!\!\!\!\!\!\!}  *{\bullet}="D"
\ar@{-} "D"; (80,0) *++!D{\,\,\,\,\,\,\,\alpha_{p+q-s+1}}  *{\circ}="E"
\ar@{-} "E"; (90,0)
\ar@{.} (90,0); (100,0)
\ar@{-} (100,0); (110,0) *++!D{\alpha_{p+q}}  *{\circ}="F"
\ar@/_1.5pc/@{<->} "A"; "F"
\ar@/_0.7pc/@{<->} "B"; "E"
\end{xy}&\mbox{for $p\neq q+1$,}\\
\begin{xy}
\ar@{-} (0,0) *++!D{\alpha_1} *{\circ}="A";
\ar@{-} "A"; (10,0)
\ar@{.} (10,0); (20,0)
\ar@{-} (20,0); (30,0) *++!D{\alpha_{s-1}}  *{\circ}="B"
\ar@{-} "B"; (40,0) *++!D{\alpha_s}  *{\circ}="C"
\ar@{-} "C"; (50,0) *++!D{\alpha_{s+1}}  *{\circ}="D"
\ar@{-} "D"; (60,0)
\ar@{.} (60,0); (70,0)
\ar@{-} (70,0); (80,0) *++!D{\alpha_{p+q}}  *{\circ}="F"
\ar@/_1.5pc/@{<->} "A"; "F"
\ar@/_0.7pc/@{<->} "B"; "D"
\end{xy}&\mbox{for $p=q+1$,}\end{cases}\]
\[\frakh:\begin{cases}\begin{xy}
\ar@{-} (0,0) *++!D{\beta_1} *{\circ}="A";
\ar@{-} "A"; (10,0)
\ar@{.} (10,0); (20,0)
\ar@{-} (20,0); (30,0) *++!D{\beta_t}  *{\circ}="B"
\ar@{-} "B"; (40,0) *++!D{\beta_{t+1}}  *{\bullet}="C"
\ar@{-} "C"; (50,0)
\ar@{.} (50,0); (60,0)
\ar@{-} (60,0); (70,0) *++!DR{\beta_{p+q-t-1}\!\!\!\!\!\!\!}  *{\bullet}="D"
\ar@{-} "D"; (80,0) *++!D{\,\,\,\,\,\,\,\beta_{p+q-s}}  *{\circ}="E"
\ar@{-} "E"; (90,0)
\ar@{.} (90,0); (100,0)
\ar@{-} (100,0); (110,0) *++!D{\beta_{p+q-1}}  *{\circ}="F"
\ar@/_1.5pc/@{<->} "A"; "F"
\ar@/_0.7pc/@{<->} "B"; "E"
\end{xy}&\mbox{for $p\neq q$,}\\
\begin{xy}
\ar@{-} (0,0) *++!D{\beta_1} *{\circ}="A";
\ar@{-} "A"; (10,0)
\ar@{.} (10,0); (20,0)
\ar@{-} (20,0); (30,0) *++!D{\beta_{t-1}}  *{\circ}="B"
\ar@{-} "B"; (40,0) *++!D{\beta_t}  *{\circ}="C"
\ar@{-} "C"; (50,0) *++!D{\beta_{t+1}}  *{\circ}="D"
\ar@{-} "D"; (60,0)
\ar@{.} (60,0); (70,0)
\ar@{-} (70,0); (80,0) *++!D{\beta_{p+q-1}}  *{\circ}="F"
\ar@/_1.5pc/@{<->} "A"; "F"
\ar@/_0.7pc/@{<->} "B"; "D"
\end{xy}&\mbox{for $p=q$.}\end{cases}\]

We choose the simple roots $\beta_i=\alpha_i$ ($1\leq i\leq p+q-1$) for $\su(p,q)\subseteq\frakh$, then $\zeta_{p+q}:=\varpi_{p+q}$ is trivial on $\su(p,q)$ and describes a character of $\fraku(1)\subseteq\frakh$. We use Appendix~\ref{app:BranchingSLn} in what follows:
\begin{itemize}
\item ($p\leq q$) Then $s=t=p$ and the pairs $(2(\varpi_i+\varpi_{p+q-i+1}),2(\zeta_i+\zeta_{p+q-i})),(2(\varpi_i+\varpi_{p+q-i+1}),2(\zeta_{i-1}+\zeta_{p+q-i+1}))$ ($1\leq i\leq p$) are contained in $\Lambda(\frakg,\frakh)$ and span $\fraka_G^\vee\times\fraka_H^\vee$, which is of dimension $s+t=2p$.
\item ($p\geq q+1$) Then $s=q+1$ and $t=q$, so that the pairs $(2(\varpi_i+\varpi_{p+q-i+1}),2(\zeta_i+\zeta_{p+q-i})),(2(\varpi_i+\varpi_{p+q-i+1}),2(\zeta_{i-1}+\zeta_{p+q-i+1}))$ ($1\leq i\leq q$) and $(2(\varpi_{q+1}+\varpi_p),2(\zeta_q+\zeta_p))$ are contained in $\Lambda(\frakg,\frakh)$ and span $\fraka_G^\vee\times\fraka_H^\vee$, which is of dimension $s+t=2q+1$.
\end{itemize}

\subsubsection*{{\rm F5)} $(\frakg,\frakh)=(\so(p,q+1),\so(p,q))$}

For $p+q=2m$ even and $s=\min(p,q+1)$, $t=\min(p,q)$ the Satake diagrams of $\frakg$ and $\frakh$ are
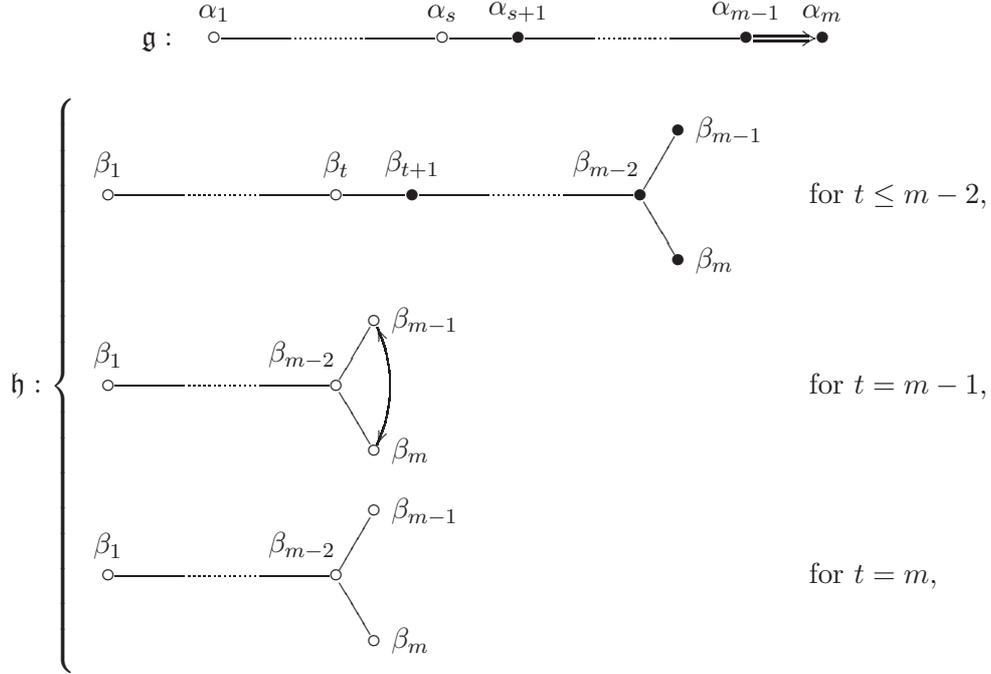
\[\frakg:\begin{xy}
\ar@{-} (0,0) *++!D{\alpha_1} *{\circ}="A"; (10,0)
\ar@{.} (10,0); (20,0) 
\ar@{-} (20,0); (30,0) *++!D{\alpha_s}  *{\circ}="B"
\ar@{-} "B"; (40,0) *++!D{\alpha_{s+1}}  *{\bullet}="C"
\ar@{-} "C"; (50,0)
\ar@{.} (50,0); (60,0) 
\ar@{-} (60,0); (70,0) *++!D{\alpha_{m-1}}  *{\bullet}="D"
\ar@{=>} "D"; (80,0) *++!D{\alpha_m}  *{\bullet}="E"
\end{xy}\]
\[\frakh:\begin{cases}\begin{xy}
\ar@{-} (0,0) *++!D{\beta_1} *{\circ}="A"; (10,0)
\ar@{.} (10,0); (20,0) 
\ar@{-} (20,0); (30,0) *++!D{\beta_t}  *{\circ}="B"
\ar@{-} "B"; (40,0) *++!D{\beta_{t+1}}  *{\bullet}="C"
\ar@{-} "C"; (50,0)
\ar@{.} (50,0); (60,0) 
\ar@{-} (60,0); (70,0) *++!D{\beta_{m-2}\ \ \ \ \ \ \ }  *{\bullet}="D"
\ar@{-} "D"; (75,8.6)  *++!L{\beta_{m-1}} *{\bullet}
\ar@{-} "D"; (75,-8.6)  *++!L{\beta_m} *{\bullet}
\end{xy}&\mbox{for $t\leq m-2$,}\\
\begin{xy}
\ar@{-} (0,0) *++!D{\beta_1} *{\circ}="A"; (10,0)
\ar@{.} (10,0); (20,0) 
\ar@{-} (20,0); (30,0) *++!D{\beta_{m-2}\ \ \ \ \ \ \ }  *{\circ}="B"
\ar@{-} "B"; (35,8.6)  *++!L{\beta_{m-1}} *{\circ}="C"
\ar@{-} "B"; (35,-8.6)  *++!L{\beta_m} *{\circ}="D"
\ar@/^0.5pc/@{<->} "C"; "D"
\end{xy}&\mbox{for $t=m-1$,}\\
\begin{xy}
\ar@{-} (0,0) *++!D{\beta_1} *{\circ}="A"; (10,0)
\ar@{.} (10,0); (20,0) 
\ar@{-} (20,0); (30,0) *++!D{\beta_{m-2}\ \ \ \ \ \ \ }  *{\circ}="B"
\ar@{-} "B"; (35,8.6)  *++!L{\beta_{m-1}} *{\circ}="C"
\ar@{-} "B"; (35,-8.6)  *++!L{\beta_m} *{\circ}="D"
\end{xy}&\mbox{for $t=m$,}\end{cases}\]

We use Appendix~\ref{app:BranchingSOn} in what follows:
\begin{itemize}
\item ($p+2<q$) Then $s=t=p\leq m-2$ and the pairs $(2\varpi_i,2\zeta_i),(2\varpi_i,2\zeta_{i-1})$ ($1\leq i\leq p$) are contained in $\Lambda(\frakg,\frakh)$ and span $\fraka_G^\vee\times\fraka_H^\vee$, which is of dimension $s+t=2p$.
\item ($p+2=q$) Again $s=t=p$, but now $t=m-1$. The pairs $(2\varpi_i,2\zeta_i),(2\varpi_i,2\zeta_{i-1})$ ($1\leq i\leq p-1$) and $(2\varpi_p,2(\zeta_p+\zeta_{p+1})),(2\varpi_p,2\zeta_{p-1})$ are contained in $\Lambda(\frakg,\frakh)$ and span $\fraka_G^\vee\times\fraka_H^\vee$, which is of dimension $s+t=2p$.
\item ($p=q$) Then $s=t=m=p$ and the pairs $(2\varpi_i,2\zeta_i),(2\varpi_i,2\zeta_{i-1})$ ($1\leq i\leq p-2$) and $(2\varpi_{p-1},2\zeta_{p-2}),(2\varpi_{p-1},2(\zeta_{p-1}+\zeta_p)),(2\varpi_p,2\zeta_p),(2\varpi_p,2\zeta_{p-1})$ are contained in $\Lambda(\frakg,\frakh)$ and span $\fraka_G^\vee\times\fraka_H^\vee$, which is of dimension $s+t=2p$.
\item ($p=q+2$) Here, $s=q+1$, $t=q$ and $t<s=m$. The pairs $(2\varpi_i,2\zeta_i),(2\varpi_i,2\zeta_{i-1})$ ($1\leq i\leq q-1$) and $(2\varpi_q,2\zeta_{q-1}),(2\varpi_q,2(\zeta_q+\zeta_{q+1})),(4\varpi_{q+1},2(\zeta_q+\zeta_{q+1}))$ are contained in $\Lambda(\frakg,\frakh)$ and span $\fraka_G^\vee\times\fraka_H^\vee$, which is of dimension $s+t=2q+1$.
\item ($p>q+2$) Again $s=q+1$, $t=q$, but now $t<s<m$. The pairs $(2\varpi_i,2\zeta_i),(2\varpi_i,2\zeta_{i-1})$ ($1\leq i\leq q$) and $(2\varpi_{q+1},2\zeta_q)$ are contained in $\Lambda(\frakg,\frakh)$ and span $\fraka_G^\vee\times\fraka_H^\vee$, which is of dimension $s+t=2q+1$.
\end{itemize}

The case $p+q=2m-1$ odd is treated similarly.

\subsubsection*{{\rm G1)} $(\frakg,\frakh)=(\frakg'+\frakg',\diag\frakg')$ with $\frakg'$ simple compact}

This is the same situation as in C).

\subsubsection*{{\rm G2)} $(\frakg,\frakh)=(\so(1,n)+\so(1,n),\diag\so(1,n))$}

The Satake diagram of $\so(1,n)$ is
$$ \so(1,n):\begin{tabular}{cc}
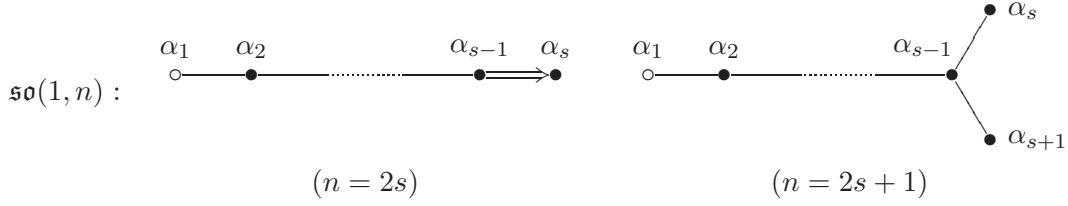
\begin{xy}
\ar@{-} (0,0) *++!D{\alpha_1} *{\circ}="A";
  (10,0) *++!D{\alpha_2}  *{\bullet}="B"
\ar@{-} "B"; (20,0)
\ar@{.} (20,0); (30,0) 
\ar@{-} (30,0); (40,0) *++!D{\alpha_{s-1}}  *{\bullet}="C"
\ar@{=>} "C"; (50,0) *++!D{\alpha_s}  *{\bullet}="D"
\end{xy}&\begin{xy}
\ar@{-} (0,0) *++!D{\alpha_1} *{\circ}="A"; 
 (10,0)  *++!D{\alpha_2} *{\bullet}="B"
\ar@{-} "B";  (20,0)
\ar@{.} (20,0); (30,0) 
\ar@{-} (30,0); (40,0) *++!D{\alpha_{s-1}\ \ \ \ \ \ } *{\bullet}="F"
\ar@{-} "F"; (45,8.6)  *++!L{\alpha_{s}} *{\bullet}
\ar@{-} "F"; (45,-8.6)  *++!L{\alpha_{s+1}} *{\bullet}
\end{xy}\\$(n=2s)$&$(n=2s+1)$\end{tabular} $$

The spherical representations of $\so(1,n)$ are of the form $F^{\so(1,n)}(2k\varpi_1)$, $k\in\NN$, and hence the spherical representations of $\frakg$ are of the form $F^{\so(1,n)}(2k_1\varpi_1)\boxtimes F^{\so(n,1)}(2k_2\varpi_1)$, $k_1,k_2\in\NN$. The representation $F^{\so(1,n)}(\varpi_1)$ is self-dual, and hence the trivial representation $F^{\so(1,n)}(0)$ is contained in the tensor product $F^{\so(1,n)}(\varpi_1)\otimes F^{\so(1,n)}(\varpi_1)$. Further, clearly
$$ F^{\so(1,n)}(\varpi_1)\otimes F^{\so(1,n)}(0) \simeq F^{\so(1,n)}(0)\otimes F^{\so(1,n)}(\varpi_1) \simeq F^{\so(1,n)}(\varpi_1). $$
Hence, $((2\varpi_1,0),2\varpi_1),((0,2\varpi_1),2\varpi_1),((2\varpi_1,2\varpi_1),0)\in\Lambda(\frakg,\frakh)$ span $\fraka_G^\vee\times\fraka_H^\vee$, which is $3$-dimensional.

\subsubsection*{{\rm H1)} $(\frakg,\frakh)=(\so(2,2n),\fraku(1,n))$}

The Satake diagrams of $\frakg$ and $\frakh$ are
$$ \frakg:
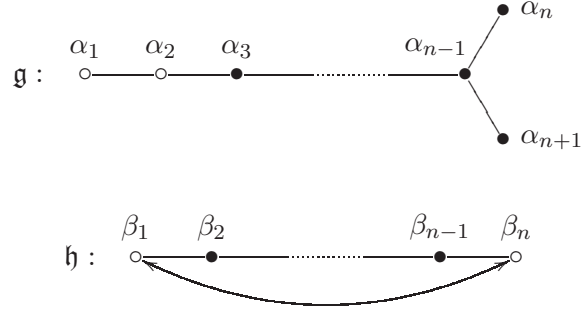
\begin{xy}
\ar@{-} (0,0) *++!D{\alpha_1} *{\circ}="A"; 
 (10,0)  *++!D{\alpha_2} *{\circ}="B"
\ar@{-} "B"; (20,0) *++!D{\alpha_3}  *{\bullet}="C"
\ar@{-} "C";  (30,0)
\ar@{.} (30,0); (40,0) 
\ar@{-} (40,0); (50,0) *++!D{\alpha_{n-1}\ \ \ \ \ \ } *{\bullet}="D"
\ar@{-} "D"; (55,8.6)  *++!L{\alpha_n} *{\bullet}
\ar@{-} "D"; (55,-8.6)  *++!L{\alpha_{n+1}} *{\bullet}
\end{xy}
$$

$$ \frakh:\begin{xy}
\ar@{-} (0,0) *++!D{\beta_1} *{\circ}="A";
  (10,0) *++!D{\beta_2}  *{\bullet}="B"
\ar@{-} "B"; (20,0)
\ar@{.} (20,0); (30,0) 
\ar@{-} (30,0); (40,0) *++!D{\beta_{n-1}}  *{\bullet}="C"
\ar@{-} "C"; (50,0) *++!D{\beta_n}  *{\circ}="D"
\ar@/_1.5pc/@{<->} "A"; "D"
\end{xy} $$

We realize the root system of $\frakg$ as $\{\pm e_i\pm e_j:1\leq i<j\leq n+1\}$ with simple roots $\alpha_i=e_i-e_{i+1}$, $i=1,\ldots,n$ and $\alpha_{n+1}=e_n+e_{n+1}$. Choose the simple roots $\beta_i=\alpha_i$ ($i=1,\ldots,n$) for $\su(1,n)\subseteq\frakh$; then $\zeta_{n+1}:=\varpi_{n+1}=\frac{1}{2}(e_1+\cdots+e_{n+1})$ defines a non-trivial character of $\fraku(1)\subseteq\frakh$.

First consider the fundamental weight $\varpi_2=e_1+e_2$; then the Weyl group orbit of $\varpi_2$ is equal to the set of roots $\{\pm e_i\pm e_j:1\leq i<j\leq n+1\}$. Hence, the weights $e_1+e_2$, $-e_n-e_{n+1}$ and $e_1-e_{n+1}$ of $F^\frakg(\varpi_2)$ are highest weights for $\frakh$ so that the representations $F^\frakh(e_1+e_2)$, $F^\frakh(-e_n-e_{n+1})$ and $F^\frakh(e_1-e_{n+1})$ are contained in $F^\frakg(\varpi_2)|_\frakh$. Note that $e_1-e_{n+1}=\zeta_1+\zeta_n$. Using the Weyl dimension formula we find
$$ \dim F^\frakg(\varpi_2) = \dim F^\frakh(e_1+e_2) + \dim F^\frakh(-e_n-e_{n+1}) + \dim F^\frakh(e_1-e_{n+1}) + 1. $$
Using the Kostant Branching Formula it is further easy to see that the remaining one-dimensional representation in $F^\frakg(\varpi_2)|_\frakh$ is the trivial representation so that $(2\varpi_2,2(\zeta_1+\zeta_n)),(2\varpi_2,0)\in\Lambda(\frakg,\frakh)$.

Next consider the highest weight $2\varpi_1=2e_1$ of $\frakg$. Similar considerations to above show that $F^\frakg(2\varpi_1)|_\frakh$ contains $F^\frakh(\zeta_1+\zeta_n)$ so that $(4\varpi_1,2(\zeta_1+\zeta_n))\in\Lambda(\frakg,\frakh)$.

Together the three pairs $(4\varpi_1,2(\zeta_1+\zeta_n)),(2\varpi_2,2(\zeta_1+\zeta_n)),(2\varpi_2,0)\in\Lambda(\frakg,\frakh)$ span $\fraka_G^\vee\times\fraka_H^\vee$ which is $3$-dimensional.

\subsubsection*{{\rm H2)} $(\frakg,\frakh)=(\su^*(2n+2),\su^*(2n)+\RR+\su(2))$}

The Satake diagrams of $\frakg$ and $\frakh$ are
$$ \frakg:\begin{xy}
\ar@{-} (0,0) *++!D{\alpha_1} *{\bullet}="A"; 
 (10,0)  *++!D{\alpha_2} *{\circ}="B"
\ar@{-} "B"; (20,0) *++!D{\alpha_3}  *{\bullet}="C"
\ar@{-} "C";  (30,0)
\ar@{.} (30,0); (40,0) 
\ar@{-} (40,0); (50,0) *++!D{\alpha_{2n}}  *{\circ}="D"
\ar@{-} "D"; (60,0) *++!D{\alpha_{2n+1}}  *{\bullet}="E"
\end{xy} $$

$$ \su^*(2n):\begin{xy}
\ar@{-} (0,0) *++!D{\beta_1} *{\bullet}="A"; 
 (10,0)  *++!D{\beta_2} *{\circ}="B"
\ar@{-} "B"; (20,0) *++!D{\beta_3}  *{\bullet}="C"
\ar@{-} "C";  (30,0)
\ar@{.} (30,0); (40,0) 
\ar@{-} (40,0); (50,0) *++!D{\beta_{2n-2}}  *{\circ}="D"
\ar@{-} "D"; (60,0) *++!D{\beta_{2n-1}}  *{\bullet}="E"
\end{xy} $$

We realize the root system of $\frakg$ as $\{\pm(e_i-e_j):1\leq i<j\leq 2n+2\}$ in the vector space $V=\{x\in\RR^{2n+2}:x_1+\cdots+x_{2n+2}=0\}$. To simplify notation, denote by $\pi(x)$ the orthogonal projection of $x\in\RR^{2n+2}$ to $V$. We choose the simple roots $\alpha_i=e_i-e_{i+1}$ for $\frakg$ and the simple roots $\beta_i=\alpha_i$ ($i=1,\ldots,2n-1$) for $\su^*(2n)$. Then $\zeta_{2n}:=\varpi_{2n}$ describes a character of $\RR\subseteq\frakh$ and $\alpha_{2n+1}$ is the non-trivial root for $\su(2)\subseteq\frakh$. Let $\zeta_{2n+1}:=\tfrac{1}{2}\alpha_{2n+1}$ denote the fundamental weight for $\su(2)$. Then a general irreducible representation of $\frakh$ takes the form $F^{\su^*(2n)}(\ell_1\zeta_1+\cdots+\ell_{2n-1}\zeta_{2n-1})\boxtimes F^\RR(\ell_{2n}\zeta_{2n})\boxtimes F^{\su(2)}(\ell_{2n+1}\zeta_{2n+1})$ with $\ell_1,\ldots,\ell_{2n-1},\ell_{2n+1}\in\NN$ and $\ell_{2n}\in\CC$.

Consider the fundamental weight $\varpi_i=\pi(e_1+\cdots+e_i)=\tfrac{2n-i+2}{2n+2}(e_1+\cdots+e_i)-\tfrac{i}{2n+2}(e_{i+1}+\cdots+e_{2n+2})$ of $\frakg$. In the Weyl group orbit of $\varpi_i$ the three weights $\pi(e_1+\cdots+e_i)$, $\pi(e_1+\cdots+e_{i-1}+e_{2n+1})$ and $\pi(e_1+\cdots+e_{i-2}+e_{2n+1}+e_{2n+2})$ are highest weights for $\frakh$. Moreover,
$$ \varpi_k = \zeta_k+\tfrac{k}{2n}\zeta_{2n} \quad (1\leq k\leq 2n-1), \qquad \varpi_{2n} = \zeta_{2n} $$
and
\begin{align*}
 \pi(e_1+\cdots+e_{i-1}+e_{2n+1}) &= \zeta_{i-1}-\tfrac{n-i+1}{2n}\zeta_{2n}+\zeta_{2n+1},\\
 \pi(e_1+\cdots+e_{i-2}+e_{2n+1}+e_{2n+2}) &= \zeta_{i-2}-\tfrac{2n-i+2}{2n}\zeta_{2n},
\end{align*}
so that the representations
\begin{align*}
 & F^{\su^*(2n)}(\zeta_i)\boxtimes F^\RR(\tfrac{k}{2n}\zeta_{2n})\boxtimes F^{\su(2)}(0),\\
 & F^{\su^*(2n)}(\zeta_{i-1})\boxtimes F^\RR(-\tfrac{n-i+1}{2n}\zeta_{2n})\boxtimes F^{\su(2)}(\zeta_{2n+1}),\\
 & F^{\su^*(2n)}(\zeta_{i-2})\boxtimes F^\RR(-\tfrac{2n-i+2}{2n}\zeta_{2n})\boxtimes F^{\su(2)}(0)
\end{align*}
are contained in $F^\frakg(\varpi_i)|_\frakh$. Using the Weyl Dimension Formula we find
\begin{multline*}
 F^\frakg(\varpi_i)|_\frakh \simeq \Big(F^{\su^*(2n)}(\zeta_i)\boxtimes F^\RR(\tfrac{k}{2n}\zeta_{2n})\boxtimes F^{\su(2)}(0)\Big)\\
 \oplus\Big(F^{\su^*(2n)}(\zeta_{i-1})\boxtimes F^\RR(-\tfrac{n-i+1}{2n}\zeta_{2n})\boxtimes F^{\su(2)}(\zeta_{2n+1})\Big)\\
 \oplus\Big(F^{\su^*(2n)}(\zeta_{i-2})\boxtimes F^\RR(-\tfrac{2n-i+2}{2n}\zeta_{2n})\boxtimes F^{\su(2)}(0)\Big).
\end{multline*}
This implies that
\begin{align*}
 & (2\varpi_2,2\zeta_2+\tfrac{2}{n}\zeta_{2n}), && (2\varpi_2,-2\zeta_{2n})\\
 & (2\varpi_{2i},2\zeta_{2i}+\tfrac{2i}{n}\zeta_{2n}), && (2\varpi_{2i},2\zeta_{2i-2}-\tfrac{2(n-i+1)}{n}\zeta_{2n}) && (2\leq i\leq n-1),\\
 & (2\varpi_{2n},2\zeta_{2n}), && (2\varpi_{2n},2\zeta_{2n-2}-\tfrac{2}{n}\zeta_{2n})
\end{align*}
are contained in $\Lambda(\frakg,\frakh)$, and since $\fraka_G^\vee\times\fraka_H^\vee$ is $2n$-dimensional, they form a basis.

\subsubsection*{{\rm H3)} $(\frakg,\frakh)=(\so^*(2n+2),\so^*(2n)+\so(2))$}

The Satake diagrams of $\frakg$ and $\frakh$ are
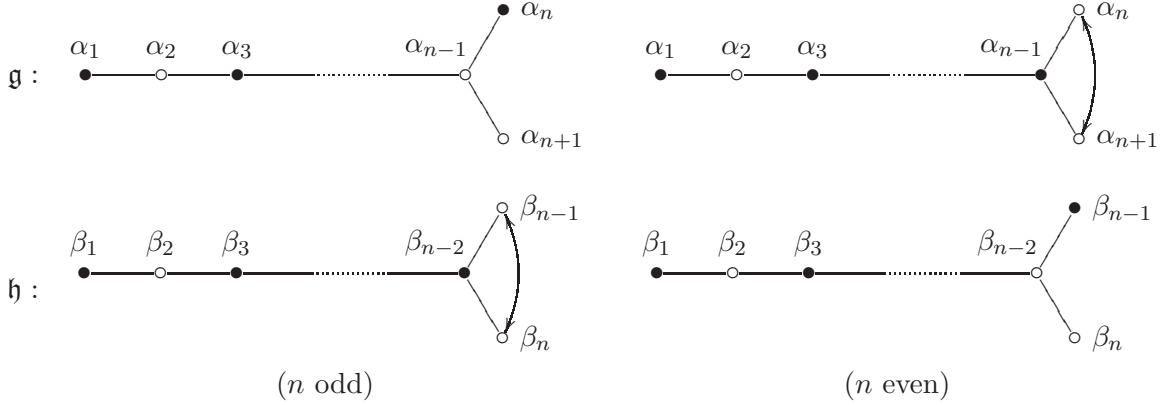
\begin{align*}
\frakg:&\begin{tabular}{cc}\begin{xy}
\ar@{-} (0,0) *++!D{\alpha_1} *{\bullet}="A"; 
 (10,0)  *++!D{\alpha_2} *{\circ}="B"
\ar@{-} "B"; (20,0) *++!D{\alpha_3}  *{\bullet}="C"
\ar@{-} "C";  (30,0)
\ar@{.} (30,0); (40,0) 
\ar@{-} (40,0); (50,0) *++!D{\alpha_{n-1}\ \ \ \ \ \ } *{\circ}="D"
\ar@{-} "D"; (55,8.6)  *++!L{\alpha_n} *{\bullet}
\ar@{-} "D"; (55,-8.6)  *++!L{\alpha_{n+1}} *{\circ}
\end{xy}&\begin{xy}
\ar@{-} (0,0) *++!D{\alpha_1} *{\bullet}="A"; 
 (10,0)  *++!D{\alpha_2} *{\circ}="B"
\ar@{-} "B"; (20,0) *++!D{\alpha_3}  *{\bullet}="C"
\ar@{-} "C";  (30,0)
\ar@{.} (30,0); (40,0) 
\ar@{-} (40,0); (50,0) *++!D{\alpha_{n-1}\ \ \ \ \ \ } *{\bullet}="D"
\ar@{-} "D"; (55,8.6)  *++!L{\alpha_n} *{\circ}="E"
\ar@{-} "D"; (55,-8.6)  *++!L{\alpha_{n+1}} *{\circ}="F"
\ar@/^0.5pc/@{<->} "E"; "F"
\end{xy}\end{tabular}\\
\frakh:&\begin{tabular}{cc}\begin{xy}
\ar@{-} (0,0) *++!D{\beta_1} *{\bullet}="A"; 
 (10,0)  *++!D{\beta_2} *{\circ}="B"
\ar@{-} "B"; (20,0) *++!D{\beta_3}  *{\bullet}="C"
\ar@{-} "C";  (30,0)
\ar@{.} (30,0); (40,0) 
\ar@{-} (40,0); (50,0) *++!D{\beta_{n-2}\ \ \ \ \ \ } *{\bullet}="D"
\ar@{-} "D"; (55,8.6)  *++!L{\beta_{n-1}} *{\circ}="E"
\ar@{-} "D"; (55,-8.6)  *++!L{\beta_n} *{\circ}="F"
\ar@/^0.5pc/@{<->} "E"; "F"
\end{xy}&\begin{xy}
\ar@{-} (0,0) *++!D{\beta_1} *{\bullet}="A"; 
 (10,0)  *++!D{\beta_2} *{\circ}="B"
\ar@{-} "B"; (20,0) *++!D{\beta_3}  *{\bullet}="C"
\ar@{-} "C";  (30,0)
\ar@{.} (30,0); (40,0) 
\ar@{-} (40,0); (50,0) *++!D{\beta_{n-2}\ \ \ \ \ \ } *{\circ}="D"
\ar@{-} "D"; (55,8.6)  *++!L{\beta_{n-1}} *{\bullet}="E"
\ar@{-} "D"; (55,-8.6)  *++!L{\beta_n} *{\circ}="F"
\end{xy}\\($n$ odd)&($n$ even)\end{tabular}
\end{align*}

We realize the root system of $\frakg$ as $\{\pm e_i\pm e_j:1\leq i<j\leq n+1\}$ with simple roots $\alpha_i=e_i-e_{i+1}$ ($1\leq i\leq n$) and $\alpha_{n+1}=e_n+e_{n+1}$. Then we may choose the simple roots for $\frakh$ as $\beta_i=\alpha_{i+1}$ ($1\leq i\leq n$). The fundamental weight $\zeta_{n+1}=\varpi_1=e_1$ is trivial on $\so^*(2n)\subseteq\frakh$ and hence defines a non-trivial character of $\so(2)\subseteq\frakh$.

Consider the fundamental weight $\varpi_i=e_1+\cdots+e_i$, $i=1,\ldots,n-1$. By branching in stages, first from $\frakg_\CC=\so(2n+2,\CC)$ to $\so(2n+1,\CC)$, then from $\so(2n+1,\CC)$ to $\so(2n,\CC)=\so^*(2n)_\CC$, it follows that
$$ F^\frakg(\varpi_i)|_{\so^*(2n)} \simeq F^{\so^*(2n)}(\zeta_i)\oplus F^{\so^*(2n)}(\zeta_{i-1})^{\oplus2}\oplus F^{\so^*(2n)}(\zeta_{i-2}). $$
To find the action of the $\so(2)$-factor we consider the weights of $F^\frakg(\varpi_i)$ which are given by
$$ \{\pm e_{k_1}\pm\cdots\pm e_{k_m}:1\leq k_1<\ldots<k_m\leq 2n+2, i-m\in2\NN\}\cup\{0\}. $$
Among these, clearly $e_2+\cdots+e_{i+1}=\zeta_i$, $\pm e_1+e_2+\cdots+e_i=\pm\zeta_{n+1}+\zeta_{i-1}$ and $e_2+\cdots+e_{i-1}=\zeta_{i-2}$ are highest weights for $\frakh$ and hence
\begin{align*}
 F^\frakg(\varpi_i)|_\frakh \simeq{}& \Big(F^{\so^*(2n)}(\zeta_i)\boxtimes F^{\so(2)}(0)\Big)\oplus\Big(F^{\so^*(2n)}(\zeta_{i-1})\boxtimes F^{\so(2)}(\zeta_{n+1})\Big)\\
 & \oplus\Big(F^{\so^*(2n)}(\zeta_{i-1})\boxtimes F^{\so(2)}(-\zeta_{n+1})\Big)\oplus\Big(F^{\so^*(2n)}(\zeta_{i-2})\boxtimes F^{\so(2)}(0)\Big).
\end{align*}
Similarly, for the highest weight $\varpi_n+\varpi_{n+1}=e_1+\cdots+e_n$ one obtains
\begin{align*}
 F^\frakg(\varpi_n+\varpi_{n+1})|_\frakh \simeq{}& \Big(F^{\so^*(2n)}(2\zeta_{n-1})\boxtimes F^{\so(2)}(0)\Big)\oplus\Big(F^{\so^*(2n)}(2\zeta_n)\boxtimes F^{\so(2)}(0)\Big)\\
 & \oplus\Big(F^{\so^*(2n)}(\zeta_{n-1}+\zeta_n)\boxtimes\big(F^{\so(2)}(\zeta_{n+1})\oplus F^{\so(2)}(-\zeta_{n+1})\big)\Big)\\
 & \oplus\Big(F^{\so^*(2n)}(\zeta_{n-2})\boxtimes F^{\so(2)}(0)\Big),
\end{align*}
and for the highest weight $2\varpi_{n+1}=e_1+\cdots+e_{n+1}$, accordingly
\begin{align*}
 F^\frakg(2\varpi_{n+1})|_\frakh \simeq{}& \Big(F^{\so^*(2n)}(2\zeta_n)\boxtimes F^{\so(2)}(\zeta_{n+1})\Big)\oplus\Big(F^{\so^*(2n)}(2\zeta_{n-1})\boxtimes F^{\so(2)}(-\zeta_{n+1})\Big)\\
 & \oplus\Big(F^{\so^*(2n)}(\zeta_{n-1}+\zeta_n)\boxtimes F^{\so(2)}(0)\Big).
\end{align*}

Now, if $n=2s-1$ is odd, then
$$ (2\varpi_{2i},2\zeta_{2i}), (2\varpi_{2i},2\zeta_{2i-2}) \quad (1\leq i\leq s-1), \qquad (4\varpi_{2s},2(\zeta_{2s-2}+\zeta_{2s-1})) $$
are contained in $\Lambda(\frakg,\frakh)$ and span $\fraka_G^\vee\times\fraka_H^\vee$, which is of dimension $s+(s-1)=2s-1$. If $n=2s$ is even, then
$$ (2\varpi_{2i},2\zeta_{2i}), (2\varpi_{2i},2\zeta_{2i-2}) \quad (1\leq i\leq s-1), \qquad (2(\varpi_{2s}+\varpi_{2s+1}),4\zeta_{2s}), (2(\varpi_{2s}+\varpi_{2s+1}),2\zeta_{2s-2}) $$
are contained in $\Lambda(\frakg,\frakh)$ and span $\fraka_G^\vee\times\fraka_H^\vee$, which is of dimension $s+s=2s$.

\subsubsection*{{\rm H4)} $(\frakg,\frakh)=(\sp(p,q+1),\sp(p,q)+\sp(1))$}

Let $s=\min(p,q+1)$ and $t=\min(p,q)$; then the Satake diagrams of $\frakg$ and $\sp(p,q)$ are
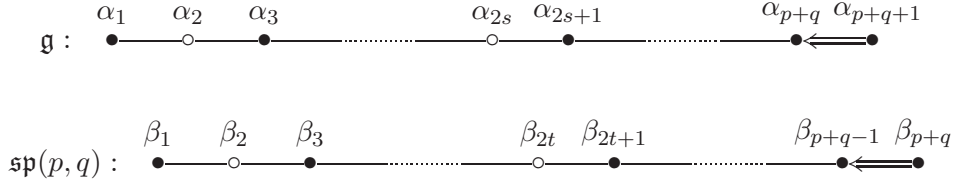
$$ \frakg:\begin{xy}
\ar@{-} (0,0) *++!D{\alpha_1} *{\bullet}="A"; 
 (10,0)  *++!D{\alpha_2} *{\circ}="B"
\ar@{-} "B"; (20,0) *++!D{\alpha_3}  *{\bullet}="C"
\ar@{-} "C";  (30,0)
\ar@{.} (30,0); (40,0) 
\ar@{-} (40,0); (50,0) *++!D{\alpha_{2s}} *{\circ}="D"
\ar@{-} "D"; (60,0) *++!D{\alpha_{2s+1}} *{\bullet}="E"
\ar@{-} "E";  (70,0)
\ar@{.} (70,0); (80,0) 
\ar@{-} (80,0); (90,0) *++!D{\alpha_{p+q}\ } *{\bullet}="F"
\ar@{<=} "F"; (100,0) *++!D{\ \alpha_{p+q+1}} *{\bullet}="G"
\end{xy} $$

$$ \sp(p,q):\begin{xy}
\ar@{-} (0,0) *++!D{\beta_1} *{\bullet}="A"; 
 (10,0)  *++!D{\beta_2} *{\circ}="B"
\ar@{-} "B"; (20,0) *++!D{\beta_3}  *{\bullet}="C"
\ar@{-} "C";  (30,0)
\ar@{.} (30,0); (40,0) 
\ar@{-} (40,0); (50,0) *++!D{\beta_{2t}} *{\circ}="D"
\ar@{-} "D"; (60,0) *++!D{\beta_{2t+1}} *{\bullet}="E"
\ar@{-} "E";  (70,0)
\ar@{.} (70,0); (80,0) 
\ar@{-} (80,0); (90,0) *++!D{\beta_{p+q-1}\ } *{\bullet}="F"
\ar@{<=} "F"; (100,0) *++!D{\ \beta_{p+q}} *{\bullet}="G"
\end{xy} $$

We realize the root system of $\frakg$ as $\{\pm e_i\pm e_j:1\leq i<j\leq p+q+1\}\cup\{\pm2e_i:1\leq i\leq p+q+1\}$ and choose $\alpha_i=e_i-e_{i+1}$ ($1\leq i\leq p+q$) and $\alpha_{p+q+1}=2e_{p+q+1}$ as simple roots. We further choose $\beta_i=\alpha_i$ ($1\leq i\leq p+q-1$) and $\beta_{p+q}=2e_{p+q}$ as simple roots for $\sp(p,q)\subseteq\frakh$; then $\beta_{p+q+1}=\alpha_{p+q+1}$ is a simple root for $\sp(1)$ and $\zeta_{p+q+1}=\varpi_{p+q+1}=e_{p+q+1}$ the corresponding fundamental weight. A general irreducible representation of $\frakh$ takes the form $F^{\sp(p,q)}(\ell_1\zeta_1+\cdots+\ell_{p+q}\zeta_{p+q})\boxtimes F^{\sp(1)}(\ell_{p+q+1}\zeta_{p+q+1})$ with $\ell_1,\ldots,\ell_{p+q+1}\in\NN$.

We consider the fundamental weight $\varpi_i=e_1+\cdots+e_i$. Then by \cite[Theorem 3.3]{WY09} we have
\begin{multline*}
 F^\frakg(\varpi_i)|_\frakh \simeq \Big(F^{\sp(p,q)}(\zeta_i)\boxtimes F^{\sp(1)}(0)\Big) \oplus \Big(F^{\sp(p,q)}(\zeta_{i-1})\boxtimes F^{\sp(1)}(\zeta_{p+q+1})\Big)\\
 \oplus \Big(F^{\sp(p,q)}(\zeta_{i-2})\boxtimes F^{\sp(1)}(0)\Big)
\end{multline*}
for $i=1,\ldots,p+q$, and
$$ F^\frakg(\varpi_{p+q+1})|_\frakh \simeq \Big(F^{\sp(p,q)}(\zeta_{p+q})\boxtimes F^{\sp(1)}(\zeta_{p+q+1})\Big)\\
 \oplus \Big(F^{\sp(p,q)}(\zeta_{p+q-1})\boxtimes F^{\sp(1)}(0)\Big). $$

Hence, the following pairs in $\Lambda(\frakg,\frakh)$ span $\fraka_G^\vee\times\fraka_H^\vee$:
\begin{itemize}
\item ($p\leq q$) Then $s=t=p$, so that $(2\varpi_{2i},2\zeta_{2i}),(2\varpi_{2i},2\zeta_{2i-2})$ ($1\leq i\leq p$) are contained in $\Lambda(\frakg,\frakh)$ and span $\fraka_G^\vee\times\fraka_H^\vee$.
\item ($p\geq q+1$) Then $s=q+1$ and $t=q$, so that $(2\varpi_{2i},2\zeta_{2i}),(2\varpi_{2i},2\zeta_{2i-2})$ ($1\leq i\leq q$) and $(2\varpi_{2q+2},2\zeta_{2q})$ are contained in $\Lambda(\frakg,\frakh)$ and span $\fraka_G^\vee\times\fraka_H^\vee$.
\end{itemize}

\subsubsection*{{\rm H5)} $(\frakg,\frakh)=(\frake_{6(-26)},\so(9,1)+\RR)$}

The Satake diagrams of $\frakg$ and $\so(9,1)\subseteq\frakh$ are
$$ \frakg:
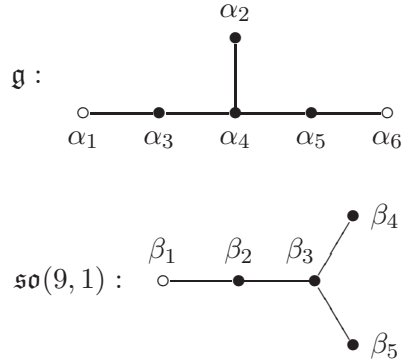
\begin{xy}
\ar@{-} (0,-5) *++!U{\alpha_1} *{\circ}="A";
  (10,-5) *++!U{\alpha_3}  *{\bullet}="B"
\ar@{-} "B"; (20,-5) *++!U{\alpha_4} *{\bullet}="C"
\ar@{-} (20,-5); (30,-5) *++!U{\alpha_5}  *{\bullet}="D"
\ar@{-} (30,-5); (40,-5) *++!U{\alpha_6}  *{\circ}="E"
\ar@{-} "C"; (20,5) *++!D{\alpha_2} *{\bullet}="C"
\end{xy} $$

$$ \so(9,1):\begin{xy}
\ar@{-} (0,0) *++!D{\beta_1} *{\circ}="A"; 
 (10,0)  *++!D{\beta_2} *{\bullet}="B"
\ar@{-} "B";  (20,0)
\ar@{-} (10,0); (20,0) *++!D{\beta_3\ \ \ } *{\bullet}="F"
\ar@{-} "F"; (25,8.6)  *++!L{\beta_4} *{\bullet}
\ar@{-} "F"; (25,-8.6)  *++!L{\beta_5} *{\bullet}
\end{xy} $$

We choose the simple roots for $\frakg$ such that $\alpha_1=\beta_1$, $\alpha_2=\beta_4$, $\alpha_3=\beta_2$, $\alpha_4=\beta_3$, $\alpha_5=\beta_5$, then $\zeta_6:=\varpi_6$ defines a non-trivial character of $\RR\subseteq\frakh$. A general irreducible representation of $\frakh$ is of the form $F^{\so(9,1)}(\ell_1\zeta_1+\cdots+\ell_5\zeta_5)\boxtimes F^\RR(\ell_6\zeta_6)$ for $\ell_1,\ldots,\ell_5\in\NN$ and $\ell_6\in\CC$.

By writing $\varpi_i$ as a linear combination of the simple roots $\alpha_j$ (see e.g. \cite[Appendix C.2]{Kna02}, and writing $\zeta_i$ as a linear combination of the simple roots $\beta_j$, we find that
$$ \zeta_1=\varpi_1-\tfrac{1}{2}\varpi_6, \quad \zeta_2=\varpi_3-\varpi_6, \quad \zeta_3=\varpi_4-\tfrac{3}{2}\varpi_6, \quad \zeta_4=\varpi_2-\tfrac{3}{4}\varpi_6, \quad \zeta_5=\varpi_5-\tfrac{5}{4}\varpi_6. $$
The spherical representations of $\frakg$ have highest weights $2k_1\varpi_1+2k_2\varpi_6$, $k_1,k_2\in\NN$, and the spherical representations of $\frakh$ have highest weights $2\ell_1\zeta_1+\ell_2\zeta_6$, $\ell_1\in\NN$ and $\ell_2\in\CC$.

By \cite{McKP81} we have
$$ F^\frakg(\varpi_1)|_{\so(9,1)} \simeq F^{\so(9,1)}(\zeta_1)\oplus F^{\so(9,1)}(\zeta_4)\oplus F^{\so(9,1)}(0). $$
To determine the action of the center of $\frakh$ on each summand we employ the weight space decomposition (using e.g. LiE) and find that the weights of $F^\frakg(\varpi_1)$ which are highest weights for $\so(9,1)$ are
$$ \varpi_1=\zeta_1+\tfrac{1}{2}\zeta_6, \qquad \varpi_2-\varpi_6=\zeta_4-\tfrac{1}{4}\zeta_6, \qquad \mbox{and} \qquad -\varpi_6=-\zeta_6, $$
so that
$$ F^\frakg(\varpi_1)|_\frakh \simeq \Big(F^{\so(9,1)}(\zeta_1)\boxtimes F^\RR(\tfrac{1}{2}\zeta_6)\Big)\oplus\Big(F^{\so(9,1)}(\zeta_4)\boxtimes F^\RR(-\tfrac{1}{4}\zeta_6)\Big)\oplus\Big(F^{\so(9,1)}(0)\boxtimes F^\RR(-\zeta_6)\Big). $$
Taking contragredient representations further gives
$$ F^\frakg(\varpi_6)|_\frakh \simeq \Big(F^{\so(9,1)}(\zeta_1)\boxtimes F^\RR(-\tfrac{1}{2}\zeta_6)\Big)\oplus\Big(F^{\so(9,1)}(\zeta_5)\boxtimes F^\RR(\tfrac{1}{4}\zeta_6)\Big)\oplus\Big(F^{\so(9,1)}(0)\boxtimes F^\RR(\zeta_6)\Big). $$
Hence,
$$ (2\varpi_1,2\zeta_1+\zeta_6),(2\varpi_6,2\zeta_1-\zeta_6),(2\varpi_1,-2\zeta_6),(2\varpi_6,2\zeta_6)\in\Lambda(\frakg,\frakh) $$
and they clearly generate $\fraka_G^\vee\times\fraka_H^\vee$, which is of dimension $4$.

\section{Lower multiplicity bounds -- Construction of symmetry breaking operators}\label{sec:ConstructionOfSBOs}

We use the results of Sections~\ref{sec:StructureTheory} and \ref{sec:FiniteDimBranching} to construct for every strongly spherical reductive pair $(G,H)$ intertwining operators between spherical principal series representations of $G$ and $H$ in terms of their distribution kernels (see Theorem~\ref{thm:MeromorphicContinuation}).

\subsection{Principal series representations and symmetry breaking operators}\label{sec:PrincipalSeriesAndSBOs}

For $(\xi,V_\xi)\in\widehat{M}_G$ and $\lambda\in\fraka_{G,\CC}^\vee$ we define the principal series representation
$$ \pi_{\xi,\lambda} = \Ind_{P_G}^G(\xi\otimes e^\lambda\otimes\1), $$
realized as the left-regular representation on the space $C^\infty(G/P,\calV_{\xi,\lambda})$ of smooth sections of the vector bundle $\calV_{\xi,\lambda}=G\times_{P_G}V_{\xi,\lambda}\to G/P_G$ associated to the representation $V_{\xi,\lambda}=\xi\otimes e^{\lambda+\rho_{\frakn_G}}\otimes\1$. Here, $\rho_{\frakn_G}=\frac{1}{2}\tr\ad_{\frakn_G}\in\fraka_G^\vee$, the half sum of all positive roots. Denote by $\calV_{\xi,\lambda}^*=\calV_{\xi^\vee,-\lambda}$ the dual bundle of $\calV_{\xi,\lambda}$.

Similarly, we define for $(\eta,W_\eta)\in\widehat{M}_H$ and $\nu\in\fraka_{H,\CC}^\vee$ principal series representations of $H$ by
$$ \tau_{\eta,\nu} = \Ind_{P_H}^H(\eta\otimes e^\nu\otimes\1) $$
and write $W_{\eta,\nu}=\eta\otimes e^{\nu+\rho_{\frakn_H}}\otimes\1$.

Consider the space
$$ \Hom_H(\pi_{\xi,\lambda}|_H,\tau_{\eta,\nu}) $$
of intertwining operators between principal series of $G$ and $H$, also referred to as \textit{symmetry breaking operators} by Kobayashi~\cite{Kob15}. By the Schwartz Kernel Theorem, symmetry breaking operators can be studied in terms of their distribution kernels.

\begin{proposition}[{\cite[Proposition 3.2]{KS15}}]\label{prop:IsoSBOSDistributionSections}
The natural map
$$ \Hom_H(\pi_{\xi,\lambda}|_H,\tau_{\eta,\nu}) \to \left(\calD'(G/P_G,\calV_{\xi,\lambda}^*)\otimes W_{\eta,\nu}\right)^{P_H}, \quad A\mapsto K^A, $$
which is characterized by $A\varphi(h)=\langle K^A,\varphi(h\blank)\rangle$, is an isomorphism.
\end{proposition}

As in \cite{KS15} we use generalized functions rather than distributions, so that $\calD'(G/P_G,\calV_{\xi,\lambda}^*)$ can be identified with the dual space of $C_c^\infty(G/P_G,\calV_{\xi,\lambda})$ and $L^1(G/P_G,\calV_{\xi,\lambda}^*)\hookrightarrow\calD'(G/P_G,\calV_{\xi,\lambda}^*)$.

To be able to apply the theory of Bernstein--Sato identities, we first translate line bundle valued distributions on $G/P_G$ to scalar-valued distributions on $G$. Consider the surjective linear map
$$ \flat:C_c^\infty(G)\otimes V_\xi \to C^\infty(G/P_G,\calV_{\xi,\lambda}), \quad \flat\varphi(g) = \int_{P_G}(\xi\otimes e^{\lambda+\rho_{\frakn_G}}\otimes\1)(p)\varphi(gp)\,dp, $$
where $dp$ denotes a left-invariant measure on $P_G$. Then its transpose $\flat^\top:\calD'(G/P_G,\calV_{\xi,\lambda}^*)\to\calD'(G)\otimes V_\xi^\vee$ is injective and embeds $(\calD'(G/P_G,\calV_{\xi,\lambda}^*)\otimes W_{\eta,\nu})^{P_H}$ into
\begin{multline*}
 \calD'(G)_{(\xi,\lambda),(\eta,\nu)} = \{u\in\calD'(G)\otimes\Hom_\CC(V_\xi,W_\eta):\\
 u(p_Hgp_G)=(\eta\otimes e^{\nu+\rho_{\frakn_H}}\otimes\1)(p_H)\circ u(g)\circ(\xi\otimes e^{\lambda-\rho_{\frakn_G}}\otimes\1)(p_G)\}.
\end{multline*}
We show that this map is actually an isomorphism:

\begin{lemma}\label{lem:IsoSBOsDistributionsOnG}
The natural map
$$ \left(\calD'(G/P_G,\calV_{\xi,\lambda}^*)\otimes W_{\eta,\nu}\right)^{P_H} \to \calD'(G)_{(\xi,\lambda),(\eta,\nu)} $$
is an isomorphism.
\end{lemma}

\begin{proof}
It suffices to show that the map is surjective, so let $u\in\calD'(G)_{(\xi,\lambda),(\eta,\nu)}$. Write $g=kan\in KA_GN_G=G$ for some maximal compact subgroup $K\subseteq G$ by the Iwasawa decomposition; then the distribution $a^{-(\lambda-\rho_G)}u(g)$ is right $A_GN_G$-invariant and therefore of the form $v\otimes1\otimes1$ according to the decomposition $G\simeq K\times A_G\times N_G$ with $v\in\calD'(K)\otimes\Hom_\CC(V_\xi,W_\eta)=(\calD'(K)\otimes V_\xi^\vee)\otimes W_\eta$. Define $w\in\calD'(G/P_G,\calV^*_{\xi,\lambda})\otimes W_{\eta,\nu}$ by
$$ \langle w,\varphi\rangle = \langle v,\varphi|_K\rangle \qquad \forall\,\varphi\in C^\infty(G/P_G,\calV_{\xi,\lambda}); $$
then it is easy to show that $w$ is $P_H$-invariant and maps to $u$ via the map $\flat^\top$.
\end{proof}

Now let us first consider the case where both $\xi\in\widehat{M}_G$ and $\eta\in\widehat{M}_H$ are the trivial representation. We use the notation
$$ \pi_\lambda = \pi_{\1,\lambda}, \qquad \tau_\nu = \tau_{\1,\nu} \qquad \mbox{and} \qquad \calD'(G)_{\lambda,\nu}=\calD'(G)_{(\1,\lambda),(\1,\nu)}. $$
For the statement of the main result of this section we denote by $\fraka_{G,+}^\vee$ and $\fraka_{H,+}^\vee$ the positive Weyl chambers corresponding to $\Sigma^+(\frakg,\fraka_G)$ and $\Sigma^+(\frakh,\fraka_H)$.

\begin{theorem}\label{thm:MeromorphicContinuation}
Assume that $(G,H)$ is a strongly spherical reductive pair such that $(\frakg,\frakh)$ is non-trivial and indecomposable. Then for every open double coset $\Omega=P_HgP_G\in P_H\backslash G/P_G$ there exists a family of distributions $K_{\lambda,\nu}\in\calD'(G)$, unique up to scalar multiples and depending holomorphically on $(\lambda,\nu)\in\fraka_{G,\CC}^\vee\times\fraka_{H,\CC}^\vee$, such that
\begin{enumerate}
\item $K_{\lambda,\nu}\in\calD'(G)_{\lambda,\nu}$ and $\supp K_{\lambda,\nu}\subseteq\overline{\Omega}$ for all $(\lambda,\nu)\in\fraka_{G,\CC}^\vee\times\fraka_{H,\CC}^\vee$,
\item $K_{\lambda,\nu}\in L^1_{\rm loc}(G)$ and $\supp K_{\lambda,\nu}=\overline{\Omega}$ for $\Re(\lambda-\rho_{\frakn_G})\in\fraka_{G,+}^\vee$ and $\Re(\nu+\rho_{\frakn_H})\in-\fraka_{H,+}^\vee$.
\end{enumerate}
\end{theorem}

By holomorphic dependence on the parameters $(\lambda,\nu)$ we mean that for every $\varphi\in C^\infty(G)$ the map $(\lambda,\nu)\mapsto \langle K_{\lambda,\nu},\varphi\rangle$ is a holomorphic function on $\fraka_{G,\CC}^\vee\times\fraka_{H,\CC}^\vee$.

The proof of Theorem~\ref{thm:MeromorphicContinuation} is carried out in Section~\ref{sec:BernsteinSatoIdentities}.

\subsection{Lower bounds for multiplicities}

Before we come to the proof, let us observe that Theorem~\ref{thm:MeromorphicContinuation} implies lower bounds for multiplicities of intertwining operators. For this denote by $(P_H\backslash G/P_G)_{\rm open}\subseteq P_H\backslash G/P_G$ the collection of open double cosets and by $\#(P_H\backslash G/P_G)_{\rm open}$ its cardinality (which is finite by Proposition~\ref{prop:CharacterizationStronglySpherical}).

\begin{corollary}\label{cor:LowerBoundsMultiplicities}
Assume that $(G,H)$ is a strongly spherical reductive pair such that $(\frakg,\frakh)$ is non-trivial and indecomposable. Then for all $(\lambda,\nu)\in\fraka_{G,\CC}^\vee\times\fraka_{H,\CC}^\vee$ we have
$$ \dim\Hom_H(\pi_\lambda|_H,\tau_\nu) \geq \#(P_H\backslash G/P_G)_{\rm open}. $$
\end{corollary}

\begin{proof}
We choose representatives $x_1,\ldots,x_n\in G$ for the open double cosets in $P_H\backslash G/P_G$. By Theorem~\ref{thm:MeromorphicContinuation} there exist holomorphic families of distributions $K_{\lambda,\nu}^i$, $1\leq i\leq n$, such that $\supp K_{\lambda,\nu}^i=\overline{P_Hx_iP_G}$ for $\Re(\lambda-\rho_{\frakn_G})\in\fraka_{G,+}^\vee$ and $\Re(\nu+\rho_{\frakn_H})\in-\fraka_{H,+}^\vee$. Let $\lambda_+\in\fraka_{G,+}^\vee$ and $\nu_+\in\fraka_{H,+}^\vee$; then for fixed $(\lambda_0,\nu_0)\in\fraka_{G,\CC}^\vee\times\fraka_{H,\CC}^\vee$ we have $\Re(\lambda_0+z\lambda_+)\in\fraka_{G,+}^\vee$ and $\Re(\nu_0-z\nu_+)\in-\fraka_{H,+}^\vee$ for $z\in\CC$, $\Re(z)\gg0$. This implies that the holomorphic functions
$$ f_i:\CC\to\calD'(G), \quad f_i(z) = K_{\lambda_0+z\lambda_+,\nu_0-z\nu_+}^i $$
satisfy $\supp f_i(z)=\overline{P_Hx_iP_G}$ whenever $\Re(z)\gg0$. Hence, $f_1(z),\ldots,f_n(z)$ are generically linearly independent and the claim follows from the next lemma combined with Proposition~\ref{prop:IsoSBOSDistributionSections} and Lemma~\ref{lem:IsoSBOsDistributionsOnG}.
\end{proof}

\begin{lemma}\label{lem:RenormHol}
Let $f_1,\ldots,f_n:\CC\to V$ be holomorphic functions with values in a complete locally convex topological vector space $V$. If $f_1(z),\ldots,f_n(z)$ are generically linearly independent, then there exists $A=(a_{ij})\in\GL(n,\CC)$ and $m_1,\ldots,m_n\geq0$ such that the functions
$$ g_i(z) = \sum_{j=1}^n a_{ij}z^{-m_i}f_i(z) $$
are holomorphic in $z=0$ and $g_1(0),\ldots,g_n(0)$ are linearly independent.
\end{lemma}

\begin{proof}
Rearrange $f_1,\ldots,f_n$ such that $f_1(0),\ldots,f_p(0)$ are linearly independent and the remaining $f_{p+1}(0),\ldots,f_n(0)$ are linear combinations of $f_1(0),\ldots,f_p(0)$. We now show that $f_{p+1}(z)$ can be replaced by a renormalized linear combination of $f_1(z),\ldots,f_{p+1}(z)$ such that $f_1(0),\ldots,f_{p+1}(0)$ are linearly independent. Applying this argument recursively to $f_{p+1},\ldots,f_n$ shows the statement.\\
So assume that
$$ f_{p+1}(0) = \sum_{i=1}^p\lambda_i^0f_i(0). $$
We form the new function $f_{p+1}^1(z)=f_{p+1}(z)-\sum_{i=1}^p\lambda_i^0f_i(z)$; then $f_{p+1}^1(0)=0$ and renormalizing
$$ \tilde{f}_{p+1}^1(z) = \frac{f_{p+1}^1(z)}{z} $$
gives a holomorphic function $\tilde{f}_{p+1}^1(z)$ such that $f_1(z),\ldots,f_p(z),\tilde{f}_{p+1}^1(z)$ are still generically linearly independent. If now $f_1(0),\ldots,f_p(0),\tilde{f}_{p+1}^1(0)$ are linearly independent, we are done, otherwise we have
$$ \tilde{f}_{p+1}^1(0) = \sum_{i=1}^p\lambda_i^1f_i(0). $$
As before, we form $f_{p+1}^2(z)=\tilde{f}_{p+1}^1(z)-\sum_{i=1}^p\lambda_i^1f_i(z)$; then $f_{p+1}^2(0)=0$ and renormalizing
$$ \tilde{f}_{p+1}^2(z) = \frac{f_{p+1}^2(z)}{z} $$
gives a holomorphic function $\tilde{f}_{p+1}^2(z)$ such that $f_1(z),\ldots,f_p(z),\tilde{f}_{p+1}^2(z)$ are still generically linearly independent. We repeat this procedure as long as possible. If the procedure does not terminate, we obtain $\lambda_i^0,\lambda_i^1,\lambda_i^2,\ldots\in\CC$, $1\leq i\leq p$, and holomorphic functions $\tilde{f}_{p+1}^k(z)$ with
$$ z\tilde{f}_{p+1}^{k+1}(z) = \tilde{f}_{p+1}^k(z) - \sum_{i=1}^p\lambda_i^kf_i(z), $$
where we put $\tilde{f}_{p+1}^0=f_{p+1}$. Form the holomorphic functions (the convergence of the sums follows by applying continuous linear functionals to the above identities)
$$ \lambda_i:\CC\to\CC, \quad \lambda_i(z) = \sum_{k=0}^\infty\lambda_i^kz^k, $$
then
$$ \sum_{i=1}^p\lambda_i(z)f_i(z) = \sum_{k=0}^\infty\sum_{i=1}^p\lambda_i^kf_i(z)z^k = \sum_{k=0}^\infty\left(\tilde{f}_{p+1}^k(z)-z\tilde{f}_{p+1}^{k+1}(z)\right)z^k = \tilde{f}_{p+1}^0(z) = f_{p+1}(z) $$
for all $z\in\CC$. But this implies that $f_1(z),\ldots,f_p(z),f_{p+1}(z)$ are linearly dependent for all $z\in\CC$, contradicting the assumption. Hence, the procedure has to terminate at some stage, which implies that $f_1(0),\ldots,f_p(0),\tilde{f}_{p+1}^k(0)$ are linearly independent for some $k$, proving our claim.
\end{proof}

\subsection{Bernstein--Sato identities and meromorphic extension}\label{sec:BernsteinSatoIdentities}

Write
$$ r = \rank_\RR(G) + \rank_\RR(H) = \dim\fraka_G + \dim\fraka_H. $$
By Proposition~\ref{prop:MatrixCoefficients} and Theorem~\ref{thm:EnoughWeights} there exists a basis $\{(\lambda_j,\nu_j)\}_{j=1,\ldots,r}\subseteq\Lambda^+(\frakg,\fraka)\times\Lambda^+(\frakh,\fraka_H)$ of $\fraka_G^\vee\times\fraka_H^\vee$ and non-zero real-analytic functions $F_j:G\to\RR$, $F_j\geq0$, such that
\begin{equation}
 F_j(m'a'n'gman) = a^{\lambda_j}(a')^{-\nu_j^\vee}F_j(g)\label{eq:EquivarianceFeta}
\end{equation}
for $g\in G$, $man\in P_G$ and $m'a'n'\in P_H$.

Now, for $(s_1,\ldots,s_r)\in\CC^r$ with $\Re(s_1),\ldots,\Re(s_r)\geq0$ we define
$$ K_{s_1,\ldots,s_r}(g) = F_1(g)^{s_1}\cdots F_r(g)^{s_m}, \qquad g\in G. $$
Let $\Omega\in P_H\backslash G/P_G$ be an open double coset and write $\chi_\Omega:G\to\{0,1\}$ for its characteristic function. Then it remains to show that $\chi_\Omega K_{s_1,\ldots,s_r}$ defines a family of distributions on $G$ which extends meromorphically in $(s_1,\ldots,s_r)\in\CC^r$. For this we use Bernstein--Sato identities.

By \cite{Sab87a,Sab87b} there exists a differential operator $D_{s_1,\ldots,s_r}$ on $G$ depending polynomially on $(s_1,\ldots,s_r)\in\CC^r$ such that
\begin{equation}
 D_{s_1,\ldots,s_r}K_{s_1,\ldots,s_r} = b(s_1,\ldots,s_r)K_{s_1-1,\ldots,s_r-1} \qquad \Re(s_1),\ldots,\Re(s_r)\gg0,\label{eq:BernsteinSatoIdentity}
\end{equation}
where
$$ b(s_1,\ldots,s_r) = c\cdot\prod_{i=1}^m \big(a_{i,1}s_1+\cdots+a_{i,r}s_r+b_i\big) $$
with $(a_{i,1},\ldots,a_{i,r})\in\CC^r\setminus\{0\}$, $b_i\in\CC$ and $c\neq0$. This identity immediately shows the meromorphic extension of the family $K_{s_1,\ldots,s_r}\in\calD'(G)$. However, if we want to meromorphically extend $\chi_\Omega K_{s_1,\ldots,s_r}$, we have to control $K_{s_1,\ldots,s_r}$ at the boundary $\partial\Omega$ of $\Omega$.

\begin{lemma}\label{lem:VanishingOnNonOpenOrbits}
For every double coset $\omega\subseteq\partial\Omega$ there exists $1\leq j\leq r$ such that
$$ F_j|_\omega=0. $$
\end{lemma}

\begin{proof}
Let $g\in\omega$; then by the equivariance property \eqref{eq:EquivarianceFeta} it suffices to show that $F_j(g)=0$ for some $1\leq j\leq r$. Since $\omega\subseteq\partial\Omega$, the double coset $P_HgP_G$ is not open. Using Theorem~\ref{thm:AHProjectionStabilizers} we find an element $Z=Z_M+Z_A+Z_N\in\frakp_H$ with $Z_A\neq0$ such that $X=\Ad(g^{-1})Z=X_M+X_A+X_N\in\frakp_G$. This implies
\begin{align*}
 e^{-t(s_1\nu_1^\vee+\cdots+s_r\nu_r^\vee)(Z_A)} K_{s_1,\ldots,s_r}(g) &= K_{s_1,\ldots,s_r}(e^{tZ}g) = K_{s_1,\ldots,s_r}(ge^{tX})\\
 &= e^{t(s_1\lambda_1+\cdots+s_r\lambda_r)(X_A)} K_{s_1,\ldots,s_r}(g)
\end{align*}
for all $t\in\RR$. Now assume that $F_j(g)\neq0$ for all $1\leq j\leq r$; then $K_{s_1,\ldots,s_r}(g)\neq0$ for all $(s_1,\ldots,s_r)\in\CC^r$ with $\Re(s_1),\ldots,\Re(s_r)\geq0$. Hence
$$ (s_1\lambda_1+\cdots+s_r\lambda_r)(X_A) = -(s_1\nu_1^\vee+\cdots+s_r\nu_r^\vee)(Z_A), $$
or equivalently
$$ s_1(\lambda_1(X_A)+\nu_1^\vee(Z_A))+\cdots+s_r(\lambda_r(X_A)+\nu_r^\vee(Z_A)) = 0 $$
for $(s_1,\ldots,s_r)\in\CC^r$ with $\Re(s_1),\ldots,\Re(s_r)\geq0$. This implies $\lambda_j(X_A)+\nu_j^\vee(Z_A)=0$ for all $1\leq j\leq r$. But since the pairs $(\lambda_j,\nu_j)$ form a basis of $\fraka_G^\vee\times\fraka_H^\vee$, the pairs $(\lambda_j,\nu_j^\vee)$ also form a basis and hence $Z_A=0$ which gives a contradiction.
\end{proof}

By the previous lemma we have $K_{s_1,\ldots,s_r}(g)=0$ whenever $g\in\partial\Omega$ and $\Re(s_1),\ldots,\Re(s_r)>0$. This implies a regularity statement for the functions $\chi_\Omega K_{s_1,\ldots,s_r}$:

\begin{lemma}\label{lem:RegularityStatement}
Let $M$ be a smooth manifold, $\Omega\subseteq M$ open and $\varphi\in C^\infty(M)$ with $\varphi\geq0$, $\varphi|_{\partial\Omega}=0$. Then $\chi_\Omega\varphi^s\in C^k(\RR)$ whenever $\Re s>k$.
\end{lemma}

The proof is an easy calculus exercise. We can now proceed to prove Theorem~\ref{thm:MeromorphicContinuation}.

\begin{proof}[Proof of Theorem~\ref{thm:MeromorphicContinuation}]
Let $\Omega\in P_H\backslash G/P_G$ be an open double coset, then $K_{s_1,\ldots,s_r}|_{\partial\Omega}=0$ for $\Re(s_1),\ldots,\Re(s_r)>0$ by Lemma~\ref{lem:VanishingOnNonOpenOrbits}. By Lemma~\ref{lem:RegularityStatement} this implies that $\chi_\Omega K_{s_1,\ldots,s_r}\in C^k(G)$ for $\Re(s_1),\ldots,\Re(s_r)>k$. Since the Bernstein--Sato operator $D_{s_1,\ldots,s_r}$ is of finite order, say $k\geq0$, this implies that $D_{s_1,\ldots,s_r}[\chi_\Omega K_{s_1,\ldots,s_r}]$ is a continuous function for $\Re(s_1),\ldots,\Re(s_r)>k$. Replacing $K_{s_1,\ldots,s_r}$ in the Bernstein--Sato identity \eqref{eq:BernsteinSatoIdentity} by $\chi_\Omega K_{s_1,\ldots,s_r}$ yields
\begin{equation}
 D_{s_1,\ldots,s_r}[\chi_\Omega K_{s_1,\ldots,s_r}] = b(s_1,\ldots,s_r)\chi_\Omega K_{s_1-1,\ldots,s_r-1}\label{eq:BernsteinSatoIdentity2}
\end{equation}
which holds on $\Omega$ by \eqref{eq:BernsteinSatoIdentity}, and trivially also on $K\setminus\overline{\Omega}$ since both sides vanish. As observed above, for $\Re(s_1),\ldots,\Re(s_r)>k$ both sides are continuous functions, and hence \eqref{eq:BernsteinSatoIdentity2} holds everywhere. This shows the meromorphic continuation of $\chi_\Omega K_{s_1,\ldots,s_r}\in\calD'(G)$ to $(s_1,\ldots,s_r)\in\CC^r$, or more precisely the holomorphic continuation of
$$ \prod_{i=1}^m\Gamma\big(a_{i,1}s_1+\cdots+a_{i,r}s_r+b_i\big)\cdot\chi_\Omega K_{s_1,\ldots,s_r}. $$
Finally, the change of variables
$$ \CC^r\to\fraka_{G,\CC}^\vee\times\fraka_{H,\CC}^\vee, \quad (s_1,\ldots,s_r)\mapsto(\lambda,\nu)=s_1(\lambda_1,-\nu_1^\vee)+\cdots+s_r(\lambda_r,-\nu_r^\vee)+(\rho_{\frakn_G},-\rho_{\frakn_H}) $$
constructs the desired family $K_{\lambda,\nu}$. The equivariance property extends by analytic continuation.
\end{proof}

\section{Upper multiplicity bounds -- Invariant distributions}\label{sec:InvariantDistributions}

In this section we establish upper bounds for the multiplicities $\dim\Hom_H(\pi_{\xi,\lambda}|_H,\tau_{\eta,\nu})$ using Bruhat's theory of invariant distributions. In particular, we obtain that for $\xi$ and $\eta$ the trivial representations, the intertwining operators constructed in Section~\ref{sec:ConstructionOfSBOs} form a basis of $\Hom_H(\pi_{\lambda}|_H,\tau_\nu)$ for generic $(\lambda,\nu)\in\fraka_{G,\CC}^\vee\times\fraka_{H,\CC}^\vee$.

\subsection{Upper bounds for multiplicities}\label{sec:UpperBoundsMultiplicities}

Let $M\subseteq M_G\cap M_H$ denote the stabilizer of a generic element $X\in\frakn^{-\sigma}$. More precisely,
$$ M = (M_G\cap M_H)^X = \{g\in M_G\cap M_H:\Ad(g)X=X\}, $$
where $X\in\frakn^{-\sigma}$ with $X=\sum_{\beta\in\Delta(\frakn^{-\sigma})} X_\beta$, $X_\beta\in\frakg^{-\sigma}(\fraka_H;\beta)\setminus\{0\}$. Since for every $\beta\in\Delta(\frakn^{-\sigma})$, the group $M_G\cap M_H$ acts either transitively on the unit sphere $S_\beta\subseteq\frakg^{-\sigma}(\fraka_H;\beta)$ or with possibly two orbits if $\dim\frakg^{-\sigma}(\fraka_H;\beta)=1$, it follows that for generic $X,Y\in\frakn^{-\sigma}$ the stabilizers $(M_G\cap M_H)^X$ and $(M_G\cap M_H)^Y$ are conjugate in $M_G\cap M_H$.

The main statement of this section is the following theorem:

\begin{theorem}\label{thm:MultiplicityBounds}
Assume that $(G,H)$ is a strongly spherical reductive pair such that $(\frakg,\frakh)$ is non-trivial and indecomposable. Then for $(\lambda,\nu)\in\fraka_{G,\CC}^\vee\times\fraka_{H,\CC}^\vee$ satisfying the generic condition
\begin{equation}
 w(\Re\lambda-\rho_{\frakn_G})|_{\fraka_H}-(\Re\nu-\rho_{\frakn_H}) \notin -w\left(\NN\mbox{-}\linspan\Sigma^+(\frakg,\fraka_G)\right)|_{\fraka_H} \qquad \forall\,w\in W(\fraka_G),\label{eq:GenericCondition}
\end{equation}
we have
$$ \dim\Hom_H(\pi_{\xi,\lambda}|_H,\tau_{\eta,\nu}) \leq \#(P_H\backslash G/P_G)_{\rm open}\cdot\dim\Hom_M(\xi^{\tilde{w}_0}|_M,\eta|_M). $$
In particular, for $\xi$ and $\eta$ the trivial representations, we have generically
$$ \dim\Hom_H(\pi_\lambda|_H,\tau_\nu) \leq \#(P_H\backslash G/P_G)_{\rm open}. $$
\end{theorem}

Here $\xi^g$ denotes the representation of $gM_Gg^{-1}$ given by $\xi^g(m)=\xi(g^{-1}mg)$.

\begin{remark}
If $\frakm_G\cap\frakm_H=\{0\}$ then condition \eqref{eq:GenericCondition} in Theorem~\ref{thm:MultiplicityBounds} can be replaced by the weaker condition
\begin{equation}
 w(\lambda-\rho_{\frakn_G})|_{\fraka_H}-(\nu-\rho_{\frakn_H}) \notin -w\left(\NN\mbox{-}\linspan\Sigma^+(\frakg,\fraka_G)\right)|_{\fraka_H} \qquad \forall w\in W(\fraka_G),\label{eq:GenericConditionWeak}
\end{equation}
This follows from the proof in Section~\ref{sec:InvDistributionsNonOpenOrbits}.
\end{remark}

\begin{remark}
From \cite[Corollary 2.7 and Lemma 5.2]{KO13} one can deduce the following estimate:
\begin{align*}
	\dim\Hom_H(\pi_\lambda|_H,\tau_\nu) &= \dim\Hom_{G\times H}(\pi_\lambda\otimes\tau_{-\nu},C^\infty((G\times H)/\diag(H)))\\
	&\leq \dim\Hom_{(\frakg\oplus\frakh,K_G\times K_H)}(\pi_\lambda\otimes\tau_{-\nu},C^\infty((G\times H)/\diag(H)))\\
	&\leq \# W(\fraka_G\oplus\fraka_H)
\end{align*}
for $\lambda\in\fraka_{G,\CC}^\vee$ and $-\nu\in\fraka_{H,\CC}^\vee$ satisfying a certain regularity and a positivity condition. Note that Theorem~\ref{thm:MultiplicityBounds} only requires a regularity condition and no positivity condition. Here $W(\fraka_G\oplus\fraka_H)\simeq W(\fraka_G)\times W(\fraka_H)$ denotes the Weyl group of the pair $(\frakg\oplus\frakh,\fraka_G\oplus\fraka_H)$ and by \cite[Corollary E]{KO13} we have $\# W(\fraka_G\oplus\fraka_H)\geq\#(P_H\backslash G/P_G)_{\rm open}$. Note that this inequality is not sharp. For instance, for the multiplicity one pairs in Fact~\ref{fact:MultiplicityOne} we have $\#(P_H\backslash G/P_G)_{\rm open}=1$ (see Lemma~\ref{lem:OpenOrbitsForMultOnePairs}) while the order of the Weyl group $W(\fraka_G\oplus\fraka_H)$ goes to infinity as the rank increases. Therefore, the estimate in Theorem~\ref{thm:MultiplicityBounds} is sharper than the one derived from \cite{KO13} and holds for an open dense subset of parameters $(\lambda,\nu)\in\fraka_{G,\CC}^\vee\times\fraka_{H,\CC}^\vee$, whereas \cite{KO13} requires the parameters to be contained in some Weyl chamber.
\end{remark}

Before we prove the theorem, note that combined with Corollary~\ref{cor:LowerBoundsMultiplicities} it implies the following formula for the generic multiplicities:

\begin{corollary}\label{cor:GenericMultiplicities}
Assume that $(G,H)$ is a strongly spherical reductive pair such that $(\frakg,\frakh)$ is non-trivial and indecomposable. Then for $(\lambda,\nu)\in\fraka_{G,\CC}^\vee\times\fraka_{H,\CC}^\vee$ satisfying the generic condition \eqref{eq:GenericCondition} we have
$$ \dim\Hom_H(\pi_\lambda|_H,\tau_\nu) = \#(P_H\backslash G/P_G)_{\rm open}. $$
\end{corollary}

The rest of this section is devoted to the proof of Theorem~\ref{thm:MultiplicityBounds}.

\subsection{Invariant distributions}\label{sec:InvariantDistrbutionsSub}

Recall from Proposition~\ref{prop:IsoSBOSDistributionSections} and Lemma~\ref{lem:IsoSBOsDistributionsOnG} that
$$ \dim\Hom_H(\pi_{\xi,\lambda}|_H,\tau_{\eta,\nu}) = \dim\calD'(G)_{(\xi,\lambda),(\eta,\nu)}. $$
Upper bounds for spaces of invariant distributions such as $\calD'(G)_{(\xi,\lambda),(\eta,\nu)}$ are provided by Bruhat's theory of invariant distributions (see e.g. \cite[Chapter 5.2]{War72} for a detailed account; see also \cite[Section 3.4]{MOO16} for an application in a similar setting). This method applies to our situation since the group $P_H\times P_G$ that acts on $G$ has only finitely many orbits. In this case, Bruhat's theory implies the following upper bound:
\begin{multline*}
 \dim\calD'(G)_{(\xi,\lambda),(\eta,\nu)} \leq \sum_{\substack{P_HgP_G\\\in P_H\backslash G/P_G}}\sum_{m=0}^\infty \dim\Hom_{P_H\cap gP_Gg^{-1}}\Big(\big(\xi\otimes e^{\lambda-\rho_{\frakn_G}}\otimes\1\big)^g,\\
 \big(\eta\otimes e^{\nu+\rho_{\frakn_H}}\otimes\1\big)\otimes\frac{\delta_{P_H}\delta_{gP_Gg^{-1}}}{\delta_{P_H\cap gP_Gg^{-1}}}\otimes S^m(V(g))\Big),
\end{multline*}
where for a Lie group $S$ we denote by $\delta_S(x)=|\det\Ad(x^{-1})|$, $x\in S$, its modular function, and $V(g)=\frakg/(\Ad(g)\frakp_G+\frakp_H)$. Clearly $\delta_{P_G}(man)=a^{-2\rho_{\frakn_G}}$ and $\delta_{P_H}(man)=a^{-2\rho_{\frakn_H}}$. We estimate the contributions from open and non-open orbits separately.

\subsubsection{Open orbits}

If $P_HgP_G$ is an open orbit, we may choose $g=e^X\tilde{w}_0$ with $X\in\frakn^{-\sigma}$ generic. Then $P_H\cap gP_Gg^{-1}=M$ and $(\xi\otimes e^{\lambda-\rho_{\frakn_G}}\otimes\1)^g|_M=\xi^{\tilde{w}_0}|_M$. Further, since $M$ is compact, $\delta_{P_H}=\delta_{gP_Gg^{-1}}=\delta_{P_H\cap gP_Gg^{-1}}=1$ on $M$. Finally, $S^m(V(g))=\{0\}$ for $m>0$, because $V(g)=\{0\}$ in this case, and $S^0(V(g))=\CC$ is the trivial representation. This implies that
\begin{multline*}
 \sum_{m=0}^\infty \dim\Hom_{P_H\cap gP_Gg^{-1}}\Big(\big(\xi\otimes e^{\lambda-\rho_{\frakn_G}}\otimes\1\big)^g,\big(\eta\otimes e^{\nu+\rho_{\frakn_H}}\otimes\1\big)\otimes\frac{\delta_{P_H}\delta_{gP_Gg^{-1}}}{\delta_{P_H\cap gP_Gg^{-1}}}\otimes S^m(V(g))\Big)\\
 = \dim\Hom_M(\xi^{\tilde{w}_0}|_M,\eta|_M).
\end{multline*}

\subsubsection{Non-open orbits}\label{sec:InvDistributionsNonOpenOrbits}

In this case it is more convenient to work with
\begin{multline}
 \Hom_{P_G\cap g^{-1}P_Hg}\Big(\big(\xi\otimes e^{\lambda-\rho_{\frakn_G}}\otimes\1\big),\big(\eta\otimes e^{\nu+\rho_{\frakn_H}}\otimes\1\big)^{g^{-1}}\otimes\frac{\delta_{P_G}\delta_{g^{-1}P_Hg}}{\delta_{P_G\cap g^{-1}P_Hg}}\otimes S^m(W(g))\Big)\\
 \simeq \Big(\big(\xi\otimes e^{\lambda-\rho_{\frakn_G}}\otimes\1\big)^\vee\otimes\big(\eta\otimes e^{\nu+\rho_{\frakn_H}}\otimes\1\big)^{g^{-1}}\otimes\frac{\delta_{P_G}\delta_{g^{-1}P_Hg}}{\delta_{P_G\cap g^{-1}P_Hg}}\otimes S^m(W(g))\Big)^{P_G\cap g^{-1}P_Hg},\label{eq:NonOpenOrbitsInvariants}
\end{multline}
where $W(g)=\frakg/(\frakp_G+\Ad(g^{-1})\frakp_H)$. If $P_HgP_G$ is a non-open orbit, then by Theorem~\ref{thm:AHProjectionStabilizers} the stabilizer $P_H\cap gP_Gg^{-1}$ of $gP_G$ in $P_H$ contains a one-parameter subgroup $\exp(\RR Z)$, where $Z=Z_M+Z_A+Z_N\in(\frakm_G\cap\frakm_H)+\fraka_H+\frakn_H$ with $Z_A\neq0$. Note that
$$ X = \Ad(g^{-1})Z = X_M + X_A + X_N \in \frakp_G\cap\Ad(g^{-1})\frakp_H $$
with $X_M=\Ad(\tilde{w}^{-1})Z_M\in\frakm_G$ and $X_A=\Ad(\tilde{w}^{-1})Z_A\in\fraka_G$, $X_A\neq0$. Then
\begin{align*}
 \big(\xi\otimes e^{\lambda-\rho_{\frakn_G}}\otimes\1\big)(e^{tX}) &= e^{t(\lambda-\rho_{\frakn_G})(X_A)}\cdot\xi(e^{tX_M}),\\
 \big(\eta\otimes e^{\nu+\rho_{\frakn_H}}\otimes\1\big)^{g^{-1}}(e^{tX}) &= e^{t(\nu+\rho_{\frakn_H})(Z_A)}\cdot\eta(e^{tZ_M}).
\end{align*}
Further, $\delta_{P_G}(e^{tX})=e^{-2t\rho_{\frakn_G}(X_A)}$ and $\delta_{g^{-1}P_Hg}(e^{tX})=e^{-2t\rho_{\frakn_H}(Z_A)}$.

\begin{lemma}\label{lem:ModularFctStabilizer}
$\delta_{P_G\cap g^{-1}P_Hg}(e^{tX})=\exp(-t\sum_{\alpha\in\Sigma^+(\frakg,\fraka_G)}n_\alpha\alpha(X_A))$ for some integers $0\leq n_\alpha\leq\dim\frakg(\fraka_G;\alpha)$.
\end{lemma}

\begin{proof}
We compute $\delta_{P_G\cap g^{-1}P_Hg}(e^X)$ for $X=X_M+X_A+X_N\in\frakm_G+\fraka_G+\frakn_G$ using the formula
$$ \delta_{P_G\cap g^{-1}P_Hg}(e^X) = \exp(-\tr\ad(X)|_{\frakp_G\cap\Ad(g^{-1})\frakp_H}). $$
Choose a basis $(Y_\alpha)_\alpha\cup(Y_\beta)_\beta\cup(Y_\gamma)_\gamma$ of $\frakp_G\cap\Ad(g^{-1})\frakp_H$ with the following properties:
\begin{enumerate}
\item $Y_\alpha=Y_{\alpha,M}+Y_{\alpha,A}+Y_{\alpha,N}\in\frakm_G+\fraka_G+\frakn_G$ such that $(Y_{\alpha,M})_\alpha$ is an orthogonal basis of the projection of $\frakp_G\cap\Ad(g^{-1})\frakp_H$ to $\frakm_G$.
\item $Y_\beta=Y_{\beta,A}+Y_{\beta,N}\in\fraka_G+\frakn_G$ such that $(Y_{\beta,A})_\beta$ is linearly independent.
\item $Y_\gamma=Y_{\gamma,N}\in\frakn_G$ such that $(Y_{\gamma,N})_\gamma$ is a basis of $\frakn_G\cap\Ad(g^{-1})\frakp_H$.
\end{enumerate}
Here orthogonal bases are taken with respect to any $\frakm_G$-invariant inner product on $\frakg$. We now study the action of $\ad(X)$ on $Y_\alpha$, $Y_\beta$ and $Y_\gamma$. First,
$$ \ad(X)Y_\alpha = [X_M,Y_{\alpha,M}] + [X_N,Y_{\alpha,M}+Y_{\alpha,A}] + [X,Y_{\alpha,N}] \in [X_M,Y_{\alpha,M}] + \frakn_G. $$
Note that $[X_M,Y_{\alpha,M}]\perp Y_{\alpha,M}$ since the inner product is $\frakm_G$-invariant. Hence, the coefficient of $Y_\alpha$ in the expression of $\ad(X)Y_\alpha$ as a linear combination of the basis elements is zero. This implies that the contribution of the basis elements $Y_\alpha$ to the trace is trivial. Next we have
$$ \ad(X)Y_\beta = [X_N,Y_{\beta,A}] + [X,Y_{\beta,N}] \in \frakn_G $$
so that $\ad(X)Y_\beta$ is a linear combination of the basis elements $Y_\gamma$. Therefore, also the basis elements $Y_\beta$ do not contribute to the trace and we have
$$ \tr\ad(X)|_{\frakp_G\cap\Ad(g^{-1})\frakp_H} = \tr\ad(X)|_{\frakn_G\cap\Ad(g^{-1})\frakp_H}. $$
We now specify a basis $(Y_\gamma)_\gamma$ of $\frakn_G\cap\Ad(g^{-1})\frakp_H$. Order the positive roots $\Sigma^+(\frakg,\fraka_G)=\{\alpha_1,\ldots,\alpha_n\}$ such that $[\frakn_G,\frakg(\fraka_G;\alpha_j)]\subseteq\bigoplus_{i=j+1}^n\frakg(\fraka_G;\alpha_i)$. Then we can choose the basis $\{Y_\gamma=Y_{j,k}:1\leq j\leq n,1\leq k\leq n_j\}$ such that 
\begin{itemize}
\item $Y_{j,k}=\sum_{i=j}^nY_{j,k}^i\in\bigoplus_{i=j}^n\frakg(\fraka_G;\alpha_i)$ for all $1\leq j\leq n$, $1\leq k\leq n_j$,
\item $(Y_{j,k}^j)_{k=1,\ldots,n_j}$ are mutually orthogonal with respect to a $\frakm_G$-invariant inner product on $\frakg(\fraka_G;\alpha_j)$.
\end{itemize}
We have
$$ \ad(X)Y_{j,k} = [X_M+X_A,Y_{j,k}]+[X_N,Y_{j,k}] \in [X_M+X_A,Y_{j,k}^j] + \bigoplus_{i=j+1}^n\frakg(\fraka_G;\alpha_i). $$
Since $[X_M,Y_{j,k}^j]\perp Y_{j,k}^j$ we obtain
$$ \ad(X)Y_{j,k} \in \alpha_j(X_A)Y_{j,k} + \bigoplus_{\ell\neq k}\CC Y_{j,\ell} + \bigoplus_{i=j+1}^n\frakg(\fraka_G;\alpha_i), $$
so that
\begin{equation*}
 \tr\ad(X)|_{\frakn_G\cap\Ad(g^{-1})\frakp_H} = \sum_{j=1}^n n_j\alpha_j(X_A).\qedhere
\end{equation*}
\end{proof}

\begin{lemma}\label{lem:NonOpenOrbitsSymPowers}
Let $P\in S^m(W(g))$, $P\neq0$, with $\ad(X)P=\lambda P$ for some $\lambda\in\CC$; then we have $\Re\lambda=-\sum_{\alpha\in\Sigma^+(\frakg,\fraka_G)}m_\alpha\alpha(X_A)$ for some integers $m_\alpha\geq0$. If additionally $X_M=0$ then $\lambda=-\sum_{\alpha\in\Sigma^+(\frakg,\fraka_G)}m_\alpha\alpha(X_A)$.
\end{lemma}

\begin{proof}
First note that $W(g)=\frakg/(\frakp_G+\Ad(g)^{-1}\frakp_H)\simeq\overline{\frakn}_G/(\overline{\frakn}_G\cap(\frakp_G+\Ad(g)^{-1}\frakp_H))$. We order the positive roots $\Sigma^+(\frakg,\fraka_G)=\{\alpha_1,\ldots,\alpha_n\}$ such that $[\frakn_G,\frakg(\fraka_G;-\alpha_j)]\subseteq\bigoplus_{i=j+1}^n\frakg(\fraka_G;-\alpha_i)$. Then we can choose a basis $\{Y_{j,k}:1\leq j\leq n,1\leq k\leq n_j\}$ of $\overline{\frakn}_G\cap(\frakp_G+\Ad(g)^{-1}\frakp_H)$ such that
\begin{itemize}
\item $Y_{j,k}=\sum_{i=j}^nY_{j,k}^i\in\bigoplus_{i=j}^n\frakg(\fraka_G;-\alpha_i)$ for all $1\leq j\leq n$, $1\leq k\leq n_j$,
\item $(Y_{j,k}^j)_{k=1,\ldots,n_j}\subseteq\frakg(\fraka_G;-\alpha_j)$ are linearly independent for every $1\leq j\leq n$.
\end{itemize}
For every $1\leq j\leq n$ we extend $(Y_{j,k}^j)_{k=1,\ldots,n_j}$ to a $\CC$-basis $(Y_{j,k}^j)_{k=1,\ldots,n_j}\cup(Z_{j,k})_{k=1,\ldots,m_j}$ of $\frakg(\fraka_G;-\alpha_j)_\CC$; then the equivalence classes $\overline{Z}_{j,k}=Z_{j,k}+(\overline{\frakn}_G\cap(\frakp_G+\Ad(g)^{-1}\frakp_H))$, $1\leq j\leq n$, $1\leq k\leq m_j$, form a basis of $\overline{\frakn}_G/(\overline{\frakn}_G\cap(\frakp_G+\Ad(g)^{-1}\frakp_H))$. Since $X_M\in\frakm_G$ is contained in a maximal torus in $\frakm_G$, we may assume that every $Z_{j,k}$ is an eigenvector of $\ad(X_M)$ with imaginary eigenvalue: $\ad(X_M)Z_{j,k}=\sqrt{-1}\lambda_{j,k}Z_{j,k}$. Further, by definition we have $\ad(X_A)Z_{j,k}=\alpha_j(X_A)Z_{j,k}$, so that
$$ \ad(X)Z_{j,k} \in (\alpha_j(X_A)+\sqrt{-1}\lambda_{j,k})Z_{j,k} + \bigoplus_{i=j+1}^n\frakg(\fraka_G;-\alpha_i)_\CC. $$
Now assume $P\in S^m(W(g))$, $P\neq0$, with $\ad(X)P=\lambda P$. Write $P$ as a linear combination of the basis elements $\overline{Z}_{j_1,k_1}\cdots\overline{Z}_{j_m,k_m}$ of $S^m(W(g))$; then there exist $1\leq j_1\leq\ldots\leq j_m\leq n$ and $1\leq k_i\leq n_{j_i}$ such that the coefficient of $\overline{Z}_{j_1,k_1}\cdots\overline{Z}_{j_m,k_m}$ in $P$ is non-zero. Choose such $(j_1,k_1),\ldots,(j_m,k_m)$ with the property that $j_1,\ldots,j_m$ are minimal. Considering only the coefficient of $\overline{Z}_{j_1,k_1}\cdots\overline{Z}_{j_m,k_m}$ in $\ad(X)P=\lambda P$ we obtain
$$ \lambda = \alpha_{j_1}(X_A)+\cdots+\alpha_{j_m}(X_A)+\sqrt{-1}\lambda_{j_1,k_1}+\cdots\sqrt{-1}\lambda_{j_m,k_m}, $$
which implies the claim.
\end{proof}

Now assume that $\phi\in V_\xi^\vee\otimes W_\eta\otimes S^m(W(g))$, $\phi\neq0$, is contained in the space of invariants \eqref{eq:NonOpenOrbitsInvariants}. Since $X_M$ resp. $Z_M$ is contained in a maximal torus in $\frakm_G$ resp. $\frakm_H$, there exist bases $(v_i)_i$ of $V_\xi^\vee$ and $(w_j)_j$ of $W_\eta$ such that $\xi^\vee(X_M)v_i=\sqrt{-1}a_iv_i$ and $\eta(Z_M)w_j=\sqrt{-1}b_jw_j$ with $a_i,b_j\in\RR$. Write $\phi$ in terms of this basis as
$$ \phi = \sum_{i,j}v_i\otimes w_j\otimes P_{i,j} $$
with $P_{i,j}\in S^m(W(g))$; then $X\cdot\phi=0$ implies, using Lemma~\ref{lem:ModularFctStabilizer}
\begin{multline*}
 \ad(X)P_{i,j} = -\Big(a_i\sqrt{-1}-(\lambda-\rho_{\frakn_G})(X_A)+b_j\sqrt{-1}+(\nu+\rho_{\frakn_H})(wX_A)\\
 -2\rho_{\frakn_G}(X_A)-2\rho_{\frakn_H}(wX_A)+\sum_{\alpha\in\Sigma^+(\frakg,\fraka_G)}n_\alpha\alpha(X_A)\Big)P_{i,j}.
\end{multline*}
At least one $P_{i,j}$ is non-trivial and hence it follows from Lemma~\ref{lem:NonOpenOrbitsSymPowers} that
\begin{multline*}
 w(\lambda-\rho_{\frakn_G})(Z_A)-(\nu-\rho_{\frakn_H})(Z_A)\\
 = -\sum_{\alpha\in\Sigma^+(\frakg,\fraka_G)}\underbrace{\big(m_\alpha+(\dim\frakg(\fraka_G;\alpha)-n_\alpha)\big)}_{\geq0}w\alpha(Z_A)+(a_i+b_j)\sqrt{-1}.
\end{multline*}
Hence, the space \eqref{eq:NonOpenOrbitsInvariants} of invariants must be trivial if \eqref{eq:GenericCondition} is satisfied. This completes the proof of Theorem~\ref{thm:MultiplicityBounds}.

\section{Application to Shintani functions}\label{sec:ShintaniFunctions}

In \cite{Kob14} Kobayashi established a connection between symmetry breaking operators and Shintani functions of a pair $(G,H)$ of real reductive groups. Combining his results with Corollary~\ref{cor:LowerBoundsMultiplicities} we prove lower bounds for the dimension of the space of Shintani functions (see Theorem~\ref{thm:LowerBoundsShintaniFunctions}).

\subsection{Shintani functions for real reductive groups}

Let $\frakt_G\subseteq\frakm_G$ and $\frakt_H\subseteq\frakm_H$ be Cartan subalgebras of $\frakm_G$ and $\frakm_H$; then $\frakj_G=\frakt_G+\fraka_G\subseteq\frakg$ and $\frakj_H=\frakt_H+\fraka_H\subseteq\frakh$ are maximally split Cartan subalgebras of $\frakg$ and $\frakh$. We identify $\frakj_{G,\CC}^\vee\simeq\frakt_{G,\CC}^\vee\oplus\fraka_{G,\CC}^\vee$ and $\frakj_{H,\CC}^\vee\simeq\frakt_{H,\CC}^\vee\oplus\fraka_{H,\CC}^\vee$. Let us choose positive systems $\Sigma^+(\frakg_\CC,\frakj_{G,\CC})\subseteq\Sigma(\frakg_\CC,\frakj_{G,\CC})$ and $\Sigma^+(\frakh_\CC,\frakj_{H,\CC})\subseteq\Sigma(\frakh_\CC,\frakj_{H,\CC})$ such that the restriction of a positive root to $\fraka_G$ resp. $\fraka_H$ is either zero or contained in $\Sigma^+(\frakg,\fraka_G)$ resp. $\Sigma^+(\frakh,\fraka_H)$. Write $\rho_\frakg$ resp. $\rho_\frakh$ for the half sum of all roots in $\Sigma^+(\frakg_\CC,\frakj_{G,\CC})$ resp. $\Sigma^+(\frakh_\CC,\frakj_{H,\CC})$. Then $\rho_\frakg=\rho_{\frakm_G}+\rho_{\frakn_G}$ with $\rho_{\frakm_G}=\rho_\frakg|_{\frakt_G}$, and similarly $\rho_\frakh=\rho_{\frakm_H}+\rho_{\frakn_H}$. Further, let $W(\frakj_{G,\CC})$ and $W(\frakj_{H,\CC})$ denote the Weyl groups of $\Sigma(\frakg_\CC,\frakj_{G,\CC})$ and $\Sigma(\frakh_\CC,\frakj_{H,\CC})$.

The Harish-Chandra isomorphism provides a natural identification
$$ \frakj_{G,\CC}^\vee/W(\frakj_{G,\CC}) \stackrel{\sim}{\longrightarrow} \Hom_{\CC-\rm alg}(Z(\frakg_\CC),\CC), \quad \Lambda\mapsto\chi_\Lambda, $$
where $Z(\frakg_\CC)$ denotes the center of the universal enveloping algebra $U(\frakg_\CC)$ of $\frakg_\CC$. We use the same notation $\chi_\Xi\in\Hom_{\CC-\rm alg}(Z(\frakh_\CC),\CC)$ for $\Xi\in\frakj_{H,\CC}^\vee$. The left resp. right action of $U(\frakg_\CC)$ on $C^\infty(G)$ will be denoted by $L_u$ resp. $R_u$, $u\in U(\frakg_\CC)$.

\begin{definition}[{see \cite[Definition 1.1]{Kob14}, \cite{MS96}}]
A function $f\in C^\infty(G)$ is called a \textit{Shintani function of type $(\Lambda,\Xi)\in\frakj_{G,\CC}^\vee\times\frakj_{H,\CC}^\vee$} if $f$ satisfies the following three properties:
\begin{enumerate}
\item $f(k'gk)=f(g)$ for any $k'\in H\cap K$, $k\in K$.
\item $R_uf=\chi_\Lambda(u)f$ for any $u\in Z(\frakg_\CC)$.
\item $L_vf=\chi_\Xi(v)f$ for any $v\in Z(\frakh_\CC)$.
\end{enumerate}
The space of Shintani functions of type $(\Lambda,\Xi)$ is denoted by $\Sh(\Lambda,\Xi)$. We write $\Sh_{\rm mod}(\Lambda,\Xi)$ for the subspace of Shintani functions of moderate growth.
\end{definition}

For the definition of \textit{moderate growth} see e.g. \cite[Definition 3.3]{Kob14}. The following result is due to Kobayashi:

\begin{fact}[{see \cite[Theorem 1.2 and 8.1]{Kob14}}]
Let $(G,H)$ be a pair of real reductive algebraic groups.
\begin{enumerate}
\item $\dim\Sh(\Lambda,\Xi)<\infty$ for all $(\Lambda,\Xi)\in\frakj_{G,\CC}^\vee\times\frakj_{H,\CC}^\vee$ if and only if $(G,H)$ is strongly spherical.
\item $\dim\Sh(\Lambda,\Xi)\neq0$ implies $\Lambda\in W(\frakj_{G,\CC})(\rho_{\frakm_G}+\fraka_{G,\CC}^\vee)$ and $\Xi\in W(\frakj_{H,\CC})(\rho_{\frakm_H}+\fraka_{H,\CC}^\vee)$.
\end{enumerate}
\end{fact}

In view of this statement we abuse notation and write
$$ \Sh(\lambda,\nu) = \Sh(\rho_{\frakm_G}+\lambda,\rho_{\frakm_H}+\nu), \qquad (\lambda,\nu)\in\fraka_{G,\CC}^\vee\times\fraka_{H,\CC}^\vee. $$

\subsection{Lower bounds for $\dim\Sh(\lambda,\nu)$}

The following statement gives more detailed information about $\dim\Sh(\lambda,\nu)$:

\begin{theorem}\label{thm:LowerBoundsShintaniFunctions}
Assume that $(G,H)$ is a strongly spherical reductive pair such that $(\frakg,\frakh)$ is non-trivial and indecomposable. Then for all $(\lambda,\nu)\in\fraka_{G,\CC}^\vee\times\fraka_{H,\CC}^\vee$ we have
$$ \dim\Sh(\lambda,\nu)\geq\dim\Sh_{\rm mod}(\lambda,\nu)\geq\#(P_H\backslash G/P_G)_{\rm open}, $$
and for generic $(\lambda,\nu)\in\fraka_{G,\CC}^\vee\times\fraka_{H,\CC}^\vee$ we have
$$ \dim\Sh_{\rm mod}(\lambda,\nu) = \#(P_H\backslash G/P_G)_{\rm open}. $$
\end{theorem}

\begin{proof}
Let $\lambda_+\in W(\fraka_G)\lambda$ and $\nu_-\in W(\fraka_H)(-\nu)$ such that $\Re\langle\lambda_+,\alpha\rangle\geq0$ for all $\alpha\in\Sigma^+(\frakg,\fraka_G)$ and $\Re\langle\nu_-,\beta\rangle\leq0$ for all $\beta\in\Sigma^+(\frakh,\fraka_H)$. Then by \cite[Theorem 8.2 and Remark 8.3]{Kob14} there exists a natural embedding
$$ \Hom_H(\pi_{\lambda_+}|_H,\tau_{\nu_-}) \hookrightarrow \Sh_{\rm mod}(\lambda,\nu) $$
which is an isomorphism for generic $(\lambda,\nu)$. Now the statement follows from Corollaries~\ref{cor:LowerBoundsMultiplicities} and \ref{cor:GenericMultiplicities}.
\end{proof}

\begin{remark}
According to Kobayashi~\cite[Remark 10.2~(4)]{Kob14} it is plausible that $\Sh_{\rm mod}(\lambda,\nu)=\Sh(\lambda,\nu)$ for all $(\lambda,\nu)\in\frakj_{G,\CC}^\vee\times\frakj_{H,\CC}^\vee$ if the pair $(G,H)$ is strongly spherical. In this case, Theorem~\ref{thm:LowerBoundsShintaniFunctions} would imply
$$ \dim\Sh(\lambda,\nu) = \#(P_H\backslash G/P_G)_{\rm open} $$
for generic $(\lambda,\nu)$.
\end{remark}

\begin{remark}
For some strongly spherical pairs $(G,H)$ of low rank, the space of Shintani functions was studied in more detail, and in some cases its dimension was computed for all parameters $(\lambda,\nu)\in\fraka_{G,\CC}^\vee\times\fraka_{H,\CC}^\vee$. The pairs $(\GL(2,\FF),\GL(1,\FF)\times\GL(1,\FF))$, $\FF=\RR,\CC$, were studied by Hirano~\cite{Hir00,Hir01}, the pairs $(\upO(1,n+1),\upO(1,n))$ by Kobayashi~\cite{Kob14}, the pairs $(\upU(1,n+1),\upU(1,n)\times\upU(1))$ by Tsuzuki~\cite{Tsu01a,Tsu01b,Tsu02}, the pairs $(\Sp(2,\RR),\Sp(1,\RR)\times\Sp(1,\RR))$ and $(\Sp(2,\RR),\SL(2,\CC))$ by Moriyama~\cite{Mor99,Mor02}. and the pair $(\SL(3,\RR),\GL(2,\RR))$ by Sono~\cite{Son11}.
\end{remark}

\section{Multiplicity one pairs}\label{sec:MultiplicityOnePairs}

For the pairs
\begin{equation}
\begin{gathered}
 (\GL(n+1,\CC),\GL(n,\CC)), \quad (\GL(n+1,\RR),\GL(n,\RR)), \quad (\upU(p,q+1),\upU(p,q)),\\
 (\SO(n+1,\CC),\SO(n,\CC)), \quad (\SO(p,q+1),\SO(p,q))
\end{gathered}\label{eq:MultOnePairs}
\end{equation}
it was proven by Sun--Zhu~\cite{SZ12} that $\dim\Hom_H(\pi|_H,\tau)\leq1$ for all irreducible smooth admissible Fréchet representations $\pi$ of $G$ and $\tau$ of $H$ of moderate growth. In this section we determine for which principal series representations $\pi=\pi_{\xi,\lambda}$ and $\tau=\tau_{\eta,\nu}$ the multiplicities are $=1$, at least for generic $(\lambda,\nu)\in\fraka_{G,\CC}^\vee\times\fraka_{H,\CC}^\vee$. The main tool to pass from spherical principal series $\pi_\lambda$ and $\tau_\nu$ to general vector-valued principal series $\pi_{\xi,\lambda}$ and $\tau_{\eta,\nu}$ is a variant of the Jantzen--Zuckerman translation principle. This method was previously used in the construction of symmetry breaking operators in \cite{MO17} and we briefly recall the main steps, adapted to our setting.

\subsection{The translation principle}

Let $(\varphi,E)$ be an irreducible finite-dimensional representation of $G$; then the space
$$ E' := E^{\overline{\frakn}_G} = \{v\in E:\varphi(X)v=0\,\forall\,X\in\overline{\frakn}_G\} $$
of $\overline{\frakn}_G$-fixed vectors is an irreducible representation of $M_GA_G$, and we obtain a $\overline{P}_G$-equivariant embedding $E'\hookrightarrow E$. Similarly, we let $(\psi,F)$ be an irreducible finite-dimensional representation of $H$ and consider the irreducible $M_HA_H$-representation $F'=F^{\overline{\frakn}_H}$. Then the projection $F\twoheadrightarrow F'$ onto the highest restricted weight space is $P_H$-equivariant, where we let $N_H$ act trivially on $F'$. We assume that $\Hom_H(E|_H,F)\neq\{0\}$.

Now suppose we are given an $H$-intertwining operator
$$ A:\Ind_{P_G}^G(V)\to\Ind_{P_H}^H(W) $$
for $V$ and $W$ finite-dimensional representations of $P_G$ and $P_H$. Tensoring with a non-trivial $H$-intertwiner $\eta:E\to F$ gives an intertwiner
\begin{equation}
 A\otimes\eta:\Ind_{P_G}^G(V)\otimes E\to\Ind_{P_H}^H(W)\otimes F,\label{eq:IntertwinerTensorEta}
\end{equation}
from which we want to construct an intertwiner between certain vector-valued principal series.

For this, we first consider the representation $\Ind_{P_G}^G(V)\otimes E$. Since $\tilde{w}_0P_G\tilde{w}_0^{-1}=\overline{P}_G$, the map $f\mapsto f(\blank\tilde{w}_0^{-1})$ defines an isomorphism
$$ \Ind_{\overline{P}_G}^G(V^{\tilde{w}_0})\to\Ind_{P_G}^G(V). $$
Next, we make use of the isomorphism
$$ \iota:\Ind_{\overline{P}_G}^G(V^{\tilde{w}_0}\otimes E|_{\overline{P}_G})\to\Ind_{\overline{P}_G}^G(V^{\tilde{w}_0})\otimes E, \quad \iota f(g)=(\id_V\otimes\varphi(g))f(g), $$
where we view both sides as $(V\otimes E)$-valued smooth functions on $G$. Now, the $\overline{P}_G$-equivariant embedding $E'\hookrightarrow E$ induces an embedding
$$ \Ind_{\overline{P}_G}^G(V^{\tilde{w}_0}\otimes E')\hookrightarrow\Ind_{\overline{P}_G}^G(V^{\tilde{w}_0}\otimes E|_{\overline{P}_G}) $$
which we compose with the isomorphism
$$ \Ind_{P_G}^G(V\otimes(E')^{\tilde{w}_0^{-1}})\to\Ind_{\overline{P}_G}^G(V^{\tilde{w}_0}\otimes E'), \quad f\mapsto f(\blank\tilde{w}_0) $$
to obtain an embedding
\begin{equation}
 \Ind_{P_G}^G(V\otimes(E')^{\tilde{w}_0^{-1}}) \hookrightarrow \Ind_{P_G}^G(V)\otimes E.\label{eq:EmbeddingE'}
\end{equation}

Similarly, without having to invoke $\tilde{w}_0$, we have a surjection
\begin{equation}
 \Ind_{P_H}^H(W)\otimes F \stackrel{\sim}{\to} \Ind_{P_H}^H(W\otimes F|_{P_H}) \twoheadrightarrow \Ind_{P_H}^H(W\otimes F'),\label{eq:SurjectionE''}
\end{equation}
where the second map is induced by the $P_H$-equivariant projection $F\to F'$.

Finally, composing \eqref{eq:IntertwinerTensorEta} with the embedding \eqref{eq:EmbeddingE'} and the surjection \eqref{eq:SurjectionE''} defines an $H$-intertwining operator
$$ \Phi(A):\Ind_{P_G}^G(V\otimes(E')^{\tilde{w}_0^{-1}})\to\Ind_{P_H}^H(W\otimes F'). $$

Now let $V=\1\otimes e^\lambda\otimes\1$ and $W=\1\otimes e^\nu\otimes\1$, and write $(E')^{\tilde{w}_0^{-1}}\simeq\xi\otimes e^{\lambda_0}\otimes\1$ and $F'=\eta\otimes e^{\nu_0}\otimes\1$ for $\xi\in\widehat{M}_G$, $\eta\in\widehat{M}_H$ and $(\lambda_0,\nu_0)\in\fraka_G^\vee\times\fraka_H^\vee$. Then $A\mapsto\Phi(A)$ defines a map
$$ \Phi:\Hom_H(\pi_\lambda|_H,\tau_\nu) \to \Hom_H(\pi_{\xi,\lambda+\lambda_0}|_H,\tau_{\eta,\nu+\nu_0}). $$
Recall that by Proposition~\ref{prop:IsoSBOSDistributionSections} every intertwining operator $A\in\Hom_H(\pi_{\xi,\lambda}|_H,\tau_{\eta,\nu})$ is given by a distribution kernel $K_A\in(\calD'(G/P_G,\calV_{\xi,\lambda}^*)\otimes W_{\eta,\nu})^{P_H}$. To describe how the distribution kernel of $\Phi(A)$ arises from the distribution kernel of $A$ we denote by $i:E'\hookrightarrow E$ the $\overline{P}_G$-equivariant embedding and by $p:F\twoheadrightarrow F'$ the $P_H$-equivariant quotient. Further let $\calE'=G\times_{P_G}(E')^{\tilde{w}_0^{-1}}$ and denote by $(\calE')^\vee=G\times_{P_G}((E')^{\tilde{w}_0^{-1}})^\vee$ the contragredient bundle.

\begin{theorem}[{\cite[Section 2]{MO17}}]\label{thm:TranslationPrinciple}
The map
$$ \Phi:\Hom_H(\pi_\lambda|_H,\tau_\nu) \to \Hom_H(\pi_{\xi,\lambda+\lambda_0}|_H,\tau_{\eta,\nu+\nu_0}) $$
has the property that the distribution kernel $K_{\Phi(A)}\in(\calD'(G/P_G,\calV_{\xi,\lambda+\lambda_0}^*)\otimes W_{\eta,\nu+\nu_0})^{P_H}$ is given in terms of the distribution kernel $K_A\in(\calD'(G/P_G,\calV_{\lambda}^*)\otimes W_\nu)^{P_H}$ as
$$ K_{\Phi(A)} = \sigma\otimes K_A, $$
where $\sigma\in(C^\infty(G/P_G,(\calE')^\vee)\otimes F')^{P_H}$ is given by
$$ \sigma(g) = p\circ\eta\circ\varphi(g\tilde{w}_0)\circ i \in \Hom_\CC(E',F')\simeq(E')^\vee\otimes F', \qquad g\in G. $$
\end{theorem}

Note that in the statement we identify $\calV_\lambda^*\otimes(\calE')^\vee\simeq\calV_{\xi,\lambda+\lambda_0}^*$ and $W_\nu\otimes F'\simeq W_{\eta,\nu+\nu_0}$.

\begin{corollary}\label{cor:TranslationPrincipleSupport}
Let $\xi\in\widehat{M}_G$ and $\eta\in\widehat{M}_H$ and assume that there exist irreducible finite-dimensional representations $E$ of $G$ and $F$ of $H$ such that $\xi^{\tilde{w}_0}\simeq E^{\frakn_G}|_{M_G}$, $\eta\simeq F^{\frakn_H}|_{M_H}$ and $\Hom_H(E|_H,F)\neq\{0\}$. Then there exist $(\lambda_0,\nu_0)\in\fraka_G^\vee\times\fraka_H^\vee$ and a linear map
$$ \Phi:\Hom_H(\pi_\lambda|_H,\tau_\nu) \to \Hom_H(\pi_{\xi,\lambda+\lambda_0}|_H,\tau_{\eta,\nu+\nu_0}) $$
which is on the level of distribution kernels given by tensoring with a fixed non-trivial real-analytic section. In particular, $\supp(K_{\Phi(A)})=G/P_G$ whenever $\supp K_A=G/P_G$.
\end{corollary}

\begin{proof}
Replacing $E$ and $F$ by their twists with the Cartan involution $\theta$ of $G$ we may assume that $\xi^{\tilde{w}_0}\simeq E^{\overline{\frakn}_G}|_{M_G}$, $\eta\simeq F^{\overline{\frakn}_H}|_{M_H}$ and $\Hom_H(E|_H,F)\neq\{0\}$. Then $E'=E^{\overline{\frakn}_G}$ and $F'=F^{\overline{\frakn}_H}$ satisfy $(E')^{\tilde{w}_0^{-1}}|_{M_G}\simeq\xi$ and $F'|_{M_H}\simeq\eta$, so that the statement follows from Theorem~\ref{thm:TranslationPrinciple}. It remains to show that $\sigma$ is a non-trivial real-analytic section. That $\sigma$ is non-trivial is a consequence of the irreducibility of $E$. Further, as a matrix coefficient of a finite-dimensional representation it is clearly real analytic.
\end{proof}

\subsection{Iwasawa decompositions}\label{sec:Iwasawa}

To construct finite-dimensional representations $E$ of $G$ and $F$ of $H$ for the translation principle, we first give the explicit Iwasawa decompositions for all pairs $(G,H)$ in \eqref{eq:MultOnePairs} (following the notation of Section~\ref{sec:StructureStronglySphericalReductivePairs}), compute the generic stabilizer $M\subseteq M_G\cap M_H$ of an element of $\frakn^{-\sigma}$ defined in Section~\ref{sec:UpperBoundsMultiplicities}, and prove in particular the following statement:

\begin{lemma}\label{lem:OpenOrbitsForMultOnePairs}
Let $(G,H)$ be one of the pairs in \eqref{eq:MultOnePairs}; then $\#(P_H\backslash G/P_G)_{\rm open}=1$.
\end{lemma}

Note that Lemma~\ref{lem:OpenOrbitsForMultOnePairs} also follows from Corollary~\ref{cor:GenericMultiplicities} combined with Fact~\ref{fact:MultiplicityOne}. However, since we need the relevant structure theory of the pairs $(G,H)$ in Section~\ref{sec:FinDimBranching}, we include an independent proof here.

\subsubsection{$(G,H)=(\GL(n+1,\CC),\GL(n,\CC))$}

Let $G=\GL(n+1,\CC)$, $n\geq2$, and define a Cartan involution on $G$ by $\theta(g)=(g^*)^{-1}$; then $K=G^\theta=\upU(n+1)$. We embed $H=\GL(n,\CC)$ in the upper-left corner of $G$; then $\theta$ leaves $H$ invariant and $H\cap K=H^\theta=\upU(n)$. We choose the maximal abelian subalgebra
$$ \fraka_H = \{\diag(t_1,\ldots,t_n,0):t_i\in\RR\} $$
of $\frakh^{-\theta}$ and extend it to the maximal abelian subalgebra
$$ \fraka_G = \{\diag(t_1,\ldots,t_n,t_{n+1}):t_i\in\RR\} $$
of $\frakg^{-\theta}$. Then
\begin{align*}
 M_G &= Z_K(\fraka_G) = \{\diag(z_1,\ldots,z_n,z_{n+1}):|z_i|=1\} \simeq \upU(1)^{n+1},\\
 M_H &= Z_{H\cap K}(\fraka_H) = \{\diag(z_1,\ldots,z_n,1):|z_i|=1\} \simeq \upU(1)^n.
\end{align*}
Further,
$$ \frakn^{-\sigma}=\left\{\left(\begin{array}{cc}\0_n&z\\0&0\end{array}\right):z\in\CC^n\right\} $$
and such an element of $\frakn^{-\sigma}$ is contained in an open $(M_G\cap M_H)A_H$-orbit if and only if $z_1,\ldots,z_n\neq0$. Hence $M=\{\1\}$ is trivial.

\subsubsection{$(G,H)=(\GL(n+1,\RR),\GL(n,\RR))$}

Let $G=\GL(n+1,\RR)$, $n\geq2$, and define a Cartan involution on $G$ by $\theta(g)=(g^\top)^{-1}$; then $K=G^\theta=\upO(n+1)$. We embed $H=\GL(n,\RR)$ in the upper left corner of $H$; then $\theta$ leaves $H$ invariant and $H\cap K=H^\theta=\upO(n)$. We choose the maximal abelian subalgebra
$$ \fraka_H = \{\diag(t_1,\ldots,t_n,0):t_i\in\RR\} $$
of $\frakh^{-\theta}$ and extend it to the maximal abelian subalgebra
$$ \fraka_G = \{\diag(t_1,\ldots,t_n,t_{n+1}):t_i\in\RR\} $$
of $\frakg^{-\theta}$. Then
\begin{align*}
 M_G &= Z_K(\fraka_G) = \{\diag(x_1,\ldots,x_n,x_{n+1}):x_i=\pm1\} \simeq \upO(1)^{n+1},\\
 M_H &= Z_{H\cap K}(\fraka_H) = \{\diag(x_1,\ldots,x_n,1):x_i=\pm1\} \simeq \upO(1)^n.
\end{align*}
Further,
$$ \frakn^{-\sigma}=\left\{\left(\begin{array}{cc}\0_n&x\\0&0\end{array}\right):x\in\RR^n\right\} $$
and such an element of $\frakn^{-\sigma}$ is contained in an open $(M_G\cap M_H)A_H$-orbit if and only if $x_1,\ldots,x_n\neq0$. Hence $M=\{\1\}$ is trivial.

\subsubsection{$(G,H)=(\upU(p,q+1),\upU(p,q))$}\label{sec:IwasawaUpq}

Let $G=\upU(p,q+1)$, $p,q\geq1$, and define a Cartan involution on $G$ by $\theta(g)=(g^*)^{-1}$, then $K=G^\theta=\upU(p)\times\upU(q+1)$. We embed $H=\upU(p,q)$ in the upper-left corner of $G$; then $\theta$ leaves $H$ invariant and $H\cap K=H^\theta=\upU(p)\times\upU(q)$. Assume that $p\geq q+1$; the case $p\leq q$ is handled similarly (see Section~\ref{sec:IwasawaOpq} for a related computation). We choose the maximal abelian subalgebra
$$ \fraka_G = \left\{\left(\begin{array}{ccc}\0_{q+1}&&D\\&\0_{p-q-1}&\\D&&\0_{q+1}\end{array}\right):D=\diag(t_1,\ldots,t_{q+1}),t_1,\ldots,t_{q+1}\in\RR\right\} $$
of $\frakg^{-\theta}$; then $\fraka_H=\fraka_G\cap\frakh$ is maximal abelian in $\frakh^{-\theta}$. Write $e_i\in\fraka_G^\vee$ for the functional mapping a matrix of the above form to $t_i$. Then
$$ \Sigma(\frakg,\fraka_G) = \begin{cases}\{\pm e_i\pm e_j:1\leq i<j\leq q+1\}\cup\{\pm2e_i:1\leq i\leq q+1\}&\mbox{for $p=q+1$,}\\\{\pm e_i\pm e_j:1\leq i<j\leq q+1\}\cup\{\pm e_i,\pm2e_i:1\leq i\leq q+1\}&\mbox{for $p>q+1$,}\end{cases} $$
with root spaces given by
\begin{align*}
 \frakg(\fraka_G;\pm(e_i+e_j)) ={}& \Big\{z(E_{i,j}\mp E_{i,p+j}\pm E_{p+i,j}-E_{p+i,p+j})\\
 & \hspace{4cm} -\overline{z}(E_{j,i}\mp E_{j,p+i}\pm E_{p+j,i}-E_{p+j,p+i}):z\in\CC\Big\},\\
 \frakg(\fraka_G;\pm(e_i-e_j)) ={}& \Big\{z(E_{i,j}\pm E_{i,p+j}\pm E_{p+i,j}+E_{p+i,p+j})\\
 & \hspace{4cm} -\overline{z}(E_{j,i}\mp E_{j,p+i}\mp E_{p+j,i}+E_{p+j,p+i}):z\in\CC\Big\},\\
 \frakg(\fraka_G;\pm2e_i) ={}& \RR\sqrt{-1}(E_{i,i}\mp E_{i,p+i}\pm E_{p+i,i}-E_{p+i,p+i}),
\intertext{and if $p>q+1$, additionally,}
 \frakg(\fraka_G;\pm e_i) ={}& \Big\{\sum_{k=1}^{p-q-1}\Big(z_k(E_{i,q+k+1}\pm E_{p+i,q+k+1})-\overline{z_k}(E_{q+k+1,i}\mp E_{q+k+1,p+i})\Big):\\
 & \hspace{10.5cm}z_k\in\CC\Big\}.
\end{align*}
We choose the positive system
$$ \Sigma^+(\frakg,\fraka_G) = \begin{cases}\{e_i\pm e_j:1\leq i<j\leq q+1\}\cup\{2e_i:1\leq i\leq q+1\}&\mbox{for $p=q+1$,}\\\{e_i\pm e_j:1\leq i<j\leq q+1\}\cup\{e_i,2e_i:1\leq i\leq q+1\}&\mbox{for $p>q+1$,}\end{cases} $$
then
\begin{align*}
 M_G &= \{\diag(z_1,\ldots,z_{q+1},k,z_1,\ldots,z_{q+1}):|z_i|=1,k\in\upU(p-q-1)\}\\
 &\simeq \upU(1)^{q+1}\times\upU(p-q-1),\\
 M_H &= \{\diag(z_1,\ldots,z_q,k,z_1,\ldots,z_q,1):|z_i|=1,k\in\upU(p-q)\}\\
 &\simeq \upU(1)^q\times\upU(p-q)
\end{align*}
and
$$ \frakn^{-\sigma} = \Big\{X(z)=\sum_{i=1}^q\Big(z_i(E_{i,p+q+1}+E_{p+i,p+q+1})+\overline{z_i}(E_{p+q+1,i}-E_{p+q+1,p+i})\Big):z_i\in\CC\Big\}. $$
The action of $M_G\cap M_H\simeq\upU(1)^q\times\upU(p-q-1)$ is given by the natural action of $\upU(1)^q$ on $z\in\CC^q$, and the second factor $\upU(p-q-1)$ acts trivially. Hence, the element $X(t)$ is contained in an open $(M_G\cap M_H)A_H$-orbit if and only if $t_1,\ldots,t_q\neq0$ and then $M=\upU(p-q-1)$. Further, $\upU(1)^q\subseteq M_G\cap M_H$ acts transitively on the product $(S^1)^q\subseteq\CC^q$ of the unit spheres $S^1\subseteq\CC$ and therefore there is a unique open orbit.

\subsubsection{$(G,H)=(\SO(n+1,\CC),\SO(n,\CC))$}

Let $G=\SO(n+1,\CC)$, $n\geq2$, and define a Cartan involution on $G$ by $\theta(g)=(g^*)^{-1}=\overline{g}$; then $K=G^\theta=\SO(n+1)$. We embed $H=\SO(n,\CC)$ in the upper-left corner of $G$; then $\theta$ leaves $H$ invariant and $H\cap K=H^\theta=\SO(n)$. Write $n=2m$ if $n$ is even and $n=2m-1$ if it is odd. We choose the maximal abelian subalgebra
$$ \fraka_G = \begin{cases}\left\{\diag(D(\sqrt{-1}t_1),\ldots,D(\sqrt{-1}t_m)):t_1,\ldots,t_m\in\RR\right\}&\mbox{for $n$ odd,}\\\left\{\diag(D(\sqrt{-1}t_1),\ldots,D(\sqrt{-1}t_m),0):t_1,\ldots,t_m\in\RR\right\}&\mbox{for $n$ even,}\end{cases} $$
of $\frakg^{-\theta}$, where
$$ D(t) = \left(\begin{array}{cc}0&t\\-t&0\end{array}\right). $$
Then $\fraka_H=\fraka_G\cap\frakh$ is maximal abelian in $\frakh^{-\theta}$. Write $e_i\in\fraka_G^\vee$ for the functional mapping a matrix of the above form to $t_i$. Then
$$ \Sigma(\frakg,\fraka_G) = \begin{cases}\{\pm e_i\pm e_j:1\leq i<j\leq m\}&\mbox{for $n$ odd,}\\\{\pm e_i\pm e_j:1\leq i<j\leq m\}\cup\{\pm e_i:1\leq i\leq m\}&\mbox{for $n$ even,}\end{cases} $$
with root spaces given by
\begin{align*}
 \frakg(\fraka_G;\pm(e_i+e_j)) ={}& \CC(E_{2i-1,2j-1}\pm\sqrt{-1}E_{2i-1,2j}\mp\sqrt{-1}E_{2i,2j-1}+E_{2i,2j}\\
 & \ \ -E_{2j-1,2i-1}\mp\sqrt{-1}E_{2j,2i-1}\pm\sqrt{-1}E_{2j-1,2i}-E_{2j,2i}) && (1\leq i<j\leq m),\\
 \frakg(\fraka_G;\pm(e_i-e_j)) ={}& \CC(E_{2i-1,2j-1}\mp\sqrt{-1}E_{2i-1,2j}\mp\sqrt{-1}E_{2i,2j-1}-E_{2i,2j}\\
 & \ \ -E_{2j-1,2i-1}\pm\sqrt{-1}E_{2j,2i-1}\pm\sqrt{-1}E_{2j-1,2i}+E_{2j,2i}) && (1\leq i<j\leq m),
\intertext{and if $n$ is even, additionally,}
 \frakg(\fraka_G;\pm e_i) ={}& \CC(E_{2i-1,n+1}\mp\sqrt{-1}E_{2i,n+1}-E_{n+1,2i-1}\pm\sqrt{-1}E_{n+1,2i}) && (1\leq i\leq m).
\end{align*}
We choose the positive system
$$ \Sigma^+(\frakg,\fraka_G) = \begin{cases}\{e_i\pm e_j:1\leq i<j\leq m\}&\mbox{for $n$ odd,}\\\{e_i\pm e_j:1\leq i<j\leq m\}\cup\{e_i:1\leq i\leq m\}&\mbox{for $n$ even;}\end{cases} $$
then
\begin{align*}
 M_G &= \begin{cases}\{\exp(\diag(D(t_1),\ldots,D(t_m))):t_i\in\RR\}\simeq\SO(2)^m&\mbox{for $n$ odd,}\\\{\exp(\diag(D(t_1),\ldots,D(t_m),0)):t_i\in\RR\}\simeq\SO(2)^m&\mbox{for $n$ even,}\end{cases}\\
 M_H &= \begin{cases}\{\exp(\diag(D(t_1),\ldots,D(t_{m-1}),0)):t_i\in\RR\}\simeq\SO(2)^{m-1}&\mbox{for $n$ odd,}\\\{\exp(\diag(D(t_1),\ldots,D(t_m))):t_i\in\RR\}\simeq\SO(2)^m&\mbox{for $n$ even,}\end{cases}
\end{align*}
and
$$ \frakn^{-\sigma} = \begin{cases}\{X(t)=\sum_{i=1}^{m-1}t_i(E_{2i-1,n+1}-\sqrt{-1}E_{2i,n+1}-E_{n+1,2i-1}+\sqrt{-1}E_{n+1,2i}):t_i\in\CC\}\\\hspace{11.7cm}\mbox{for $n$ odd,}\\\{X(t)=\sum_{i=1}^mt_i(E_{2i-1,n+1}-\sqrt{-1}E_{2i,n+1}-E_{n+1,2i-1}+\sqrt{-1}E_{n+1,2i}):t_i\in\CC\}\\\hspace{11.7cm}\mbox{for $n$ even.}\end{cases} $$
The action of
$$ M_G\cap M_H \simeq \begin{cases}\SO(2)^{m-1}&\mbox{for $n$ odd,}\\\SO(2)^m&\mbox{for $n$ even,}\end{cases} $$
is given by the natural action of $\SO(2)\simeq\upU(1)$ on $t_i\in\CC$ by rotation. Hence, the element $X(t)$ is contained in an open $(M_G\cap M_H)A_H$-orbit if and only if $t_1,\ldots,t_{m-1}\neq0$ resp. $t_1,\ldots,t_m\neq0$ and then $M=\{\1\}$. Further, $\SO(2)$ acts transitively on the unit sphere in $\CC$ and hence $M_G\cap M_H$ acts transitively on the product of the unit spheres in $\frakn^{-\sigma}$, so there is a unique open orbit.

\subsubsection{$(G,H)=(\SO(p,q+1),\SO(p,q))$}\label{sec:IwasawaOpq}

Let $G=\SO(p,q+1)$, $p,q\geq1$, and define a Cartan involution on $G$ by $\theta(g)=(g^\top)^{-1}$; then $K=G^\theta=\upS(\upO(p)\times\upO(q+1))$. We embed $H=\SO(p,q)$ in the upper-left corner of $G$; then $\theta$ leaves $H$ invariant and $H\cap K=H^\theta=\upS(\upO(p)\times\upO(q))$. Assume that $p\leq q$; the case $p\geq q+1$ is handled similarly (see Section~\ref{sec:IwasawaUpq} for a related computation). We choose the maximal abelian subalgebra
$$ \fraka_G = \fraka_H = \left\{\left(\begin{array}{ccc}\0_p&D&\\D&\0_p&\\&&\0_{q-p+1}\end{array}\right):D=\diag(t_1,\ldots,t_p),t_1,\ldots,t_p\in\RR\right\} $$
of $\frakg^{-\theta}$ and $\frakh^{-\theta}$. Write $e_i\in\fraka_G^\vee$ for the functional mapping a matrix of the above form to $t_i$. Then
$$ \Sigma(\frakg,\fraka_G) = \{\pm e_i\pm e_j:1\leq i<j\leq p\}\cup\{\pm e_i:1\leq i\leq p\} $$
with root spaces given by
\begin{align*}
 \frakg(\fraka_G;\pm(e_i+e_j)) ={}& \RR(E_{i,j}-E_{j,i}\mp E_{i,j+p}\pm E_{j,i+p}\\
 & \ \ \ \pm E_{i+p,j}\mp E_{j+p,i}\mp E_{i+p,j+p}\pm E_{j+p,i+p}) && (1\leq i<j\leq p),\\
 \frakg(\fraka_G;\pm(e_i-e_j)) ={}& \RR(E_{i,j}-E_{j,i}\pm E_{i,j+p}\pm E_{j,i+p}\\
 & \ \ \ \pm E_{i+p,j}\pm E_{j+p,i}\pm E_{i+p,j+p}\mp E_{j+p,i+p}) && (1\leq i<j\leq p),\\
 \frakg(\fraka_G;\pm e_i) ={}& \bigoplus_{k=1}^{q-p+1} \RR(E_{i,2p+k}\pm E_{i+p,2p+k}+E_{2p+k,i}\mp E_{2p+k,i+p}) && (1\leq i\leq p).
\end{align*}
We choose the positive system
$$ \Sigma^+(\frakg,\fraka_G)=\{e_i\pm e_j:1\leq i<j\leq p\}\cup\{e_i:1\leq i\leq p\}; $$
then
\begin{align*}
 M_G &= \{\diag(x_1,\ldots,x_p,x_1,\ldots,x_p,k):x_i=\pm1,k\in\SO(q-p+1)\} \simeq \upO(1)^p\times\SO(q-p+1),\\
 M_H &= \{\diag(x_1,\ldots,x_p,x_1,\ldots,x_p,k,1):x_i=\pm1,k\in\SO(q-p)\} \simeq \upO(1)^p\times\SO(q-p),
\end{align*}
and
$$ \frakn^{-\sigma} = \Big\{X(t)=\sum_{i=1}^pt_i(E_{i,p+q+1}+E_{i+p,p+q+1}+E_{p+q+1,i}-E_{p+q+1,i+p}):t_i\in\RR\Big\}. $$
The action of $M_G\cap M_H\simeq\upO(1)^p\times\SO(q-p)$ is given by the natural action of $\upO(1)^p$ on $t\in\RR^p$ by sign changes of the coordinates, and the second factor $\SO(q-p)$ acts trivially. Hence, the element $X(t)$ is contained in an open $(M_G\cap M_H)A_H$-orbit if and only if $t_1,\ldots,t_p\neq0$ and then $M=\SO(q-p)$. Further, $\upO(1)^p\subseteq M_G\cap M_H$ acts transitively on the product $\{\pm1\}^p\subseteq\RR^p$ of the unit spheres $\{\pm1\}\subseteq\RR$ and therefore there is a unique open orbit.

\subsection{Finite-dimensional branching}\label{sec:FinDimBranching}

Using case-by-case arguments, we prove in this subsection the following statement:

\begin{proposition}\label{prop:ExGRepsForMultOnePairs}
Let $(G,H)$ be one of the pairs in \eqref{eq:MultOnePairs} and $(\xi,\eta)\in\widehat{M}_G\times\widehat{M}_H$. Then
\begin{enumerate}
\item\label{prop:ExGRepsForMultOnePairs1} $\dim\Hom_M(\xi^{\tilde{w}_0}|_M,\eta|_M)=\dim\Hom_M(\xi|_M,\eta|_M)\leq1$,
\item\label{prop:ExGRepsForMultOnePairs2} whenever $\Hom_M(\xi|_M,\eta|_M)\neq\{0\}$, there exist finite-dimensional representations $E$ of $G$ and $F$ of $H$ such that $\xi\simeq E^{\frakn_G}|_{M_G}$, $\eta\simeq F^{\frakn_H}|_{M_H}$ and $\Hom_H(E|_H,F)\neq\{0\}$.
\end{enumerate}
\end{proposition}

In all five cases, $G$ and $H$ have connected complexifications $G_\CC$ and $H_\CC$, and we can parametrize irreducible finite-dimensional representations of $G_\CC$ and $H_\CC$ by their highest weights $\lambda\in\frakj_{G,\CC}^\vee$ and $\nu\in\frakj_{H,\CC}^\vee$ for $\frakj_G\subseteq\frakg$ and $\frakj_H\subseteq\frakh$ Cartan subalgebras. Denote the restrictions of the corresponding representations to $G$ and $H$ by $F^G(\lambda)$ and $F^H(\nu)$.

In each case, we use the notation introduced in Section~\ref{sec:Iwasawa}.

\subsubsection{$(G,H)=(\GL(n+1,\CC),\GL(n,\CC))$}

We extend $\fraka_G$ to the Cartan subalgebra $\frakj_G=\frakt_G\oplus\fraka_G$ with
$$ \frakt_G = \frakm_G = \sqrt{-1}\fraka_G = \{\sqrt{-1}\diag(t_1,\ldots,t_{n+1}):t_i\in\RR\}. $$
Let $f_i\in\frakt_{G,\CC}^\vee$ ($1\leq i\leq n+1$) be the linear functionals mapping a diagonal matrix as above to $\sqrt{-1}t_i$. Then the root system $\Sigma(\frakg_\CC,\frakj_{G,\CC})$ is of the form
\begin{equation*}
 \{\pm(e_i-e_j)\pm(f_i-f_j):1\leq i<j\leq n+1\}
\end{equation*}
and putting
$$ \varepsilon_i' = e_i+f_i \qquad \mbox{and} \qquad \varepsilon_i'' = e_i-f_i $$
we have $\Sigma(\frakg_\CC,\frakj_{G,\CC})=\{\pm(\varepsilon_i'-\varepsilon_j'),\pm(\varepsilon_i''-\varepsilon_j''):1\leq i<j\leq n+1\}$. Further, the positive system $\Sigma^+(\frakg_\CC,\frakj_{G,\CC})=\{\varepsilon_i'-\varepsilon_j',\varepsilon_i''-\varepsilon_j'':1\leq i<j\leq n+1\}$ is compatible with the positive system $\Sigma^+(\frakg,\fraka_G)$.

We choose the complexification $G_\CC=\GL(n+1,\CC)\times\GL(n+1,\CC)$ with embedding $G\hookrightarrow G_\CC$ given by $g\mapsto(g,\overline{g})$; then $\frakj_{G,\CC}$ is a Cartan subalgebra of $\frakg_\CC=\gl(n+1,\CC)+\gl(n+1,\CC)$ and the irreducible finite-dimensional representations of $G_\CC$ are parametrized by their highest weights $\lambda=\lambda_1'\varepsilon_1'+\cdots+\lambda_{n+1}'\varepsilon_{n+1}'+\lambda_1''\varepsilon_1''+\cdots+\lambda_{n+1}''\varepsilon_{n+1}''$, $\lambda_1'\geq\ldots\geq\lambda_{n+1}'$, $\lambda_1''\geq\ldots\geq\lambda_{n+1}''$, $\lambda_i',\lambda_i''\in\ZZ$. Clearly, an element $g=\exp(\sqrt{-1}\diag(t_1,\ldots,t_{n+1}))\in M_G$ acts on the highest weight space of $F^G(\lambda)$ by $e^{\sqrt{-1}(\lambda_1'-\lambda_1'')t_1}\cdots e^{\sqrt{-1}(\lambda_m'-\lambda_m'')t_m}$.

In the same way we parametrize irreducible finite-dimensional representations of $H_\CC=\GL(n,\CC)\times\GL(n,\CC)$ by highest weight $\nu=\nu_1'\varepsilon_1'+\cdots+\nu_n'\varepsilon_n'+\nu_1''\varepsilon_1''+\cdots+\nu_n''\varepsilon_n''$, $\nu_1'\geq\ldots\geq\nu_n'$, $\nu_1''\geq\ldots\geq\nu_n''$, $\nu_i',\nu_i''\in\ZZ$. On its highest weight space an element of the form $h=\exp(\sqrt{-1}\diag(t_1,\ldots,t_n,0))\in M_H$ acts by $e^{\sqrt{-1}(\nu_1'-\nu_1'')t_1}\cdots e^{\sqrt{-1}(\nu_n'-\nu_n'')t_n}$.

Now, $M_G=\upU(1)^{n+1}$, $M_H=\upU(1)^n$ and $M=\{\1\}$. Hence, irreducible representations of $M_G$ and $M_H$ are one-dimensional and $\dim\Hom_M(\xi^{\tilde{w}_0}|_M,\eta|_M)=\dim\Hom_M(\xi_M,\eta|_M)=1$ for all $(\xi,\eta)\in\widehat{M}_G\times\widehat{M}_H$. Every $\xi\in\widehat{M}_G$ has the form $\xi(\exp(\sqrt{-1}\diag(t_1,\ldots,t_{n+1})))=e^{\sqrt{-1}\xi_1t_1}\cdots e^{\sqrt{-1}\xi_{n+1}t_{n+1}}$ with $\xi_1,\ldots,\xi_{n+1}\in\ZZ$. Similarly, every $\eta\in\widehat{M}_H$ has the form $\eta(\exp(\sqrt{-1}\diag(t_1,\ldots,t_n,0)))=e^{\sqrt{-1}\eta_1t_1}\cdots e^{\sqrt{-1}\eta_nt_n}$ with $\eta_1,\ldots,\eta_n\in\ZZ$. By the above observations $\xi\simeq F^G(\lambda)^{\frakn_G}|_{M_G}$ if and only if $\lambda_i'-\lambda_i''=\xi_i$ ($1\leq i\leq n+1$). Further, $\eta\simeq F^H(\nu)^{\frakn_H}|_{M_H}$ if and only if $\nu_i'-\nu_i''=\eta_i$ ($1\leq i\leq n$). Moreover, we have $\Hom_H(F^G(\lambda),F^H(\nu))\neq\{0\}$ if and only if
$$ \lambda_1'\geq\nu_1'\geq\lambda_2'\geq\ldots\geq\lambda_n'\geq\nu_n'\geq\lambda_{n+1}' \qquad \mbox{and} \qquad \lambda_1''\geq\nu_1''\geq\lambda_2''\geq\ldots\geq\lambda_n''\geq\nu_n''\geq\lambda_{n+1}''. $$

We first choose $(\lambda_{n+1}',\lambda_{n+1}'')\in\ZZ\times\ZZ$ with $\lambda_{n+1}'-\lambda_{n+1}''=\xi_{n+1}$. Next we choose $(\nu_n',\nu_n'')\in\NN\times\NN$ with $\nu_n'\geq\lambda_{n+1}'$, $\nu_n''\geq\lambda_{n+1}''$ and $\nu_n'-\nu_n''=\eta_n$. Iterating this procedure constructs (not necessarily unique) highest weights $\lambda$ and $\nu$ with the desired properties.

\subsubsection{$(G,H)=(\GL(n+1,\RR),\GL(n,\RR))$}

We choose the natural complexification $G_\CC=\GL(n+1,\CC)$; then $\fraka_{G,\CC}$ is a maximal torus in $\frakg_\CC=\gl(n+1,\CC)$ and the irreducible finite-dimensional representations of $G_\CC$ are parametrized by their highest weights $\lambda=\lambda_1e_1+\cdots+\lambda_{n+1}e_{n+1}$, $\lambda_1\geq\ldots\geq\lambda_{n+1}$, $\lambda_i\in\ZZ$. For $H$ we use similar notation: $\nu=\nu_1e_1+\cdots+\nu_ne_n$, $\nu_1\geq\ldots,\nu_n$, $\nu_i\in\ZZ$.

Note that an element $g=\diag((-1)^{k_1},\ldots,(-1)^{k_{n+1}})\in M_G$ acts on the highest weight space $F^G(\lambda)^{\frakn_G}$ by $(-1)^{k_1\lambda_1+\cdots+k_{n+1}\lambda_{n+1}}$. Similarly, $h=\diag((-1)^{\ell_1},\ldots,(-1)^{\ell_n},1)\in M_H$ acts on $F^H(\nu)^{\frakn_H}$ by $(-1)^{\ell_1\nu_1+\cdots+\ell_n\nu_n}$. By the classical branching laws, the restriction of $F^G(\lambda)$ to $H$ contains all representations $F^H(\nu)$ with
\begin{equation}
 \lambda_1\geq\nu_1\geq\lambda_2\geq\ldots\geq\lambda_n\geq\nu_n\geq\lambda_{n+1}.\label{eq:InterlacingGLn}
\end{equation}

Now, $M_G=\upO(1)^{n+1}$, $M_H=\upO(1)^n$ and $M=\{\1\}$. Hence, irreducible representations of $M_G$ and $M_H$ are one-dimensional and $\dim\Hom_M(\xi^{\tilde{w}_0}|_M,\eta|_M)=\dim\Hom_M(\xi|_M,\eta|_M)=1$ for all $(\xi,\eta)\in\widehat{M}_G\times\widehat{M}_H$. Every $\xi\in\widehat{M}_G$ has the form $\xi(\diag((-1)^{k_1},\ldots,(-1)^{k_{n+1}}))=(-1)^{k_1\xi_1+\cdots+k_{n+1}\xi_{n+1}}$ with $\xi_i\in\ZZ/2\ZZ$. Similarly, every representation $\eta\in\widehat{M}_H$ has the form $\eta(\diag((-1)^{\ell_1},\ldots,(-1)^{\ell_n},1))=(-1)^{\ell_1\eta_1+\cdots+\ell_n\eta_n}$ with $\eta_i\in\ZZ/2\ZZ$. By the above observations, $\xi\simeq F^G(\lambda)^{\frakn_G}|_{M_G}$ if and only if $\xi_i=\lambda_i+2\ZZ$. Further, $\eta\simeq F^H(\nu)^{\frakn_H}|_{M_H}$ if and only if $\eta_i=\nu_i+2\ZZ$. It is clear that for fixed $\xi_1,\ldots,\xi_{n+1},\eta_1,\ldots,\eta_n\in\ZZ/2\ZZ$ there always exist integers $\lambda_1,\ldots,\lambda_{n+1}$ and $\nu_1,\ldots,\nu_n$ satisfying the interlacing condition \eqref{eq:InterlacingGLn} and $\xi_i=\lambda_i+2\ZZ$, $\eta_i=\nu_i+2\ZZ$.

\subsubsection{$(G,H)=(\upU(p,q+1),\upU(p,q))$}

Assume that $p\geq q+1$; the case $p\leq q$ is handled similarly. We extend $\fraka_G$ to the Cartan subalgebra $\frakj_G=\frakt_G\oplus\fraka_G$ with
$$ \frakt_G = \{\sqrt{-1}\diag(t_1,\ldots,t_{q+1},s_1,\ldots,s_{p-q-1},t_1,\ldots,t_{q+1}):s_i,t_i\in\RR\} \subseteq \frakm_G. $$
Let $f_i\in\frakt_{G,\CC}^\vee$ ($1\leq i\leq q+1$) be the linear functionals mapping a diagonal matrix as above to $\sqrt{-1}t_i$, and $g_i\in\frakt_{G,\CC}^\vee$ ($1\leq i\leq p-q-1$) the ones mapping to $\sqrt{-1}s_i$. Then the root system $\Sigma(\frakg_\CC,\frakj_{G,\CC})$ is of the form
\begin{multline*}
 \{\pm e_i\pm e_j\pm(f_i-f_j):1\leq i<j\leq q+1\}\cup\{\pm2e_i:1\leq i\leq q+1\}\\
 \cup\{\pm e_i\pm(f_i-g_j):1\leq i\leq q+1,1\leq j\leq p-q-1\}\cup\{\pm(g_i-g_j):1\leq i<j\leq p-q-1\}
\end{multline*}
and putting
$$ (\varepsilon_1,\ldots,\varepsilon_{p+q+1})=(f_1+e_1,\ldots,f_{q+1}+e_{q+1},g_1,\ldots,g_{p-q-1},f_{q+1}-e_{q+1},\ldots,f_1-e_1) $$
we have $\Sigma(\frakg_\CC,\frakj_{G,\CC})=\{\pm(\varepsilon_i-\varepsilon_j):1\leq i<j\leq p+q+1\}$. Further, the positive system $\Sigma^+(\frakg_\CC,\frakj_{G,\CC})=\{\varepsilon_i-\varepsilon_j:1\leq i<j\leq p+q+1\}$ is compatible with the positive system $\Sigma^+(\frakg,\fraka_G)$.

We choose the complexification $G_\CC=\GL(p+q+1,\CC)$, then $\frakj_{G,\CC}$ is a Cartan subalgebra of $\frakg_\CC=\gl(p+q+1,\CC)$ and the irreducible finite-dimensional representations of $G_\CC$ are parametrized by their highest weights $\lambda=\lambda_1\varepsilon_1+\cdots+\lambda_{p+q+1}\varepsilon_{p+q+1}$, $\lambda_1\geq\ldots\geq\lambda_{p+q+1}$, $\lambda_i\in\ZZ$.

An element $g=\diag(z_1,\ldots,z_{q+1},\1_{p-q-1},z_1,\ldots,z_{q+1})\in M_G$ acts on the highest restricted weight space of $F^G(\lambda)$ by $z_1^{\lambda_1+\lambda_{p+q+1}}\cdots z_{q+1}^{\lambda_{q+1}+\lambda_{p+1}}$. Further, $\upU(p-q-1)\subseteq M_G$ has roots $\pm(\varepsilon_i-\varepsilon_j)$ ($q+2\leq i<j\leq p$) and therefore its action on the highest restricted weight space is given by $F^{\upU(p-q-1)}(\lambda_{q+2}\varepsilon_{q+2}+\cdots+\lambda_p\varepsilon_p)$.

In the same way we parametrize irreducible finite-dimensional representations of $H_\CC=\GL(p+q,\CC)$ by their highest weights $\nu=\nu_1\varepsilon_1+\cdots+\nu_{p+q}\varepsilon_{p+q}$. An element of the form $h=\diag(z_1,\ldots,z_q,\1_{p-q},z_1,\ldots,z_q)\in M_H$ acts on the highest restricted weight space by $z_1^{\nu_1+\nu_{p+q}}\cdots z_q^{\nu_q+\nu_{p+1}}$ and $\upU(p-q)\subseteq M_H$ acts by $F^{\upU(p-q)}(\nu_{q+1}\varepsilon_{q+1}+\cdots+\nu_p\varepsilon_p)$.

Now, $M_G=\upU(1)^{q+1}\times\upU(p-q-1)$, $M_H=\upU(1)^q\times\upU(p-q)$ and $M=\upU(p-q-1)$. An irreducible representation $\xi\in\widehat{M}_G$ is of the form $\xi=\xi'\boxtimes\xi''$ with
$$ \xi'(\diag(z_1,\ldots,z_{q+1},\1_{p-q-1},z_1,\ldots,z_{q+1})) = z_1^{\xi_1'}\cdots z_{q+1}^{\xi_{q+1}'}, $$
$\xi_1',\ldots,\xi_{q+1}'\in\ZZ$, and $\xi''=F^{\upU(p-q-1)}(\xi_1''\varepsilon_{q+2}+\cdots+\xi_{p-q-1}''\varepsilon_p)$, $\xi_1''\geq\ldots\geq\xi_{p-q-1}''$. Similarly, every $\eta\in\widehat{M}_H$ has the form $\eta=\eta'\boxtimes\eta''$ with
$$ \eta'(\diag(z_1,\ldots,z_q,\1_{p-q},z_1,\ldots,z_q)) = z_1^{\eta_1'}\cdots z_q^{\eta_q''}, $$
$\eta_1',\ldots,\eta_q'\in\ZZ$, and $\eta''=F^{\upU(p-q)}(\eta_1''\varepsilon_{q+1}+\cdots+\eta_{p-q}''\varepsilon_p)$, $\eta_1''\geq\ldots\geq\eta_{p-q}''$.

This implies that we have to put
$$ (\lambda_{q+2},\ldots,\lambda_p) = (\xi_1'',\ldots,\xi_{p-q-1}'') \qquad \mbox{and} \qquad (\nu_{q+1},\ldots,\nu_p) = (\eta_1'',\ldots,\eta_{p-q}''). $$
Since $\tilde{w}_0=\diag(\1_p,-\1_{q+1})$ commutes with $M_G$ we have $\xi^{\tilde{w}_0}=\xi$, so that $\Hom_M(\xi^{\tilde{w}_0}|_M,\eta|_M)=\Hom_M(\xi|_M,\eta|_M)$, and the condition $\Hom_M(\xi|_M,\eta|_M)\neq\{0\}$ is equivalent to the condition $\Hom_{\upU(p-q-1)}(\xi'',\eta''|_{\upU(p-q-1)})\neq\{0\}$ which is in turn equivalent to
$$ \eta_1''\geq\xi_1''\geq\eta_2''\geq\ldots\geq\xi_{p-q-1}''\geq\eta_{p-q}''. $$
Therefore, the already chosen $\lambda_i$'s and $\nu_j$'s satisfy the necessary interlacing condition for $\Hom_H(F^G(\lambda),F^H(\nu))\neq\{0\}$. It remains to show that one can choose the remaining $\lambda_i$'s and $\nu_j$'s such that the interlacing condition still holds and additionally $\lambda_i+\lambda_{p+q-i+2}=\xi_i'$ and $\nu_j+\nu_{p+q-j+1}=\eta_j'$, which is an easy exercise.

\subsubsection{$(G,H)=(\SO(n+1,\CC),\SO(n,\CC))$}

Assume that $n=2m$ is even; the case of odd $n$ is treated similarly. We extend $\fraka_G$ to the Cartan subalgebra $\frakj_G=\frakt_G\oplus\fraka_G$ with
$$ \frakt_G = \frakm_G = \sqrt{-1}\fraka_G = \{\sqrt{-1}\diag(D(t_1),\ldots,D(t_m)):t_i\in\RR\}. $$
Let $f_i\in\frakt_{G,\CC}^\vee$ ($1\leq i\leq m$) be the linear functionals mapping a diagonal matrix as above to $\sqrt{-1}t_i$. Then the root system $\Sigma(\frakg_\CC,\frakj_{G,\CC})$ is of the form
\begin{multline*}
 \{\pm(e_i+f_i)\pm(e_j+f_j):1\leq i<j\leq m\}\cup\{\pm(e_i+f_i):1\leq i\leq m\}\\
 \cup\{\pm(e_i-f_i)\pm(e_j-f_j):1\leq i<j\leq m\}\cup\{\pm(e_i-f_i):1\leq i\leq m\}
\end{multline*}
and putting
$$ \varepsilon_i' = e_i+f_i \qquad \mbox{and} \qquad \varepsilon_i'' = e_i-f_i $$
we have $\Sigma(\frakg_\CC,\frakj_{G,\CC})=\{\pm\varepsilon_i'\pm\varepsilon_j',\pm\varepsilon_i''\pm\varepsilon_j'':1\leq i<j\leq m\}\cup\{\pm\varepsilon_i',\pm\varepsilon_i'':1\leq i\leq m\}$. Further, the positive system $\Sigma^+(\frakg_\CC,\frakj_{G,\CC})=\{\varepsilon_i'\pm\varepsilon_j',\varepsilon_i''\pm\varepsilon_j'':1\leq i<j\leq m\}\cup\{\varepsilon_i',\varepsilon_i'':1\leq i\leq m\}$ is compatible with the positive system $\Sigma^+(\frakg,\fraka_G)$.

We choose the complexification $G_\CC=\SO(n+1,\CC)\times\SO(n+1,\CC)$ with embedding $G\hookrightarrow G_\CC$ given by $g\mapsto(g,\overline{g})$; then $\frakj_{G,\CC}$ is a Cartan subalgebra of $\frakg_\CC=\so(n+1,\CC)+\so(n+1,\CC)$ and the irreducible finite-dimensional representations of $G_\CC$ are parametrized by their highest weights $\lambda=\lambda_1'\varepsilon_1'+\cdots+\lambda_m'\varepsilon_m'+\lambda_1''\varepsilon_1''+\cdots+\lambda_m''\varepsilon_m''$, $\lambda_1'\geq\ldots\geq\lambda_{m-1}'\geq|\lambda_m'|$, $\lambda_1''\geq\ldots\geq\lambda_{m-1}''\geq|\lambda_m''|$, $\lambda_i',\lambda_i''\in\ZZ$. An element $g=\exp(\sqrt{-1}\diag(D(t_1),\ldots,D(t_m)))\in M_G$ acts on the highest weight space of $F^G(\lambda)$ by $e^{\sqrt{-1}(\lambda_1'-\lambda_1'')t_1}\cdots e^{\sqrt{-1}(\lambda_m'-\lambda_m'')t_m}$.

In the same way, we parametrize irreducible finite-dimensional representations of $H_\CC=\SO(n,\CC)\times\SO(n,\CC)$ by their highest weights $\nu=\nu_1'\varepsilon_1'+\cdots+\nu_{m-1}'\varepsilon_{m-1}'+\nu_1''\varepsilon_1''+\cdots+\nu_{m-1}''\varepsilon_{m-1}''$, $\nu_1'\geq\ldots\geq\nu_{m-1}'$, $\nu_1''\geq\ldots\geq\nu_{m-1}''$, $\nu_i',\nu_i''\in\ZZ$. On its highest weight space, an element $h=\exp(\sqrt{-1}\diag(D(t_1),\ldots,D(t_{m-1}),0))\in M_H$ acts by $e^{\sqrt{-1}(\nu_1'-\nu_1'')t_1}\cdots e^{\sqrt{-1}(\nu_{m-1}'-\nu_{m-1}'')t_{m-1}}$.

Now, $M_G=\SO(2)^m$, $M_H=\SO(2)^{m-1}$ and $M=\{\1\}$. Hence, irreducible representations of $M_G$ and $M_H$ are one-dimensional and $\dim\Hom_M(\xi^{\tilde{w}_0}|_M,\eta|_M)=\dim\Hom_M(\xi|_M,\eta|_M)=1$ for all $(\xi,\eta)\in\widehat{M}_G\times\widehat{M}_H$. Every $\xi\in\widehat{M}_G$ has the form
$$ \xi(\exp(\sqrt{-1}\diag(D(t_1),\ldots,D(t_m)))) = e^{\sqrt{-1}\xi_1t_1}\cdots e^{\sqrt{-1}\xi_mt_m} $$
with $\xi_1,\ldots,\xi_m\in\ZZ$. Similarly, every $\eta\in\widehat{M}_H$ has the form
$$ \eta(\exp(\sqrt{-1}\diag(D(t_1),\ldots,D(t_{m-1}),0))) = e^{\sqrt{-1}\eta_1t_1}\cdots e^{\sqrt{-1}\eta_{m-1}t_{m-1}} $$
with $\eta_1,\ldots,\eta_{m-1}\in\ZZ$. By the above observations, $\xi\simeq F^G(\lambda)^{\frakn_G}|_{M_G}$ if and only if $\lambda_i'-\lambda_i''=\xi_i$ ($1\leq i\leq m$). Further, $\eta\simeq F^H(\nu)^{\frakn_H}|_{M_H}$ if and only if $\nu_i'-\nu_i''=\eta_i$ ($1\leq i\leq m-1$). Moreover, we have $\Hom_H(F^G(\lambda),F^H(\nu))\neq\{0\}$ if and only if
\begin{align*}
 & \lambda_1'\geq\nu_1'\geq\lambda_2'\geq\ldots\geq\lambda_{m-1}'\geq\nu_{m-1}'\geq|\lambda_m'| & \mbox{and}\\
 & \lambda_1''\geq\nu_1''\geq\lambda_2''\geq\ldots\geq\lambda_{m-1}''\geq\nu_{m-1}''\geq|\lambda_m''|.
\end{align*}

We first choose $(\lambda_m',\lambda_m'')\in\ZZ\times\ZZ$ with $\lambda_m'-\lambda_m''=\xi_m$. Next we choose $(\nu_{m-1}',\nu_{m-1}'')\in\NN\times\NN$ with $\nu_{m-1}'\geq|\lambda_m'|$, $\nu_{m-1}''\geq|\lambda_m''|$ and $\nu_{m-1}'-\nu_{m-1}''=\eta_{m-1}$. Iterating this procedure shows the claim.

\subsubsection{$(G,H)=(\SO(p,q+1),\SO(p,q))$}

Assume that $p\leq q$; the case $p\geq q+1$ is handled similarly. We further assume that $q-p=2m$ is even, leaving the odd case to the reader. We extend $\fraka_G$ to the Cartan subalgebra $\frakj_G=\frakt_G\oplus\fraka_G$ with
$$ \frakt_G = \{\diag(\0_{2p},D(t_1),\ldots,D(t_m),0):t_i\in\RR\} \subseteq \frakm_G. $$
Let $e_{p+i}\in\frakt_{G,\CC}^\vee$ ($1\leq i\leq m$) be the linear functional mapping a matrix as above to $\sqrt{-1}t_i$. Then the root system $\Sigma(\frakg_\CC,\frakj_{G,\CC})$ is of the form
\begin{equation*}
 \{\pm e_i\pm e_j:1\leq i<j\leq p+m\}\cup\{\pm e_i:1\leq i\leq p+m\}
\end{equation*}
and the positive system
$$ \Sigma^+(\frakg_\CC,\frakj_{G,\CC}) = \{e_i\pm e_j:1\leq i<j\leq p+m\}\cup\{e_i:1\leq i\leq p+m\} $$
is compatible with the positive system $\Sigma^+(\frakg,\fraka_G)$.

We choose the complexification $G_\CC=\SO(p+q+1,\CC)$, then $\frakj_{G,\CC}$ is a Cartan subalgebra of $\frakg_\CC=\so(p+q+1,\CC)$ and the irreducible finite-dimensional representations of $G_\CC$ are parametrized by their highest weights $\lambda=\lambda_1e_1+\cdots+\lambda_{p+m}e_{p+m}$, where $\lambda_1\geq\ldots\geq\lambda_{p+m}\geq0$, $\lambda_i\in\ZZ$. An element $g=\diag((-1)^{k_1},\ldots,(-1)^{k_p},(-1)^{k_1},\ldots,(-1)^{k_p},\1_{q-p-1})\in M_G$ acts on the highest restricted weight space of $F^G(\lambda)$ by $(-1)^{k_1\lambda_1+\cdots+k_p\lambda_p}$. Further, $\SO(q-p+1)\subseteq M_G$ acts on the highest restricted weight space by $F^{\SO(q-p+1)}(\lambda_{p+1}e_{p+1}+\cdots+\lambda_{p+m}e_{p+m})$.

In the same way we parametrize irreducible finite-dimensional representations of $H_\CC=\SO(p+q,\CC)$ by their highest weights $\nu=\nu_1e_1+\cdots+\nu_{p+m}e_{p+m}$. An element $h=\diag(x_1,\ldots,x_p,x_1,\ldots,x_p,\1_{q-p+1})\in M_H$ acts by $x_1^{\nu_1}\cdots x_p^{\nu_p}$ and $\SO(q-p)\subseteq M_H$ acts by $F^{\SO(q-p)}(\nu_{p+1}e_{p+1}+\cdots+\nu_{p+m}e_{p+m})$.

Now, $M_G=\upO(1)^p\times\SO(q-p+1)$, $M_H=\upO(1)^p\times\SO(q-p)$ and $M=\SO(q-p)$. An irreducible representation $\xi\in\widehat{M}_G$ is of the form $\xi=\xi'\boxtimes\xi''$ with
$$ \xi'(x_1,\ldots,x_p,x_1,\ldots,x_p,\1_{q-p+1})=x_1^{\xi_1'}\cdots x_p^{\xi_p'}, $$
$\xi_1',\ldots,\xi_p'\in\ZZ/2\ZZ$, and $\xi''=F^{\SO(q-p+1)}(\xi_1''e_{p+1}+\cdots+\xi_m''e_{p+m})$, $\xi_1''\geq\ldots\geq\xi_m''\geq0$. Similarly, every $\eta\in\widehat{M}_H$ has the form $\eta=\eta'\boxtimes\eta''$ with
$$ \eta'(\diag(x_1,\ldots,x_p,x_1,\ldots,x_p,\1_{q-p+1})) = x_1^{\eta_1'}\cdots x_p^{\eta_p'}, $$
$\eta_1',\ldots,\eta_p'\in\ZZ/2\ZZ$, and $\eta''=F^{\SO(q-p)}(\eta_1''e_{p+1}+\cdots+\eta_m''e_{p+m})$, $\eta_1''\geq\ldots\geq\eta_{m-1}''\geq|\eta_m''|$.

This implies that we have to put
$$ (\lambda_{p+1},\ldots,\lambda_{p+m}) = (\xi_1'',\ldots,\xi_m'') \qquad \mbox{and} \qquad (\nu_{p+1},\ldots,\nu_{p+m}) = (\eta_1'',\ldots,\eta_m''). $$
Since $\tilde{w}_0=\diag(-\1_p,\1_{q+1})$ commutes with $M_G$ we have $\xi^{\tilde{w}_0}=\xi$, so that $\Hom_M(\xi^{\tilde{w}_0}|_M,\eta|_M)=\Hom_M(\xi|_M,\eta|_M)$, and the condition $\Hom_M(\xi|_M,\eta|_M)\neq\{0\}$ is equivalent to the condition $\Hom_{\SO(q-p)}(\xi''|_{\SO(q-p)},\eta'')\neq\{0\}$ which is in turn equivalent to
$$ \xi_1''\geq\eta_1''\geq\xi_2''\geq\ldots\geq\eta_{m-1}''\geq\xi_m''\geq|\eta_m''|. $$
Therefore, the already chosen $\lambda_i$'s and $\nu_j$'s satisfy the necessary interlacing condition for $\Hom_H(F^G(\lambda),F^H(\nu))\neq\{0\}$. It remains to show that one can choose the remaining $\lambda_i$'s and $\nu_j$'s such that the interlacing condition holds and additionally $\xi_i'=\lambda_i+2\ZZ$ and $\eta_j'=\nu_j+2\ZZ$, which is an easy exercise.

\subsection{Generic multiplicities}

Using the intertwining operators between spherical principal series constructed in Section~\ref{sec:ConstructionOfSBOs} and the upper multiplicity bounds obtained in Section~\ref{sec:InvariantDistributions}, we prove the following generic multiplicity formula for general principal series representations of multiplicity one pairs:

\begin{theorem}\label{thm:MultiplicitiesMultOnePairs}
Assume that $(G,H)$ is one of the multiplicity one pairs in \eqref{eq:MultOnePairs}. Then for all $(\xi,\eta)\in\widehat{M}_G\times\widehat{M}_H$ and $(\lambda,\nu)\in\fraka_{G,\CC}^\vee\times\fraka_{H,\CC}^\vee$ we have the lower multiplicity bound
$$ \dim\Hom_H(\pi_{\xi,\lambda}|_H,\tau_{\eta,\nu}) \geq 1 \qquad \mbox{whenever }\Hom_M(\xi|_M,\eta|_M)\neq\{0\}, $$
and for $(\lambda,\nu)\in\fraka_{G,\CC}^\vee\times\fraka_{H,\CC}^\vee$ satisfying the generic condition~\eqref{eq:GenericCondition} we have
$$ \dim\Hom_H(\pi_{\xi,\lambda}|_H,\tau_{\eta,\nu}) = \begin{cases}1&\mbox{for $\Hom_M(\xi|_M,\eta|_M)\neq\{0\}$,}\\0&\mbox{for $\Hom_M(\xi|_M,\eta|_M)=\{0\}$.}\end{cases} $$
\end{theorem}

\begin{proof}
By Lemma~\ref{lem:OpenOrbitsForMultOnePairs} we have $\#(P_H\backslash G/P_G)_{\rm open}=1$. Hence, it follows from Theorem~\ref{thm:MultiplicityBounds} and Proposition~\ref{prop:ExGRepsForMultOnePairs}~\eqref{prop:ExGRepsForMultOnePairs1} that
$$ \dim\Hom_H(\pi_{\xi,\lambda}|_H,\tau_{\eta,\nu}) \leq \dim\Hom_M(\xi^{\tilde{w}_0}|_M,\eta|_M) = \dim\Hom_M(\xi|_M,\eta|_M) \leq 1 $$
for $(\lambda,\nu)\in\fraka_{G,\CC}^\vee\times\fraka_{H,\CC}^\vee$ satisfying \eqref{eq:GenericCondition}. It therefore suffices to show that for $\Hom_M(\xi|_M,\eta|_M)\neq\{0\}$ we have $\dim\Hom_H(\pi_{\xi,\lambda}|_H,\tau_{\eta,\nu})\geq1$ for all $(\lambda,\nu)\in\fraka_{G,\CC}^\vee\times\fraka_{H,\CC}^\vee$.\\
By Proposition~\ref{prop:ExGRepsForMultOnePairs}~(2) and Corollary~\ref{cor:TranslationPrincipleSupport} there exist $(\lambda_0,\nu_0)\in\fraka_G^\vee\times\fraka_H^\vee$ and a linear map
$$ \Phi:\Hom_H(\pi_\lambda|_H,\tau_\nu)\to\Hom_H(\pi_{\xi,\lambda+\lambda_0}|_H,\tau_{\eta,\nu+\nu_0}) $$
which is on the level of distribution kernels given by tensoring with a fixed non-trivial real-analytic section. We apply this map to the holomorphic family of intertwining operators obtained in Theorem~\ref{thm:MeromorphicContinuation}. More precisely, by Theorem~\ref{thm:MeromorphicContinuation} there exists a family $K_{\lambda,\nu}$ of distribution kernels of intertwining operators $A_{\lambda,\nu}\in\Hom_H(\pi_\lambda|_H,\tau_\nu)$, depending holomorphically on $(\lambda,\nu)\in\fraka_{G,\CC}^\vee\times\fraka_{H,\CC}^\vee$, such that generically $\supp K_{\lambda,\nu}=G/P_G$. Then the distribution kernels of $\Phi(A_{\lambda,\nu})$ depend holomorphically on $(\lambda,\nu)\in\fraka_{G,\CC}^\vee\times\fraka_{H,\CC}^\vee$ since they are given by tensoring the holomorphic family $K_{\lambda,\nu}$ with a fixed smooth section. Further, by Corollary~\ref{cor:TranslationPrincipleSupport} they are generically supported on $G/P_G$ and hence the holomorphic family $\Phi(A_{\lambda,\nu})\in\Hom_H(\pi_{\xi,\lambda+\lambda_0}|_H,\tau_{\eta,\nu+\nu_0})$ is non-trivial. Now the desired lower multiplicity bound follows from Lemma~\ref{lem:RenormHol}.
\end{proof}

Combining Theorem~\ref{thm:MultiplicitiesMultOnePairs} with the multiplicity one statement in Fact~\ref{fact:MultiplicityOne} we immediately obtain the following corollary:

\begin{corollary}\label{cor:IrreducibleMultiplicitiesMultOnePairs}
Let $(G,H)$ be one of the pairs in \eqref{eq:MultOnePairs} and assume that $\pi_{\xi,\lambda}$ and $\tau_{\eta,\nu}$ are irreducible. Then, if $\Hom_M(\xi|_M,\eta|_M)\neq\{0\}$, we have
$$ \dim\Hom_H(\pi_{\xi,\lambda}|_H,\tau_{\eta,\nu}) = 1. $$
\end{corollary}

\subsection{The Gross--Prasad conjecture for complex orthogonal groups}

In 1992 B. Gross and D. Prasad~\cite{GP92} formulated a conjecture about the multiplicities $\dim_H(\pi|_H,\tau)$ for the reductive pair $(G,H)=(\SO(n+1),\SO(n))$ over local and global fields. For the field $k=\CC$ the local conjecture takes the following form:

\begin{conjecture}[{\cite[Conjecture 11.5]{GP92}}]\label{conj:GrossPrasadComplex}
Let $(G,H)=(\SO(n+1,\CC),\SO(n,\CC))$ and assume that $\pi_{\xi,\lambda}$ and $\tau_{\eta,\nu}$ are irreducible; then $\dim\Hom_H(\pi_{\xi,\lambda}|_H,\tau_{\eta,\nu})=1$.
\end{conjecture}

Using our results from the previous section we can prove this conjecture. It follows from the following more general statement:

\begin{corollary}\label{cor:GPConjecture}
Let $(G,H)$ be one of the pairs in \eqref{eq:MultOnePairs} and assume that $p=q$ or $p=q+1$ in the case of indefinite unitary or orthogonal groups. Then, if the representations $\pi_{\xi,\lambda}$ and $\tau_{\eta,\nu}$ are irreducible, we have
$$ \dim\Hom_H(\pi_{\xi,\lambda}|_H,\tau_{\eta,\nu}) = 1. $$
\end{corollary}

\begin{proof}
In all cases we have $M=\{\1\}$ so that $\Hom_M(\xi|_M,\eta|_M)\neq\{0\}$ for all $\xi\in\widehat{M}_G$ and $\eta\in\widehat{M}_H$. Then the statement follows from Corollary~\ref{cor:IrreducibleMultiplicitiesMultOnePairs}.
\end{proof}

\section{An example: $(G,H)=(\GL(n+1,\RR),\GL(n,\RR))$}\label{sec:ExGLn}

For the multiplicity one pair $(G,H)=(\GL(n+1,\RR),\GL(n,\RR))$ we describe the meromorphic families of intertwining operators between principal series as constructed in Section~\ref{sec:MultiplicityOnePairs} explicitly. Over $p$-adic fields such operators were previously constructed by Murase--Sugano~\cite{MS96} (see also the recent work by Neretin~\cite{Ner15} in the context of finite-dimensional representations) and the formulas for the distribution kernels turn out to be formally the same in the real case.

\subsection{Distribution kernels}

For $1\leq p\leq n+1$ and $1\leq q\leq n$, we define the following polynomial functions on $M((n+1)\times(n+1),\RR)$:
$$ \Phi_p(x) = \det((x_{ij})_{1\leq i,j\leq p}), \qquad \Psi_q(x) = \det((x_{ij})_{2\leq i\leq q+1,1\leq j\leq q}). $$
With the representative
$$ \tilde{w}_0 = \left(\begin{array}{ccc}&&1\\&\reflectbox{$\ddots$}&\\1&&\end{array}\right) $$
for the longest Weyl group element $w_0\in W(\fraka_G)$, we then consider the functions $g\mapsto\Phi_k(\tilde{w}_0g),\Psi_k(\tilde{w}_0g)$ on $G$. For $d=\diag(d_1,\ldots,d_{n+1})\in M_GA_G$ and $n\in N_G$ we have
\begin{align*}
 \Phi_k(\tilde{w}_0gdn) &= d_1\cdots d_k\Phi_k(\tilde{w}_0g), & \Psi_k(\tilde{w}_0gdn) &= d_1\cdots d_k\Psi_k(\tilde{w}_0g),
\intertext{and for $d=\diag(d_1,\ldots,d_n,1)\in M_HA_H$ and $n\in N_H$, additionally,}
 \Phi_k(\tilde{w}_0dng) &= d_{n-k+2}\cdots d_n\Phi_k(\tilde{w}_0g), & \Psi_k(\tilde{w}_0dng) &= d_{n-k+1}\cdots d_n\Psi_k(\tilde{w}_0g).
\end{align*}

We identify $\widehat{M}_G\simeq(\ZZ/2\ZZ)^{n+1}$ by mapping $\xi=(\xi_1,\ldots,\xi_{n+1})\in(\ZZ/2\ZZ)^{n+1}$ to the character
$$ M_G\to\{\pm1\}, \quad \diag(x_1,\ldots,x_{n+1})\mapsto\sgn(x_1)^{\xi_1}\cdots\sgn(x_{n+1})^{\xi_{n+1}}. $$
Similarly $\widehat{M}_H\simeq(\ZZ/2\ZZ)^n$. Further, we identify $\fraka_{G,\CC}^\vee\simeq\CC^{n+1}$ by $\lambda\mapsto(\lambda(E_{1,1}),\ldots,\lambda(E_{n+1,n+1}))$ and similarly $\fraka_{H,\CC}^\vee\simeq\CC^n$. Then for $\xi=(\xi_1,\ldots,\xi_{n+1})\in(\ZZ/2\ZZ)^{n+1}\simeq\widehat{M}_G$, $\eta=(\eta_1,\ldots,\eta_n)\in(\ZZ/2\ZZ)^n\simeq\widehat{M}_H$ and $\lambda\in\CC^{n+1}\simeq\fraka_{G,\CC}^\vee$, $\nu\in\CC^n\simeq\fraka_{H,\CC}^\vee$ we put
$$ K_{(\xi,\lambda),(\eta,\nu)}(g) = \Phi_1(\tilde{w}_0g)^{s_1,\delta_1}\cdots\Phi_{n+1}(\tilde{w}_0g)^{s_{n+1},\delta_{n+1}}\Psi_1(\tilde{w}_0g)^{t_1,\varepsilon_1}\cdots\Psi_n(\tilde{w}_0g)^{t_n,\varepsilon_n}, \qquad g\in G, $$
where
$$ s_i = \lambda_i-\nu_{n-i+1}-\tfrac{1}{2} \quad (1\leq i\leq n), \quad s_{n+1} = \lambda_{n+1}+\tfrac{n}{2}, \quad t_j = \nu_{n-j+1}-\lambda_{j+1}-\tfrac{1}{2} \quad (1\leq j\leq n) $$
and
$$ \delta_i = \xi_i-\eta_{n-i+1} \quad (1\leq i\leq n), \quad \delta_{n+1} = \xi_{n+1}, \quad \varepsilon_j = \eta_{n-j+1}-\xi_{j+1} \quad (1\leq j\leq n). $$
Here we have used the notation
$$ x^{s,\varepsilon} = \sgn(x)^\varepsilon|x|^s, \qquad x\in\RR^\times,s\in\CC,\varepsilon\in\ZZ/2\ZZ. $$
Then $K_{(\xi,\lambda),(\eta,\nu)}$ satisfies
$$ K_{(\xi,\lambda),(\eta,\nu)}(m'a'n'gman) = \xi(m)a^{\lambda-\rho_{\frakn_G}}\cdot\eta(m')(a')^{\nu+\rho_{\frakn_H}}\cdot K_{(\xi,\lambda),(\eta,\nu)}(g) $$
for $g\in G$, $man\in P_G$, $m'a'n'\in P_H$. Hence, the functions $K_{(\xi,\lambda),(\eta,\nu)}$ define a meromorphic family of intertwining operators $A_{(\xi,\lambda),(\eta,\nu)}:\pi_{\xi,\lambda}|_H\to\tau_{\eta,\nu}$ by
$$ A_{(\xi,\lambda),(\eta,\nu)}f(h) = \int_K K_{(\xi,\lambda),(\eta,\nu)}(h^{-1}k)f(k)\,dk, \qquad h\in H. $$

\begin{remark}
It is easy to see that the functions $g\mapsto\Phi_k(g),\Psi_k(g)$ are matrix coefficients for the irreducible finite-dimensional representation of $\GL(n+1,\RR)$ on $\bigwedge^k\RR^{n+1}$.
\end{remark}

\appendix

\section{Finite-dimensional branching rules for strong Gelfand pairs}

We list the explicit branching rules for some small representations of the strong Gelfand pairs $(\sl(n+1,\CC),\gl(n,\CC))$ and $(\so(n+1,\CC),\so(n,\CC))$. The classical branching rules for these pairs can be found e.g. in \cite[Chapter 8]{GW09}.

\subsection{$(\sl(n+1,\CC),\gl(n,\CC))$}\label{app:BranchingSLn}

We label the Dynkin diagrams of $\sl(n+1,\CC)$ and $\sl(n,\CC)$ as usual:

$$ \sl(n+1,\CC):\begin{xy}
\ar@{-} (0,0) *+!D{\alpha_1} *{\circ}="A"; 
 (10,0)  *+!D{\alpha_2} *{\circ}="B"
\ar@{-} "B";  (20,0)
\ar@{.} (20,0); (30,0) 
\ar@{-} (30,0); (40,0) *+!D{\alpha_{n-1}}  *{\circ}="C"
\ar@{-} "C"; (50,0) *+!D{\alpha_n}  *{\circ}="D"
\end{xy} $$

$$ \sl(n,\CC):\begin{xy}
\ar@{-} (0,0) *+!D{\beta_1} *{\circ}="A"; 
 (10,0)  *+!D{\beta_2} *{\circ}="B"
\ar@{-} "B";  (20,0)
\ar@{.} (20,0); (30,0) 
\ar@{-} (30,0); (40,0) *+!D{\beta_{n-1}}  *{\circ}="C"
\end{xy} $$

Realize the root system of $\sl(n+1,\CC)$ as $\{\pm(e_i-e_j):1\leq i<j\leq n+1\}$ in the vector space $V=\{x\in\RR^{n+1}:x_1+\cdots+x_{n+1}=0\}$. To simplify notation, denote by $\pi(x)$ the orthogonal projection of $x\in\RR^{n+1}$ to $V$. We choose the simple roots $\alpha_i=e_i-e_{i+1}$ ($1\leq i\leq n$) for $\sl(n+1,\CC)$ and the simple roots $\beta_i=\alpha_i$ ($i=1,\ldots,n-1$) for $\sl(n,\CC)$. Denote by $\varpi_1,\ldots,\varpi_n$ the corresponding fundamental weights for $\sl(n+1,\CC)$ and by $\zeta_1,\ldots,\zeta_{n-1}$ the fundamental weights for $\sl(n,\CC)$. Further put $\zeta_n:=\varpi_n$, then $\zeta_n$ describes a character of $\frakz(\gl(n,\CC))\simeq\CC$.

Consider the fundamental weight $\varpi_i=\pi(e_1+\cdots+e_i)$. From the classical branching laws we know that $F^\frakg(\varpi_i)|_\frakh$ decomposes into the direct sum of the two $\frakh$-representations with highest weights $\pi(e_1+\cdots+e_i)$ and $\pi(e_1+\cdots+e_{i-1}+e_{n+1})$. Using
$$ \varpi_i = \zeta_i+\tfrac{i}{n}\zeta_n, \qquad 1\leq i\leq n-1, $$
it follows that
$$ F^{\sl(n+1,\CC)}(\varpi_i)|_{\gl(n,\CC)} \simeq \begin{cases}\Big(F^{\sl(n,\CC)}(\zeta_1)\boxtimes F^\CC(\tfrac{1}{n}\zeta_n)\Big)\oplus\Big(F^{\sl(n,\CC)}(0)\boxtimes F^\CC(-\zeta_n)\Big)\hspace{.6cm}\mbox{for $i=1$,}\\\Big(F^{\sl(n,\CC)}(\zeta_i)\boxtimes F^\CC(\tfrac{i}{n}\zeta_n)\Big)\oplus\Big(F^{\sl(n,\CC)}(\zeta_{i-1})\boxtimes F^\CC(-\tfrac{n-i+1}{n}\zeta_n)\Big)\\\hspace{8cm}\mbox{for $2\leq i\leq n-1$,}\\\Big(F^{\sl(n,\CC)}(0)\boxtimes F^\CC(\zeta_n)\Big)\oplus\Big(F^{\sl(n,\CC)}(\zeta_{n-1})\boxtimes F^\CC(-\tfrac{1}{n}\zeta_n)\Big)\\\hspace{9.4cm}\mbox{for $i=n$.}\end{cases} $$

Now consider the fundamental weight $\varpi_i+\varpi_{n-i+1}=(e_1+\cdots+e_i)-(e_{n-i+2}+\cdots+e_{n+1})$ ($1\leq i\leq\frac{n}{2}$). From the classical branching laws we know that $F^\frakg(\varpi_i+\varpi_{n-i+1})|_\frakh$ decomposes into the direct sum of the four $\frakh$-representations with highest weights
\begin{align*}
 & (e_1+\cdots+e_i)-(e_{n-i+1}+\cdots+e_n), && (e_1+\cdots+e_i)-(e_{n-i+2}+\cdots+e_{n+1}),\\
 & (e_1+\cdots+e_{i-1})-(e_{n-i+1}+\cdots+e_n)+e_{n+1}, && (e_1+\cdots+e_{i-1})-(e_{n-i+2}+\cdots+e_n),
\end{align*}
so that
\begin{align*}
 & F^{\sl(n+1,\CC)}(\varpi_i+\varpi_{n-i+1})|_{\gl(n,\CC)}\\
 & \qquad\simeq \Big(F^{\sl(n,\CC)}(\zeta_i+\zeta_{n-i})\boxtimes F^\CC(0)\Big)\oplus\Big(F^{\sl(n,\CC)}(\zeta_i+\zeta_{n-i+1})\boxtimes F^\CC(\tfrac{n+1}{n}\zeta_n)\Big)\\
 & \qquad\qquad\oplus\Big(F^{\sl(n,\CC)}(\zeta_{i-1}+\zeta_{n-i})\boxtimes F^\CC(-\tfrac{n+1}{n}\zeta_n)\Big)\oplus\Big(F^{\sl(n,\CC)}(\zeta_{i-1}+\zeta_{n-i+1})\boxtimes F^\CC(0)\Big).
\end{align*}
Note that for $i=\frac{n}{2}$ with $n=2m$ even, the formula still holds and we have $\zeta_i+\zeta_{n-i}=2\zeta_m$. Similarly one obtains for $i=\frac{n+1}{2}$ with $n=2m-1$ odd,
\begin{align*}
 F^{\sl(n+1,\CC)}(2\varpi_m)|_{\gl(n,\CC)} \simeq{}& \Big(F^{\sl(n,\CC)}(2\zeta_m)\boxtimes F^\CC(\tfrac{n+1}{n}\zeta_n)\Big)\oplus\Big(F^{\sl(n,\CC)}(2\zeta_{m-1})\boxtimes F^\CC(-\tfrac{n+1}{2}\zeta_n)\Big)\\
 & \qquad\qquad\qquad\oplus\Big(F^{\sl(n,\CC)}(\zeta_{m-1}+\zeta_m)\boxtimes F^\CC(0)\Big).
\end{align*}

\subsection{$(\so(n+1,\CC),\so(n,\CC))$}\label{app:BranchingSOn}

We label the Dynkin diagrams of $\so(n+1,\CC)$ and $\so(n,\CC)$ as usual:

$$ \begin{tabular}{rcc}$\so(n+1,\CC):$&
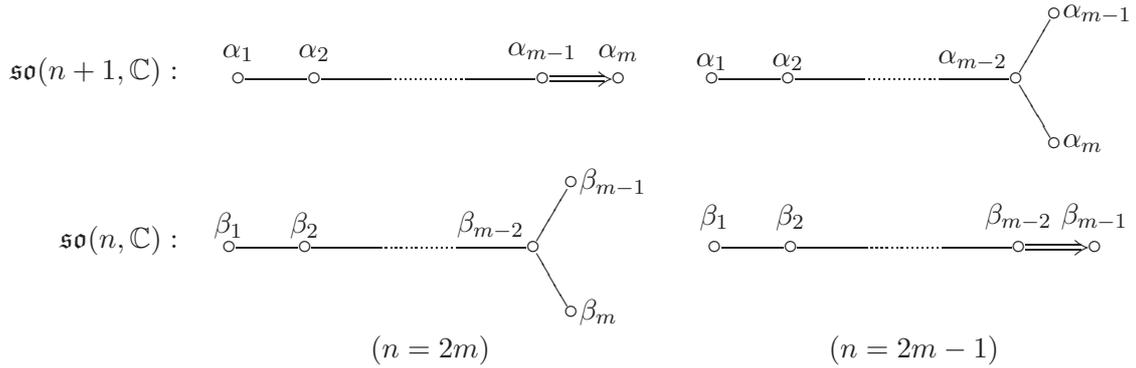
\begin{xy}
\ar@{-} (0,0) *++!D{\alpha_1} *{\circ}="A";
  (10,0) *++!D{\alpha_2}  *{\circ}="B"
\ar@{-} "B"; (20,0)
\ar@{.} (20,0); (30,0) 
\ar@{-} (30,0); (40,0) *++!D{\alpha_{m-1}}  *{\circ}="C"
\ar@{=>} "C"; (50,0) *++!D{\alpha_m}  *{\circ}="D"
\end{xy}&\begin{xy}
\ar@{-} (0,0) *+!D{\alpha_1} *{\circ}="A"; 
 (10,0)  *+!D{\alpha_2} *{\circ}="B"
\ar@{-} "B";  (20,0)
\ar@{.} (20,0); (30,0) 
\ar@{-} (30,0); (40,0) *+!DR{\alpha_{m-2}} *{\circ}="F"
\ar@{-} "F"; (45,8.6)  *+!L{\alpha_{m-1}} *{\circ}
\ar@{-} "F"; (45,-8.6)  *+!L{\alpha_m} *{\circ}
\end{xy}\\
$\so(n,\CC):$&\begin{xy}
\ar@{-} (0,0) *+!D{\beta_1} *{\circ}="A"; 
 (10,0)  *+!D{\beta_2} *{\circ}="B"
\ar@{-} "B";  (20,0)
\ar@{.} (20,0); (30,0) 
\ar@{-} (30,0); (40,0) *+!DR{\beta_{m-2}} *{\circ}="F"
\ar@{-} "F"; (45,8.6)  *+!L{\beta_{m-1}} *{\circ}
\ar@{-} "F"; (45,-8.6)  *+!L{\beta_m} *{\circ}
\end{xy}&\begin{xy}
\ar@{-} (0,0) *++!D{\beta_1} *{\circ}="A";
  (10,0) *++!D{\beta_2}  *{\circ}="B"
\ar@{-} "B"; (20,0)
\ar@{.} (20,0); (30,0) 
\ar@{-} (30,0); (40,0) *++!D{\beta_{m-2}}  *{\circ}="C"
\ar@{=>} "C"; (50,0) *++!D{\beta_{m-1}}  *{\circ}="D"
\end{xy}\\
&$(n=2m)$&$(n=2m-1)$\end{tabular} $$

From the classical branching rules it follows that, for $n=2m$ even,
\begin{align*}
 F^{\so(n+1,\CC)}(\varpi_i)|_{\so(n,\CC)} \simeq{}& F^{\so(n,\CC)}(\zeta_i)\oplus F^{\so(n,\CC)}(\zeta_{i-1}) \qquad\qquad\qquad (1\leq i\leq m-2),\\
 F^{\so(n+1,\CC)}(\varpi_{m-1})|_{\so(n,\CC)} \simeq{}& F^{\so(n,\CC)}(\zeta_{m-1}+\zeta_m)\oplus F^{\so(n,\CC)}(\zeta_{m-2}),\\
 F^{\so(n+1,\CC)}(\varpi_m)|_{\so(n,\CC)} \simeq{}& F^{\so(n,\CC)}(\zeta_m)\oplus F^{\so(n,\CC)}(\zeta_{m-1}),\\
 F^{\so(n+1,\CC)}(2\varpi_m)|_{\so(n,\CC)} \simeq{}& F^{\so(n,\CC)}(2\zeta_m)\oplus F^{\so(n,\CC)}(2\zeta_{m-1})\oplus F^{\so(n,\CC)}(\zeta_{m-1}+\zeta_m),
\end{align*}
and for $n=2m-1$ odd,
\begin{align*}
 F^{\so(n+1,\CC)}(\varpi_i)|_{\so(n,\CC)} &\simeq F^{\so(n,\CC)}(\zeta_i)\oplus F^{\so(n,\CC)}(\zeta_{i-1}) \qquad\qquad\qquad (1\leq i\leq m-2),\\
 F^{\so(n+1,\CC)}(\varpi_{m-1})|_{\so(n,\CC)} &\simeq F^{\so(n+1,\CC)}(\varpi_m)|_{\so(n,\CC)} \simeq F^{\so(n,\CC)}(\zeta_{m-1}),\\
 F^{\so(n+1,\CC)}(\varpi_{m-1}+\varpi_m)|_{\so(n,\CC)} &\simeq F^{\so(n,\CC)}(2\zeta_{m-1})\oplus F^{\so(n,\CC)}(\zeta_{m-2}).
\end{align*}

\providecommand{\bysame}{\leavevmode\hbox to3em{\hrulefill}\thinspace}
\providecommand{\href}[2]{#2}


\end{document}